\documentclass[11pt]{article}
\usepackage{amsthm,bm,amsmath,enumerate,bbm,nicefrac,hyperref,mathtools,dsfont,xcolor,authblk,eufrak}
\usepackage[a4paper,margin={2.8cm, 2.8cm}]{geometry}
\usepackage{comment}

\usepackage{listings}   
\usepackage{xcolor}
\definecolor{codegreen}{rgb}{0,0.6,0}
\definecolor{codegray}{rgb}{0.5,0.5,0.5}
\definecolor{codepurple}{rgb}{0.58,0,0.82}
\definecolor{backcolour}{rgb}{0.95,0.95,0.92}
\lstdefinestyle{mystyle}{
    backgroundcolor=\color{backcolour},   
    commentstyle=\color{codegreen},
    keywordstyle=\color{magenta},
    numberstyle=\tiny\color{codegray},
    stringstyle=\color{codepurple},
    basicstyle=\ttfamily\footnotesize,
    breakatwhitespace=false,         
    breaklines=true,                 
    captionpos=b,                    
    keepspaces=true,                 
    numbers=left,                    
    numbersep=5pt,                  
    showspaces=false,                
    showstringspaces=false,
    showtabs=false,                  
    tabsize=2
}
\def\C{{\mathbbm C}}
\def\E{{\mathbbm E}}
\def\H{{\mathbbm H}}

\def\N{{\mathbbm N}}
\def\P{{\mathbbm P}}
\def\R{{\mathbbm R}}
\def\U{{\mathbbm U}}
\def\Y{{\mathbbm Y}}

\newcommand{\calB}{\mathcal B}

\newcommand{\calF}{\mathcal F}

\newcommand{\calL}{\mathcal L}

\newcommand{\calS}{\mathcal S}

\newcommand{\calW}{\mathcal W}

\newcommand{\bfLambda}{\mathbf{\Lambda}}
\newcommand{\bfA}{\mathbf{A}}
\newcommand{\bfB}{\mathbf{B}}

\newcommand{\bfH}{\mathbf{H}}

\newcommand{\bfP}{\mathbf{P}}

\newcommand{\bfS}{\mathbf{S}}

\newcommand{\eps}{\varepsilon}
\renewcommand{\phi}{\varphi}

\newcommand{\one}{\mathbbm{1}}

\DeclarePairedDelimiter{\norm}{\lVert}{\rVert}
\DeclarePairedDelimiter{\abs}{\lvert}{\rvert}

\newcommand{\twovector}[2]{\begin{pmatrix} #1 \\ #2 \end{pmatrix}} 

\newcommand{\round}[2]{\lfloor #1 \rfloor_{#2}} %

\DeclareMathOperator{\id}{id}

\DeclareMathOperator{\Lip}{Lip}

\theoremstyle{plain}
\newtheorem{theorem}{Theorem}[section]
\theoremstyle{remark}

\theoremstyle{plain}
\newtheorem{corollary}[theorem]{Corollary}
\newtheorem{lemma}[theorem]{Lemma}
\newtheorem{proposition}[theorem]{Proposition}

\newtheorem{setting}[theorem]{Setting}

\newcommand{\bchange}{\begin{color}{red}}
\newcommand{\echange}{\end{color}}

\newcommand{\wen}{\color{black}}

\allowdisplaybreaks

\title{Weak convergence rates for temporal\linebreak numerical approximations of stochastic wave\linebreak equations with multiplicative noise}

\author{Sonja Cox$^{1}$, Arnulf Jentzen$^{2,3,4} $, and Felix Lindner$^{5} $
	\bigskip 
	\\
	\small{$^1$ Korteweg-de Vries Institute for Mathematics, Faculty of Science,}
	\\
	\small{University of Amsterdam,  The Netherlands; e-mail: s.g.cox@uva.nl}
	\medskip
	\\
	\small{$^2$ School of Data Science and Shenzhen Research Institute of Big Data,}
	\\
	\small{The Chinese University of Hong Kong, Shenzhen (CUHK-Shenzhen), China;} \\
	\small{e-mail: ajentzen@cuhk.edu.cn} 
	\medskip
	\\
	\small{$^3$ Applied Mathematics: Institute for Analysis and Numerics,}
	\\
	\small{University of Münster, Germany; e-mail: ajentzen@uni-muenster.de}
	\medskip
	\\
	\small{$^4$ Seminar for Applied Mathematics, Department of Mathematics,}
	\\
	\small{ETH Z{\"u}rich, Switzerland; e-mail: arnulf.jentzen@sam.math.ethz.ch} 
	\medskip
	\\
	\small{$^5$  Institute of Mathematics, Faculty of Mathematics and Natural Sciences,}
	\\
	\small{University of Kassel, Germany; e-mail: lindner@mathematik.uni-kassel.de}
}
\date{}                     

\begin{document}

\allowdisplaybreaks
\maketitle 
\begin{abstract}
In this work we establish weak convergence rates for 
temporal discretisations of stochastic wave equations with multiplicative noise, 
in particular, for the hyperbolic Anderson model. 
For this class of stochastic partial differential equations the weak convergence rates we obtain are indeed twice the known strong rates. 
To the best of our knowledge, our findings are the first in the scientific literature which provide essentially sharp weak convergence rates for 
temporal discretisations of stochastic wave equations  
with multiplicative noise. 
Key ideas of our proof are a sophisticated splitting of the error and applications of the recently introduced mild It\^{o} formula. We complement our analytical findings by means of numerical simulations in Python for the decay of the weak approximation error for SPDEs for four different test functions.
\end{abstract}

\renewcommand{\thefootnote}{\fnsymbol{footnote}} 
\renewcommand{\thefootnote}{\arabic{footnote}} 
\renewcommand{\thefootnote}{\fnsymbol{footnote}} 
\renewcommand{\thefootnote}{\arabic{footnote}} 

\setcounter{tocdepth}{3}
\tableofcontents
\section{Introduction}\label{sec:intro}
Stochastic partial differential equations (SPDEs) are used to model various evolutionary processes subject to random forces.
For example, stochastic wave equations may model the motion of a strand of DNA in a liquid or heat flow around a ring; see, e.g.,~\cite{Dalang:2009, Thomas:2012}. In general the solution to an SPDE cannot be given explicitly, whence it is desirable to prove convergence rates for numerical approximations. 
Here one distinguishes strong convergence rates, i.e., 
with respect to the strong (mean square) error, and weak convergence rates, i.e., 
with respect to 
a suitable weak approximation error. \label{remark:weak_error} 
Typically, the convergence rate for the weak error is twice the convergence rate for the strong error.
However, there does not exist a straightforward way to establish this. Moreover, non-trivial exceptions to this rule exist; see, e.g., \cite{Alfonsi:2005, HefterJentzen:2017}.

For both parabolic and hyperbolic semilinear SPDEs 
with coefficients depending on the state in a globally Lipschitz continuous way, 
 strong convergence is by now well-understood. In particular, strong convergence rates for numerical approximations of 
stochastic wave equations 
have been established in, e.g.,~\cite{AntonEtAl:2016, CohenEtAl:2013, CohenQuerSardanyons:2016, CaoYin:2017, KovacsEtAl:2013,KovacsEtAl:2010, QuerSans:2006, Walsh:2006, Wang:2015, WangGanTang:2014}.  
Recently, a strong approximation scheme for a stochastic wave equation with a cubic nonlinearity has been analysed in 
\cite{CuiEtAl:2019}.

Establishing optimal weak convergence rates for both hyperbolic and parabolic SPDEs is currently active field of research; see, e.g.,~\cite{AnderssonEtAl:2016b, AnderssonEtAl:2015, AnderssonEtAl:2016a, AnderssonEtAl:2016d, AnderssonEtAl:2016c, AnderssonLarsson:2016, AnderssonLindner:2018, AnderssonLindner:2017,  Brehier:2012, Brehier:2014, Brehier:2017, BrehierEtAl:2017, BrehierGoudenege:2018, BrehierKopec:2017, ConusEtAl:2014, CuiHong:2017, CuiHong:2018, deBouardDebussche:2006, Debussche:2011, DebusschePrintems:2009, GeissertEtAl:2009, HarmsMueller:2017, Hausenblas:2003,  Hausenblas:2010, HefterEtAl:2016, JacobedeNauroisJentzenWelti:2021, JentzenKurniawan:2015, Kopec:2014, KovacsEtAl:2013, KovacsEtAl:2012, KovacsEtAl:2015, KovacsPrintems:2014, Kruse:2014, Lindgren:2012, LindnerSchilling:2013, Shardlow:2003, Wang:2015, Wang:2016, WangGan:2013}. Arguably, the most relevant basic SPDEs are the parabolic and hyperbolic Anderson model, i.e., the heat equation with multiplicative noise and the wave equation with multiplicative noise. 
However, establishing optimal weak convergence rates for SPDEs with multiplicative noise is challenging. Indeed, of the articles cited above only~\cite{BrehierEtAl:2017, ConusEtAl:2014, CuiHong:2017, deBouardDebussche:2006, Debussche:2011, HefterEtAl:2016, JacobedeNauroisJentzenWelti:2021, JentzenKurniawan:2015} provide weak rates for SPDEs with multiplicative noise. 
Roughly speaking, there are two successful approaches to obtain optimal weak convergence rates for parabolic SPDEs with multiplicative noise. One is based on regularity results for the corresponding Kolmogorov equation and Malliavin calculus; see, e.g.,~\cite{BrehierEtAl:2017, Debussche:2011}. The other is based on more elementary regularity results of the Kolmogorov equation and the mild It\^o formula; see, e.g., ~\cite{ConusEtAl:2014, HefterEtAl:2016, JentzenKurniawan:2015}. 

No successful approach for proving optimal weak convergence rates has been developed yet for temporal discretisations of hyperbolic SPDEs with multiplicative noise. Indeed, the two approaches mentioned above are not applicable as they rely strongly on the smoothing effect of the semigroup. In this work we tackle this problem and develop a technique that allows one to establish optimal weak convergence rates for hyperbolic SPDEs with multiplicative noise. A special case of our main result is presented in the following theorem. 
\begin{theorem}\label{thm:hyperbolic_anderson}
Let $ T, \vartheta \in ( 0 , \infty ) $, 
$ b_0,\, b_1 \in \R$, $H=L^2((0,1);\R)$,
let $ ( \Omega , \calF , \P, (\mathbbm{F}_t )_{ t \in [ 0 , T ] } ) $ be a filtered probability space which fulfills the usual conditions, 
let $ ( W_t )_{ t \in [ 0 , T ] } $ be an $\id_{H} $-cylindrical $ ( \mathbbm{F}_t )_{ t \in [ 0 , T ] } $-Wiener process, 
let $A\colon D(A) \subseteq H \rightarrow H$
be the Dirichlet Laplacian on $H$, 
let $ ( H_{r} , \langle \cdot, \cdot \rangle_{H_r}, \left\| \cdot \right\|_{H_{r}})$, $ r \in \R $, 
be a family of interpolation spaces\footnote{
Note that for every 
$r\in[0,\infty)$, 
$v\in H_r$ it holds that $H_r=D((-A)^r)$ and $\|v\|_{H_r}=\|(-A)^rv\|_H$.
}
associated to $- A $, 
let $\bfH_0 = H_{0} \times H_{-\nicefrac{1}{2}}$, $\bfH_1 = H_{\nicefrac{1}{2}} \times H_{0} $, let 
$ \bfA \colon D(\bfA)\subseteq\bfH_0 \to \bfH_0 $ be the linear operator which satisfies that $D(\bfA)=\bfH_1$ and $[\forall\,(v,w) \in D(\bfA)\colon\bfA( v , w ) = ( w , \vartheta A v )]$, 
let  $ \xi \in \calL^6 ( \P|_{\mathbbm{F}_0} ; \bfH_{1}  ) $, 
$ \phi \in C^4( \bfH_0, \R ) $ satisfy\footnote{
Observe that for every $k\in\N$ and every $k$-linear bounded operator 
$C\colon(\bfH_0)^k\to\R$ 
it holds that $\|C\|_{L^{(k)}(\bfH_0,\R)}=
\sup\{|C(x_1,x_2,\ldots,x_k)|\colon x_1,x_2,\ldots,x_k\in\bfH_0
,\,
\|x_1\|_{\bfH_0}=\|x_2\|_{\bfH_0}=\ldots=\|x_k\|_{\bfH_0}=1\}.$
}
 that $\sup_{k\in \{1,2,3,4\},\,x\in \bfH_0}\| \phi^{(k)}(x) \|_{L^{(k)}(\bfH_0,\R)} < \infty$, 
let $\bfB \colon \bfH_0 \to L_2( H, \bfH_0 ) $ 
be the function which satisfies for every 
$ (v,w) \in \bfH_0 $, $u\in H_{\nicefrac12}$ that 
$ \bfB( v, w ) u  = \bigl( 0, (b_0  + b_1 v) u \bigr)$, 
let $ X\colon [0,T] \times \Omega \rightarrow  \bfH_0$
be an $(\mathbbm{F}_t)_{t\in[0,T]}$-predictable stochastic process which satisfies 
for every $t\in [0,T]$ that 
$\sup_{s\in[0,T]}\E\big[\| X_s \|_{\bfH_0}^2\big] < \infty$  
and 
\begin{align}
 X_t & 
 =
 e^{t\bfA } \xi
 + \int_{0}^{t} e^{ (t-s)\bfA} \bfB( X_s )\,\mathrm{d}W_s,
\end{align}
and let 
$Y^{N}\colon \{0,1,2,\ldots,N\} \times \Omega \rightarrow \bfH_0$, $N\in \N$,
be the stochastic processes which satisfy for all $N\in \N$, $n\in \{1,2,\ldots,N\}$ 
that $Y^{N}_0 = \xi$ and
\begin{equation}
  Y^{N}_n 
  = 
  e^{(\nicefrac{T}{N})\bfA} 
  \bigg( 
    Y^{N}_{n-1} 
    + 
	\int_{\nicefrac{(n-1)T}{N}}^{\nicefrac{nT}{N}} \bfB(Y^N_{n-1})\,\mathrm{d}W_s
    \bigg).
\end{equation}
Then it holds for all $\eps \in (0,\infty)$ that
$
  \sup_{N\in \N} 
  \left(
    N^{1-\eps}
    \left|
      \E\! \left[ \phi\big(X_T\big) \right] 
      - 
      \E\! \left[ \phi\big(Y_N^N\big)\right] 
    \right|
  \right)
   < \infty.
$
\end{theorem}
Note that we obtain rate of convergence $1^{-}$, which is indeed twice the known strong rate. Theorem~\ref{thm:hyperbolic_anderson} is an immediate consequence of Corollary~\ref{cor:hyperbolic_anderson} below. 
Corollary~\ref{cor:hyperbolic_anderson} follows from Theorem~\ref{prop:fin_dim_weak_error}, which is the main result of this article. 
Indeed, Theorem~\ref{prop:fin_dim_weak_error} establishes an upper bound for the weak error of a temporal discretisation of a hyperbolic SPDE 
with multiplicative noise. 
Similar as in the parabolic case, a key ingredient of the proof of this upper bound is the mild It\^o formula
developed in~\cite{DaPratoJentzenRoeckner:2012}. 
For the convenience of the reader, we recall a suitable variant of the mild It\^o formula in Proposition~\ref{thm:mildIto} below. \label{remark:mildIto} 
For
parabolic SPDEs the mild It\^o formula is used to insert the semigroup in an appropriate place so the smoothing property can be exploited. Here, however, the mild It\^o formula is used to rewrite certain terms in the error as integrals over an interval of length at most $\nicefrac{T}{N}$. 
The use of the mild It\^o formula is crucial: if one would apply the `classical' It\^o formula, then one would obtain a term involving an unbounded operator.  
Although the underlying semigroup does not does not enjoy a smoothing property as in the parabolic case, by using the mild It\^{o} formula one can avoid the appearance of an unbounded operator and thus the roughing effect accompanied by it. 
Another key ingredient of the proof is an elegant decomposition of the error into terms that can be treated using this mild It\^o formula approach, and terms that can be dealt with in a relatively straightforward manner; see~\eqref{eq:fin_dim_weak_error}--\eqref{eq:split_noise_term} in the proof of Theorem~\ref{prop:fin_dim_weak_error}. 
It is to be expected that this method of proof can also be applied to other types of temporal discretisations, as well as to spatial discretisations such as the finite element method. Moreover,
although we consider the Hilbert space setting in this work, our approach can be extended in a straightforward way to the Banach space setting; see~\cite{CoxEtAl:2016}. This would allow one to prove optimal weak rates for more general semilinear drift and diffusion coefficients; see~\cite{HefterEtAl:2016} for analogous results for parabolic SPDEs. 
For completeness we note that optimal weak convergence rates for spatial spectral Galerkin approximations of 
stochastic wave equations  
have been established in~\cite{NauroisEtAl:2017, JacobedeNauroisJentzenWelti:2021}. 
The approach taken in~\cite{NauroisEtAl:2017, JacobedeNauroisJentzenWelti:2021} essentially relies on the specific structure of the spatial spectral Galerkin approximations and can thus neither be extended to temporal approximations nor to other more complicated spatial approximations such as the finite element method. 
\par 
Let us comment on the optimality of the convergence rate obtained in Theorem~\ref{thm:hyperbolic_anderson} above. 
Lower bounds for strong and weak approximation errors of numerical discretisations of SPDEs have been derived in, e.g., \cite{BeckerGessJentzenKloeden:2018,ConusEtAl:2014,DavieGaines:2001,NauroisEtAl:2017,JentzenKurniawan:2015,mr:07a,MuellerGronbachRitterWagner:2008b,mrw:08}. 
In particular, lower bounds for weak approximation errors of spatial spectral Galerkin approximations of stochastic wave equations can be found in \cite{NauroisEtAl:2017}. 
Lower bounds for strong and weak approximation errors of temporal numerical approximations of stochastic wave equations remain an open problem for future research. Nevertheless, we conjecture that the weak convergence rate for the exponential Euler scheme in Theorem~\ref{thm:hyperbolic_anderson} above can in general not essentially be improved.
\par
The remainder of this article is structured as follows. 
In Section~\ref{sec:Kolmogorov_Hilbert} we present some essentially well-known auxiliary results. 
More specifically, a standard existence and uniqueness result for semilinear SDEs in Hilbert spaces is recalled in Subsection~\ref{ssec:SDE}, 
regularity results for the associated Kolmogorov equations are provided in Subsection~\ref{ssec:Kolmogorov}, 
and further preparatory lemmas are collected in Subsection~\ref{ssec:prep}. 
Section~\ref{sec:weak_convergence_rates} is devoted to the weak error analysis for temporal discretisations of a class of stochastic wave equations with multiplicative noise and contains our main abstract results. 
A general setting for our convergence analysis is presented in Subsection~\ref{ssec:setting_wave}, 
and some elementary properties of the wave semigroup are collected in Subsection~\ref{ssec:wavesg}. 
Theorem~\ref{prop:fin_dim_weak_error} in Subsection~\ref{ssec:upperboundsweak} establishes upper bounds for the weak errors of temporal discretisations of spatial spectral Galerkin approximations. 
This combined with the uniform moment bounds obtained in Subsection~\ref{ssec:momentbounds} and
the strong convergence of the Galerkin approximations proven in Subsection~\ref{ssec:upperboundsstrong} establishes the weak convergence rates of the temporal discretisations, see Corollary~\ref{cor:weak_conv} below. 
In Section~\ref{sec:hyperbolic_anderson} we apply the weak convergence result 
from  
Corollary~\ref{cor:weak_conv} to the hyperbolic Anderson model. 
After specifying a suitable setting in Subsection~\ref{ssec:setting_hyperbolic_anderson}, 
we collect some results on multipliers on Sobolev-Slobodeckij spaces in Subsection~\ref{ssec:prepHA}, 
which we use in Subsection~\ref{ssec:HA} to verify that Corollary~\ref{cor:weak_conv} implies Corollary~\ref{cor:hyperbolic_anderson}. 
Recall that Corollary~\ref{cor:hyperbolic_anderson} implies Theorem~\ref{thm:hyperbolic_anderson}. 
Numerical simulations illustrating Corollary~\ref{cor:hyperbolic_anderson} are presented in Subsection~\ref{ssec:simulations}. 
\subsection{Setting}\label{ssec:setting}
\noindent{}Throughout this article we shall frequently use the following setting.
\begin{setting}\label{setting:Part1}
For every pair of $\R$-Hilbert spaces $ (\mathcal V , \langle \cdot , \cdot \rangle_{\mathcal V } , \left\| \cdot\right\|_{\mathcal V} ) $ and
$ ( \mathcal W , \langle \cdot , \cdot \rangle_{ \mathcal W } , \left\| \cdot\right\|_{ \mathcal W } ) $ 
let $(L_2(\mathcal V,\mathcal W), \langle \cdot,\cdot\rangle_{L_2(\mathcal V,\mathcal W)},\left\| \cdot \right\|_{L_2(\mathcal V,\mathcal W)} )$ be the $\R$-Hilbert space 
of Hilbert-Schmidt operators from $\mathcal V$ to $\mathcal W$, 
for every $k\in \N$ and every pair of $ \R $-Banach spaces $ ( \mathcal V, \left\| \cdot \right\|_{\mathcal V} ) $ and $ ( \mathcal W, \left\| \cdot \right\|_{\mathcal W} ) $ 
let $(\Lip ( \mathcal V , \mathcal W ), \left\| \cdot\right\|_{ \Lip ( \mathcal V , \mathcal W ) })$
be the $\R$-Banach space of Lipschitz continuous mappings from $\mathcal V$ to $\mathcal W$, 
let  $(C_{\mathrm{b}}^{k}( \mathcal V,\mathcal W ), \|\cdot\|_{C_{\mathrm{b}}^{k}( \mathcal V,\mathcal W )})$  be the $\R$-Banach space of $k$-times continuously 
Fr\'echet differentiable functions 
from $\mathcal V$ to $\mathcal W$ with globally bounded derivatives, 
and let $(L^{(k)}(\mathcal V,\mathcal W), \linebreak \left\|\cdot\right\|_{ L^{(k)}(\mathcal V,\mathcal W)})$ be the $\R$-Banach space of $k$-linear bounded operators from $\mathcal V^k$ to $\mathcal W$, 
for every measure space $(\Omega,\calF,\mu)$, 
every measurable space $(S,\Sigma)$,
and every function $f\colon \Omega \rightarrow S$ let $[f]_{\mu,\Sigma}$ be the set given by
\begin{equation}
\begin{aligned}
[ f ]_{\mu,\Sigma} 
 =
 \Big\{ g \colon \Omega \rightarrow S \colon 
 \left[
 \substack{ 
 [\exists\, A \in \calF \colon ( \mu(A) = 0 
 \text{ and }
 \{ \omega \in \Omega \colon f(\omega) \neq g(\omega)\} 
 \subseteq A )]\\
 \text{ and }[\forall\, A\in \Sigma\colon g^{-1}(A)\in \calF]
 }
 \right]
 \Big\},
\end{aligned}
\end{equation}
let  $ ( U , \langle \cdot , \cdot \rangle_{ U } , \left\| \cdot\right\|_{ U } ) $ be a separable $ \R $-Hilbert space, 
let $ \U \subseteq U $ be an orthonormal basis of $ U $, let $ T \in ( 0 , \infty ) $, let $ ( \Omega , \calF , \P, (\mathbbm{F}_t)_{t\in [0,T]} ) $ be a filtered probability space which fulfills the usual conditions, and let $ ( W_t )_{ t \in [ 0 , T ] } $ be an 
$\id_U$-cylindrical $ ( \mathbbm{F}_t )_{ t \in [ 0 , T ] } $-Wiener process.
\end{setting}

\noindent In Setting~\ref{setting:Part1} we introduced for every pair of $ \R $-Banach spaces $ ( \mathcal V, \left\| \cdot \right\|_{\mathcal V} ) $ and $ ( \mathcal W, \left\| \cdot \right\|_{\mathcal W} )$ 
the tuples $(\Lip( \mathcal V , \mathcal W ), \left\| \cdot\right\|_{ \Lip ( \mathcal V , \mathcal W ) })$ and 
$(C_{\mathrm{b}}^{k}( \mathcal V,\mathcal W ), \left\|\cdot\right\|_{C_{\mathrm{b}}^{k}( \mathcal V,\mathcal W )})$, $k\in\N$. 
Note that for every pair of $ \R $-Banach spaces $ ( \mathcal V, \left\| \cdot \right\|_{\mathcal V} ) $ and $ ( \mathcal W, \left\| \cdot \right\|_{\mathcal W} )$ 
and every $k\in\N$, $F\in\Lip( \mathcal V , \mathcal W )$, $\varphi\in C_{\mathrm{b}}^{k}( \mathcal V,\mathcal W )$ it holds that 
$\|F\|_{\Lip( \mathcal V , \mathcal W )}=\|F(0)\|_{\mathcal W}+
\sup(\{\nicefrac{\|F(x)-F(y)\|_{\mathcal W}}{\|x-y\|_{\mathcal V}}\colon x,y\in\mathcal V,\,x\neq y\}\cup\{0\})$ 
and $\left\| \phi \right\|_{C_{\mathrm{b}}^{k}( \mathcal V,\mathcal W )} = 
\|\varphi(0)\|_{\mathcal W} + \sum_{j=1}^{k} \sup_{x\in\mathcal V}\| \phi^{(j)}(x) \|_{L^{(j)}(\mathcal V,\mathcal W)}$.
\section{Preliminaries}\label{sec:Kolmogorov_Hilbert}
%
%
%
%
\subsection[SDEs in Hilbert spaces]{Stochastic differential equations in Hilbert spaces}\label{ssec:SDE}
The existence and uniqueness result in Theorem~\ref{thm:existence} below is essentially well-known in the literature; 
cf., for example, Da~Prato \& Zabczyk~\cite[Theorem~7.4]{DaPratoZabczyk:1992}. 
%
\begin{theorem} \label{thm:existence}
Assume Setting~\ref{setting:Part1}, let $( H , \langle \cdot, \cdot \rangle_{ H } , \left\| \cdot\right\|_{ H } ) $ be a separable $ \R $-Hilbert space,  
let $ S \colon [ 0, \infty )\to L(H) $ be a strongly continuous semigroup, and 
let 
$ p \in [2, \infty) $, $ F \in \Lip( H, H ) $, $ B \in \Lip( H, L_2( U, H ) ) $, $ \xi \in \calL^p( \P\vert_{\mathbbm{F}_0}; H ) $. 
Then there exists an up to modifications unique $ ( \mathbbm{F}_t )_{t \in [0,T]} $-predictable stochastic process $ X \colon [0,T] \times \Omega \to H $ which satisfies for every $t\in [0,T]$
that $ \sup_{ s \in [0,T] } \E\!\left[ \norm{ X_s }_{H }^p \right] < \infty $ and
\begin{equation}
[X_t]_{\P,\,\calB(H)} 
= \left[ S_t  \xi \right]_{\P,\,\calB(H)} 
+  \int_0^t S_{t-s} F( X_s ) \,\mathrm{d}s 
+ \int_0^t S_{t-s} B( X_s ) \,\mathrm{d}W_s.
\end{equation}
\end{theorem}
\subsection{Kolmogorov equations in Hilbert spaces}\label{ssec:Kolmogorov}
In Lemma~\ref{lem:kolmogorov} below we present a specific variant of a standard regularity result for backward Kolmogorov equations in Hilbert spaces that is suitable for our purpose. For the convenience of the reader we include a sketch of its elementary proof. 
\begin{lemma} \label{lem:kolmogorov}
Assume Setting~\ref{setting:Part1},  let $( H , \langle \cdot, \cdot \rangle_{ H } , \left\| \cdot\right\|_{ H } ) $ be a non-trivial separable $ \R $-Hilbert space, 
for every 
$ A \in L(H) $, 
$ F \in C_\mathrm{b}^1( H, H ) $, 
$ B \in C_\mathrm{b}^1( H, L_2( U, H ) ) $, 
$x\in H$
let $ X^{A,F,B,x} \colon [0,T] \times \Omega \to H $ be an 
$ ( \mathbbm{F}_t )_{t \in [0,T]} $-predictable stochastic process 
which satisfies for every $t\in [0,T]$ that 
$ \sup_{ s \in [0,T] } \E\big[ \| X_s^{A,F,B,x}
\|_{H}^2 \big] < \infty $ and
\begin{equation}
[X_t^{A,F,B,x} ]_{\P,\,\calB(H)} 
= 
  [ e^{ tA } x ]_{\P,\,\calB(H)} 
  + \int_0^t e^{ (t-s)A }F( X_s^{A,F,B,x} ) \,\mathrm{d}s 
+ \int_0^t e^{ (t-s)A } B( X_s^{A,F,B,x} ) \,\mathrm{d}W_s,
\end{equation}
and for every 
$ A \in L(H) $, 
$ F \in C_\mathrm{b}^1( H, H ) $, 
$ B \in C_\mathrm{b}^1( H, L_2( U, H ) ) $, $\phi \in C^1_{\mathrm{b}}(H,\R)$
let\linebreak $ v^{A,F,B,\phi} \colon [0,T] \times H \to \R $ be the function which satisfies for every $ t \in [0,T] $, $ x \in H $ that 
$ v^{A,F,B,\phi}(t,x) = \E[ \phi( X_{T-t}^{A,F,B,x} ) ] $. 
Then 
\begin{enumerate}[(i)]
\item \label{it:kolmogorov_deriv}
it holds for every $ A \in L(H) $, $ F \in C_\mathrm{b}^2( H, H ) $,  
$ B \in C_\mathrm{b}^2( H, L_2( U, H ) ) $, $\phi\in C^2_{\mathrm{b}}(H,\R)$, $t\in [0,T]$,
$x\in H$ 
that 
$ v^{A,F,B,\phi} \in C^{ 1 , 2 }( [0,T] \times H, \R ) $ and 
\begin{equation}
\begin{aligned}
& 
\bigl( \tfrac{\partial}{\partial t} v^{A,F,B,\phi} \bigr) (t,x) 
+
\bigl( \tfrac{\partial}{\partial x} v^{A,F,B,\phi} \bigr) (t,x) ( A x + F(x) ) 
\\ & 
+ 
\tfrac{1}{2} 
\textstyle\! \sum\limits_{ u \in \U }\displaystyle
\bigl( \tfrac{\partial^2}{\partial x^2} v^{A,F,B,\phi} \bigr) (t,x)( B(x) u, B(x)u )
= 0,
\end{aligned}
\end{equation}
\item \label{it:kolmogorov_diffble}
it holds for every $k\in \N$, $ A \in L(H) $, $ F \in C_\mathrm{b}^k( H, H ) $,  
$ B \in C_\mathrm{b}^k( H, L_2( U, H ) ) $, $\phi\in C^k_{\mathrm{b}}(H,\R)$, $t\in [0,T]$
that 
$ (H\ni x \mapsto v^{A,F,B,\phi}(t,x)\in \R) \in C_{\mathrm{b}}^{ k }( H, \R ) $,
and
\item \label{it:kolmogorov_deriv_est}
it holds for every $k\in \N$, $c\in (0,\infty)$ that 
\begin{align}
 \sup\!
  \left\{ 
      \tfrac{ \| (\frac{\partial^k}{\partial x^k} v^{A,F,B,\phi})(t,x) \|_{ L^{(k)}(H, \R ) }}
      {\| \phi \|_{ C_{ \mathrm{b} }^k( H, \R ) }}
    \colon  
  \substack{
    t\in [0,T],\, x\in H,\,
    \phi\in C^k_{\mathrm{b}}(H,\R)\backslash\{0\},\,
    A \in L(H),\,
    \\
    F \in C_\mathrm{b}^k( H, H ),\,
    B \in C_\mathrm{b}^k( H, L_2( U, H ) ) \text{ with }
    \\
    \sup\limits_{t\in [0,T]} \| e^{tA} \|_{L(H)}
      +
      \left\| F \right\|_{C_\mathrm{b}^k( H, H ) }
      +
      \left\| B \right\|_{C_\mathrm{b}^k( H, L_2(U,H) )}
    \leq c
  }
  \right\}
  \notag
\\  
< \infty.
\end{align}
\end{enumerate}
\end{lemma}
\begin{proof}[Proof of Lemma~\ref{lem:kolmogorov}.]
Throughout this proof for every set $ S $ let $\left|S\right|\in \{0,1,2,\ldots\}\cup\{\infty\}$ be the cardinality of $S$, for every set $ S$ 
let $\mathcal{P}(S)$ be the power set of $S$, and
for every $j\in \N$ let $\Pi_j$ be
the set given by
\begin{multline}
 \Pi_j  = 
 \big\{
  \calS \subseteq \mathcal{P}(\N)
  \colon 
  \\
  [\emptyset \notin \calS]
  \land 
  [ \forall S_1,S_2 \in \calS \colon S_1\neq S_2 \Rightarrow S_1\cap S_2 = \emptyset]
  \land 
  [
    \cup_{S \in \calS} S = \{1,2,\ldots, j\} 
  ]
 \big\}
\end{multline}
(the set of all partitions of $\{1,\ldots,j\}$).
Observe that, e.g., 
Da Prato \& Zabczyk \cite[Theorem 9.16]{DaPratoZabczyk:1992} 
(cf., for example, also 
Harms \& M\"uller~\cite[item~(ii) in Lemma 2.2]{HarmsMueller:2017}) establishes 
item~\eqref{it:kolmogorov_deriv}.
Moreover, note that 
Andersson et al.~\cite[item~(ii) in Lemma~3.2]{AnderssonEtAl:2016b}
implies item~\eqref{it:kolmogorov_diffble}.
Next observe that Andersson et al.~\cite[item~(ix) in Theorem~2.1]{AnderssonEtAl:2016a}
demonstrates that for every $k\in \N$, $ A \in L(H) $, $ F \in C_\mathrm{b}^k( H, H ) $, 
$ B \in C_\mathrm{b}^k( H, L_2( U, H ) ) $, $t\in [0,T]$, $p\in [1,\infty)$ 
it holds that 
\begin{equation}
\big(
    H \ni x 
    \mapsto 
    [X^{A,F,B,x}_t]_{\P,\calB(H)} 
    \in L^p(\P;H)
\big)
\in C^{k}_{\mathrm{b}}(H,L^p(\P;H)).
\end{equation}
In addition, note that Andersson et al.~\cite[items (i)--(ii) and items (ix)--(x) in Theorem~2.1]{AnderssonEtAl:2016a}
(with $\alpha=\beta=\delta_1=\ldots=\delta_k =0$ in the notation of~\cite[item~(ii) in Theorem~2.1]{AnderssonEtAl:2016a})
ensures that for every $k\in \N$, $c\in (0,\infty)$, $p\in [1,\infty)$ it holds that 
\begin{multline}\label{eq:DX_bound}
\sup\!
  \left\{
    \tfrac{
      \|
	\frac{\partial^k}{\partial x^k}[X^{A,F,B,x}_t]_{\P,\calB(H)}
      \|_{L^{(k)}(H, L^p(\P;H))}
    }{
      t^{((\nicefrac{1}{2})\one_{[2,\infty)}(k))}    
     }
    \colon  
  \substack{
    t\in (0,T],\, x\in H,\, 
    A \in L(H),\,
    \\
    F \in C_\mathrm{b}^k( H, H ),\,
    B \in C_\mathrm{b}^k( H, L_2( U, H ) ) \text{ with }
    \\
      \sup\limits_{t\in [0,T]} \| e^{tA} \|_{L(H)}
      +
      \left\| F \right\|_{C_\mathrm{b}^k( H, H ) }
      +
      \left\| B \right\|_{C_\mathrm{b}^k( H, L_2(U,H) )}
    \leq c
  }
  \right\}
\\ 
  < \infty.
\end{multline}
Moreover, observe that Andersson et al.~\cite[item (v) in Lemma 3.2]{AnderssonEtAl:2016b}
(with $\alpha=\beta=\delta_1=\ldots=\delta_k =0$ in the notation of~\cite[item (v) in Lemma 3.2]{AnderssonEtAl:2016b})
proves that for every $k\in \N$, $ A \in L(H) $, 
$ F \in C_\mathrm{b}^k( H, H ) $,
$ B \in C_\mathrm{b}^k( H, L_2( U, H ) ) $, 
$\phi \in C_{\mathrm{b}}^k(H,\R)$ it holds that 
\begin{align}\label{eq:DPhi_bound}
&\notag
 \sup_{t\in [0,T),\,x\in H} 
 \big\|
   \big(\tfrac{\partial^k}{\partial x^k} v^{A,F,B,\phi}\big)(t,x)
 \big\|_{L^{(k)}(H,\R)}
 \leq 
 |\max\{ 1, T
  \}|^{\nicefrac{1}{2}\lfloor \nicefrac{k}{2} \rfloor}
 \left\| \phi \right\|_{C_{\mathrm{b}}^{k}(H,\R)}
\\ & \cdot 
  \Bigg[
  \sum_{\pi \in \Pi_k }
  \bigg(
    \prod_{I \in \pi}
    \bigg[
    \sup_{t\in (0,T],\,x\in H}
    \tfrac{
	\|\frac{\partial^{|I|}}{\partial x^{|I|}}[X^{A,F,B,x}_t]_{\P,\calB(H)}
	\|_{L^{(\left|I\right|)}(H, L^{\left| \pi \right|}(\P;H))}.
      }{
	t^{((\nicefrac{1}{2})\one_{[2,\infty)}(|I|))}
      }
      \bigg]
    \bigg)
  \Bigg].
\end{align}
Next note that 
for every \wen $k\in \N$, $ A \in L(H) $, $ F \in C_\mathrm{b}^k( H, H ) $,
$ B \in C_\mathrm{b}^k( H, L_2( U, H ) ) $, $\phi \in C_{\mathrm{b}}^k(H,\R)$ 
it holds that 
\begin{align}\label{eq:DPhi_bound_trivial}
 \sup_{ x \in H} 
 \big\|
   \big(\tfrac{\partial^k}{\partial x^k} v^{A,F,B,\phi}\big)(T,x)
 \big\|_{L^{(k)}(H,\R)}
 =
 \sup_{x\in H}
 \|
 \varphi^{(k)}(x)
 \|_{L^{(k)}(H,\R)}
 \leq 
 \| \phi \|_{C^{k}_{\mathrm{b}}(H,\R)}.
 \end{align}
Combining this,~\eqref{eq:DX_bound}, and~\eqref{eq:DPhi_bound} establishes item~\eqref{it:kolmogorov_deriv_est}. 
The proof of Lemma~\ref{lem:kolmogorov} is thus completed.
\end{proof}
\subsection{Preparatory lemmas}\label{ssec:prep}
The next result, Lemma~\ref{lemma:Lip_est_standard} below, is frequently used throughout this article. 

\begin{lemma}\label{lemma:Lip_est_standard}
Assume Setting~\ref{setting:Part1}, let $(\mathcal V, \left\|\cdot \right\|_{\mathcal V})$ and $(\mathcal W, \left\|\cdot \right\|_{\mathcal W})$ be $\R$-Banach spaces, and let $F\in \Lip(\mathcal V,\mathcal W)$.
Then 
\begin{enumerate}[(i)]
 \item\label{it:Lip_est_standard1} it holds for every 
$v \in \mathcal V$ that
$
 \| F(v) \|_{\mathcal W}
 \leq 
 \| F \|_{\Lip(\mathcal V,\mathcal W)}
 \max\{
  1,\| v \|_{\mathcal V}
 \}
$
and
\item\label{it:Lip_est_standard2}it holds for every 
$p\in [1,\infty)$, $\xi \in \calL^p(\P;\mathcal V)$ that
\begin{equation}
 \left\| F(\xi) \right\|_{\calL^p(\P;\mathcal W)}
 \leq  
 \left\| F \right\|_{\Lip(\mathcal V,\mathcal W)}
 (
  1 
  +
  \left\| \xi \right\|_{\calL^p(\P;\mathcal V)}
 ) .
\end{equation}
\end{enumerate}
\end{lemma}

\begin{proof}[Proof of Lemma~\eqref{lemma:Lip_est_standard}]
Observe that the fact that
\begin{align}
\| F\|_{ \Lip ( V, W ) } 
& 
= 
\| F(0) \|_{ W } 
+ 
\sup\!\left( 
  \left\{
    \tfrac{ \| F (x) - F (y) \|_W }{ \|x - y\|_{ V } }
    \colon x,y\in V, x\neq y
  \right\}
  \cup \{ 0 \}
\right)
\end{align}
implies that for every $v\in V$ it holds that 
\begin{equation}
\begin{split}
\left\| F(v)\right\|_{W} 
&\leq 
\left\| F(v) - F(0) \right\|_{W} + \left\| F(0)\right\|_{W}\\
&\leq
\Big[\tfrac{\|F(v)-F(0)\|_W}{\max\{1,\|v\|_V\}}\Big]\max\{1,\|v\|_V\}+\|F(0)\|_W\max\{1,\|v\|_V\}\\ 
&=
\Big[\|F(0)\|_W+\tfrac{\|F(v)-F(0)\|_W}{\max\{1,\|v\|_V\}}\Big]\max\{1,\|v\|_V\}\\ 
&\leq
\left\| F \right\|_{\Lip(V,W)}
 \max\{
  1,\left\| v \right\|_{V}
 \}.
\end{split}
\end{equation}
This establishes item~\eqref{it:Lip_est_standard1}. Moreover, note that item~\eqref{it:Lip_est_standard1} implies that for every 
$p\in [1,\infty)$, $\xi \in \calL^p(\P;V)$ it holds that
$
 \left\| F(\xi) \right\|_{\calL^p(\P;W)}
 \leq  
 \big\| \left\| F \right\|_{\Lip(V,W)}( 1 + \| \xi \|_V ) \big\|_{\calL^p(\P;\R)}
 \leq 
 \| F \|_{\Lip(V,W)}
 (
  1 
  +
  \left\| \xi \right\|_{\calL^p(\P;V)}
 ) .
$
This proves item~\eqref{it:Lip_est_standard2}. The proof of Lemma~\ref{lemma:Lip_est_standard} is thus completed.
\end{proof}

\begin{lemma}\label{lem:sup_estimates}
It holds 
\begin{enumerate}[(i)]
 \item\label{it:sup_estimates1} that $\sup_{\alpha\in [0,2], t\in (0,\infty)} 
    (t^{-\alpha} |1-\cos(t)|)= 2$ and
 \item\label{it:sup_estimates2} that $\inf_{\alpha \in \R\backslash [0,2]} \sup_{t\in (0,\infty)} (t^{-\alpha} |1-\cos(t)| ) = \infty$.
\end{enumerate}
\end{lemma}

\begin{proof}[Proof of Lemma~\ref{lem:sup_estimates}.]
First, note that for every $\alpha \in [0,2], t\in (1,\infty)$
it holds that 
\begin{equation}\label{eq:sup_estimates0}
    t^{-\alpha} |1-\cos(t)| \leq 2t^{-\alpha} \leq  2.
\end{equation}
Next observe that the fundamental theorem of calculus assures that for every $t_1\in(0,\infty)$, $t_2\in(t_1,\infty)$ it holds that
\begin{equation}\label{eq:sup_estimates00}
\begin{split}
&(t_2)^{-2}|1-\cos(t_2)|
=(t_2)^{-2}(1-\cos(t_2))\\
&=(t_1)^{-2}(1-\cos(t_1))+\int_{t_1}^{t_2}\bigg(\frac{\sin(s)s^2-(1-\cos(s))2s}{s^4}\bigg)\,ds\\
&=(t_1)^{-2}(1-\cos(t_1))+\int_{t_1}^{t_2}\bigg(\frac{\sin(s)s-(1-\cos(s))2}{s^3}\bigg)\,ds.
\end{split}
\end{equation}
In addition, note that the fundamental theorem of calculus implies that for every $s\in(0,\pi)$ it holds that
\begin{equation}
\begin{split}
&\sin(s)s-(1-\cos(s))2\\
&= \int_0^s (\cos(u)u+\sin(u)-2\sin(u))\,du=\int_0^s (\cos(u)u-\sin(u))\,du\\
&= \int_0^s\int_0^u(-\sin(r)r+\cos(r)-\cos(r))\,dr\,du= \int_0^s\int_0^u(-\sin(r)r)\,dr\,du\leq 0.
\end{split}
\end{equation}
This and \eqref{eq:sup_estimates00} demonstrate that the function $[(0,\pi)\ni t\mapsto t^{-2}|1-\cos(t)|\in(0,\infty)]$ is monotonically decreasing, i.e., that for every $t_1\in(0,\pi)$, $t_2\in(t_1,\pi)$ it holds that
\begin{equation}\label{eq:sup_estimates01}
(t_1)^{-2}|1-\cos(t_1)|\geq (t_2)^{-2}|1-\cos(t_2)|.
\end{equation}
Moreover, note that the fundamental theorem of calculus proves that for every $t\in(0,\infty)$ it holds that
\begin{equation}
\begin{split}
&t^{-2}|1-\cos(t)| 
= t^{-2}(\cos(0)-\cos(t)) \\
&= -t^{-2}\bigg[\int_0^t(-\sin(s))\,ds\bigg] 
= t^{-2}\bigg[\int_0^t\sin(s)\,ds\bigg] 
= t^{-2}\bigg[\int_0^t\int_0^s\cos(u)\,du\,ds\bigg] \\
&= t^{-2}\bigg[\int_0^t\int_0^s 1\,du\,ds\bigg] + t^{-2}\bigg[\int_0^t\int_0^s(\cos(u)-\cos(0))\,du\,ds\bigg] \\
&= \tfrac12 + t^{-2}\bigg[\int_0^t\int_0^s\int_0^u(-\sin(r))\,dr\,du\,ds\bigg].
\end{split}
\end{equation}
Hence, we obtain that for every $t\in(0,\infty)$ it holds that
\begin{equation}
\begin{split}
&\big|t^{-2}|1-\cos(t)|-\tfrac12\big| 
= t^{-2}\bigg|\int_0^t\int_0^s\int_0^u \sin(r)\,dr\,du\,ds\bigg|\\
&\leq t^{-2}\bigg[\int_0^t\int_0^s\int_0^u 1\,dr\,du\,ds\bigg] 
= t^{-2}\big[\tfrac{t^3}{3!}\big]
=\tfrac{t}6.
\end{split}
\end{equation}
Therefore, we obtain that 
\begin{equation}\label{eq:hopital}
\limsup_{t\searrow 0}\big|t^{-2}|1-\cos(t)|-\tfrac12\big|=0.
\end{equation}
Combining this and \eqref{eq:sup_estimates01} ensures that
\color{black}
for every $\alpha \in [0,2], t\in (0,1]$ it holds 
that 
\begin{equation}\label{eq:sup_estimates1}
\begin{aligned}
 t^{-\alpha} |1-\cos(t)| 
& \leq
t^{-2}|1-\cos(t)| \leq \tfrac12.
\end{aligned}
\end{equation}
In addition, note that $\sup_{\alpha\in [0,2], t\in (0,\infty)} (t^{-\alpha} |1-\cos(t)|) \geq \pi^{-0} |1-\cos(\pi)|=| 1 + 1 | = 2$.
Combining this, \eqref{eq:sup_estimates0}, and~\eqref{eq:sup_estimates1} establishes 
item~\eqref{it:sup_estimates1}.
Furthermore, observe that for every $\alpha \in (-\infty,0)$ it holds that  
\begin{equation}\label{eq:sup_estimates3}
\limsup_{t\rightarrow \infty} (t^{-\alpha} |1-\cos(t)|)  = \infty.
\end{equation}
\color{black}
In addition, note that \eqref{eq:hopital} shows that for every $\alpha\in(2,\infty)$ it holds that
\begin{equation}
\begin{split}
&\limsup_{t\searrow0}\big(t^{-\alpha}|1-\cos(t)|\big)
=
\limsup_{t\searrow0}\big(t^{-(\alpha-2)}\big[t^{-2}|1-\cos(t)|\big]\big)\\
&\geq 
\Big[\limsup_{t\searrow0}t^{-(\alpha-2)}\Big]\Big[\liminf_{t\searrow0}\big(t^{-2}|1-\cos(t)|\big)\Big]
= \tfrac12\Big[\limsup_{t\searrow0}t^{-(\alpha-2)}\Big]=\infty.
\end{split}
\end{equation}
\color{black}
Combining this and~\eqref{eq:sup_estimates3}
establishes item~\eqref{it:sup_estimates2}. The proof of Lemma~\ref{lem:sup_estimates} is thus completed.
\end{proof}
\smallskip

\noindent The estimate in Lemma~\ref{lemma:Schatten_hoelder} below is essentially well-known 
and follows from the H\"older inequality for Schatten norms.   
We refer to Meise \& Vogt \cite[Chapter~16]{MeiseVogt:2011} for the definition and further standard properties of Schatten 
class 
operators. 

\begin{lemma}\label{lemma:Schatten_hoelder}
Let $(H, \langle \cdot , \cdot \rangle_{H}, \left\| \cdot \right\|_{H})$
and 
$(U, \langle \cdot , \cdot \rangle_{U}, \left\| \cdot \right\|_{U})$
be separable $\R$-Hilbert spaces, 
let $\U\subseteq U$ be an orthonormal basis of $U$,
for every $p\in [1,\infty)$ let $(L_p(U,H), \left\| \cdot \right\|_{L_p(U,H)})$
be the $\R$-Banach space of Schatten-$p$ operators from $U$ to $H$ and let 
\color{black}
$r\in (0,\infty)$, $T\in L^{(2)}(H,\R)$, $A\in L_{1+r}(U,H)$,
$B\in L_{1+\nicefrac{1}{r}}(U,H)$. 
\color{black}
Then 
it holds that 
\begin{equation}\label{eq:Schatten_hoelder}
\begin{aligned}
\textstyle\sum\limits_{u\in \U}\displaystyle |T (Au,Bu)|
& 
\leq 
\left\| T \right\|_{L^{(2)}(H,\R)}
\left\| A \right\|_{L_{1+r}(U,H)}
\left\| B \right\|_{L_{1+\nicefrac{1}{r}}(U,H)}.
\end{aligned}
\end{equation}
\end{lemma}
\begin{proof}[Proof of Lemma~\ref{lemma:Schatten_hoelder}.]
Throughout this proof for every $p\in [1,\infty)$ let $(L_p(H,U),$ $\left\| \cdot \right\|_{L_p(H,U)})$ be
the $\R$-Banach space of Schatten-$p$ operators from $H$ to $U$ and for 
every $p\in [1,\infty)$ 
let $(L_p(U), \left\| \cdot \right\|_{L_p(U)})$ be 
the $\R$-Banach space of Schatten-$p$ operators from $U$ to $U$. 
Observe that the Riesz representation theorem ensures 
that there exists a unique $T_0\in L(H)$ such that for every $x,y \in H$ it holds that 
\begin{equation}\label{eq:Riesz}
 \langle T_0 x, y \rangle_{H} = T(x,y)
 \qquad \text{and} \qquad 
\| T_0 \|_{L(H)} = \| T \|_{L^{(2)}(H,\R)} 
 .
\end{equation}
\color{black}
Therefore, we obtain that
\begin{equation}\label{eq:schatten_hoelder_0}
\textstyle\sum\limits_{u\in \U}\displaystyle |T (Au,Bu)| 
=
\textstyle\sum\limits_{u\in \U}\displaystyle |\langle T_0 A u, Bu \rangle_{H}|
=
\textstyle\sum\limits_{u\in \U}\displaystyle |\langle B^* T_0 A u, u \rangle_{U}|.
\end{equation}
Next note that, e.g., Meise \& Vogt \cite[item 6.\ in Lemma~16.6 and item 2.\ in Lemma 16.7]{MeiseVogt:2011} ensures that $\left\| B^* \right\|_{L_{1+\nicefrac{1}{r}}(H,U)} = \left\| B \right\|_{L_{1+\nicefrac{1}{r}}(U,H)}$ and $\left\| T_0A \right\|_{L_{1+r}(U,H)}\leq\left\| T_0 \right\|_{L(H)}
\left\| A \right\|_{L_{1+r}(U,H)}$.
Combining the H\"older inequality for Schatten norms (see, e.g., Dunford \& Schwartz \cite[item~(c) in Lemma XI.9.14]{DunfordSchwartz:1963}) and \eqref{eq:Riesz} hence establishes that $B^*T_0 A\in L_1(U)$ and
\begin{equation}\label{eq:schatten_hoelder_1}
\| B^* T_0 A \|_{L_1(U)}
\leq 
\left\| B \right\|_{L_{1+\nicefrac{1}{r}}(U,H)}
\left\| T \right\|_{L^{(2)}(H,\R)}
\left\| A \right\|_{L_{1+r}(U,H)}
.
\end{equation}
\color{black}
Furthermore, note that for all sequences $(x_k)_{k\in\N},(y_k)_{k\in\N}\subseteq U$ with $\sum_{k\in\N}\|x_k\|_{U}\|y_k\|_U$ $<\infty$ and $B^*T_0A=\sum_{k\in\N}\langle\,\cdot\,,x_k\rangle_Uy_k$ it holds that
\begin{equation}
\begin{split}
\textstyle\sum\limits_{u\in\U}\displaystyle|\langle B^*T_0Au,u\rangle_U|
&=
\textstyle\sum\limits_{u\in\U}\displaystyle\Big|\Big\langle\textstyle\sum\limits_{k\in\N}\displaystyle\langle u,x_k\rangle_U y_k,u\Big\rangle_U\Big|
=
\textstyle\sum\limits_{u\in\U}\displaystyle\Big|\textstyle\sum\limits_{k\in\N}\displaystyle\langle u,x_k\rangle_U \langle y_k,u\rangle_U\Big|
\\
&\leq
\textstyle\sum\limits_{k\in\N}\sum\limits_{u\in\U}\displaystyle\big|\langle u,x_k\rangle_U \langle y_k,u\rangle_U\big|
\leq
\textstyle\sum\limits_{k\in\N}\displaystyle\|x_k\|_U\|y_k\|_U.
\end{split}
\end{equation}
The fact that 
\begin{equation}
\|B^*T_0A\|_{L_1(U)}
=
\inf\!\left\{\textstyle\sum\limits_{k\in\N}\displaystyle\|x_k\|_U\|y_k\|_U\colon
\substack{
(x_k)_{k\in\N},(y_k)_{k\in\N}\subseteq\, U \text{ with }\\
\sum\limits_{k\in\N}\|x_k\|_{U}\|y_k\|_U<\infty\text{ and }B^*T_0A=\sum\limits_{k\in\N}\langle\,\cdot\,,x_k\rangle_Uy_k}
\right\}
\end{equation}
therefore implies that
$
\sum_{u\in \U}|\langle B^* T_0 A u, u \rangle_{U}|
\leq \|B^*T_0A\|_{L_1(U)}.
$  
Combining this, \eqref{eq:schatten_hoelder_0}, and \eqref{eq:schatten_hoelder_1} establishes \eqref{eq:Schatten_hoelder}.
The proof of Lemma~\ref{lemma:Schatten_hoelder} is thus completed. 
\end{proof}
\section[Weak convergence rates for semilinear stochastic wave equations]{Weak convergence rates for temporal numerical approximations of semilinear stochastic wave equations}\label{sec:weak_convergence_rates}
%
%
%
%
\subsection{Setting}\label{ssec:setting_wave} 

In Setting~\ref{setting:Part2} below we present a framework for our convergence analysis that is frequently used throughout this section. 
For an easier understanding we provide some informal comments in advance: 
To begin with, note that 
the spectral structure of the operator $A$ is specified by its eigenbasis $\H$ and eigenvalues $\lambda_h$, $h\in\H$. 
Our standard choice for $A$ is the Dirichlet Laplacian on the  $L^2$ space of equivalence classes of square-integrable functions on the unit interval. 
The interpolation spaces $H_r$, $r\in\R$, are employed to measure smoothness in terms of powers of $-A$, 
while the product spaces $\bfH_r$, $r\in\R$, and the  operator 
$\bfA$ 
facilitate the interpretation of the wave equation 
as  a system of first order equations in time. 
The operator $\bfLambda$ is introduced to further simplify the analysis as it allows to interpret the spaces $\bfH_r$, $r\in\R$, as interpolation spaces associated to $\bfLambda$; compare Lemma~\ref{lem:Lambda_interpolation} in Subsection~\ref{ssec:wavesg} below. 
Further note that the function $\varphi$ 
plays the role of the test function occuring in the weak approximation error, 
whereas $\xi$, $F$, and $B$ 
constitute
the 
initial condition, the 
semilinearity in 
the drift term, and the diffusion coefficient of the abtract 
SPDE 
to be approximated. 
The parameters $\gamma$, $\beta$, and $\rho$ 
determine 
the spatial regularity of the solution; compare Lemma~\ref{lemma:Y_unif_Lpbddness} in Subsection~\ref{ssec:momentbounds} below. 
Depending on the choice of the step size parameter $h$ and the 
subset $I$ of the eigenbasis $\H$, 
the process defined by \eqref{eq:Y_mild} represents either 
an approximation of the mild solution or the mild solution itself.

\begin{setting}\label{setting:Part2}
Assume Setting~\ref{setting:Part1}, let
$\gamma\in (0,\infty)$, 
$\beta \in (\nicefrac{\gamma}{2},\gamma]\cap [\gamma-\nicefrac{1}{2},\gamma]$, $\rho \in [0,2(\gamma-\beta)]$, 
let $( H , \langle \cdot, \cdot \rangle_{ H } , \left\| \cdot\right\|_{ H } ) $ be a non-trivial separable $ \R $-Hilbert space, 
let $ \H \subseteq H $ be an 
%
%
orthonormal basis of $ H $, let $ \lambda \colon \H \to \R $ 
satisfy  
$ \sup_{ h \in \H } \lambda_{h} < 0 $ and $\sum_{h\in \H} |\lambda_h|^{-\beta} < \infty$, 
let $ A \colon D(A) \subseteq H \to H $ 
satisfy 
$
 D(A) = 
\bigl\{ v \in H \colon \sum_{ h \in \H } \abs{ \lambda_h \langle h,v \rangle_{H} }^2 < \infty \bigr\}
$
and 
$\big[\forall\, v \in D(A) \colon A v = \sum_{ h \in \H} \lambda_{ h } \langle h,v \rangle_{ H } h \big]$, 
let $ ( H_r , \langle \cdot , \cdot \rangle_{ H_r } , \left\| \cdot\right\|_{ H_r } )$, $ r \in \R $, be a family of interpolation spaces associated to $ - A $, 
for every $r\in\R$ let $ ( \bfH_r , \langle \cdot , \cdot \rangle_{ \bfH_r } , \left\| \cdot\right\|_{ \bfH_r } )$ be the $ \R $-Hilbert space 
which satisfies 
$ 
( \bfH_r , \langle \cdot , \cdot \rangle_{ \bfH_r } , \left\| \cdot\right\|_{ \bfH_r } ) 
=$ $ 
\bigl( H_{ \nicefrac{r}{2} } \times H_{ \nicefrac{r}{2} - \nicefrac{1}{2} }, 
\langle \cdot, \cdot \rangle_{ H_{ \nicefrac{r}{2} } \times H_{ \nicefrac{r}{2} - \nicefrac{1}{2} } },  
\norm{ \cdot }_{ H_{ \nicefrac{r}{2} } \times H_{ \nicefrac{r}{2} - \nicefrac{1}{2} } } \bigl) 
$,
let $ \bfA \colon D(\bfA) \subseteq \bfH_0 \to \bfH_0 $ 
satisfy 
$D(\bfA)=\bfH_1$ and 
$[\forall\, (v,w) \in \bfH_1 \colon \bfA( v , w )$ $=$ $( w , A v ) ]$, 
let $ \bfLambda \colon D(\bfLambda) \subseteq \bfH_0 \to \bfH_0 $ 
satisfy 
$D(\bfLambda)=\bfH_1$ and 
$\bigl[\forall\, (v,w)\in \bfH_1 \colon 
\bfLambda (v,w) = 
\sum_{ h \in \H } \abs{\lambda_h}^{ \nicefrac{1}{2} }  (\langle h ,v \rangle_{ H} h, \langle h , w \rangle_{ H } h) 
\bigr]
$,
let $\phi\in C^{4}_\mathrm{b}(\bfH_0,\R)$, 
$\xi\in \calL^2(\P\vert_{\mathbbm{F}_0};\bfH_{\max\{\rho, \gamma-\beta \}})$ 
satisfy 
$\E[ \left\| \xi \right\|_{\bfH_{0}}^6]<\infty$, let 
$F\in \Lip(\bfH_{\beta-\gamma},\bfH_0)$, 
$B \in  C_{\mathrm{b}}^4( \bfH_0, L_2(U,\bfH_0))$
satisfy 
$F|_{\bfH_0}\in C_{\mathrm{b}}^4(\bfH_0, \bfH_0)$, 
$F|_{\bfH_{\rho}} \in \Lip(\bfH_{\rho},\bfH_{2(\gamma-\beta)})$, 
and 
$B|_{\bfH_{\rho}} \in \Lip( \bfH_{\rho}, L(U, \bfH_{\gamma}) \cap L_2(U,\bfH_{\rho}))$, 
%
%
let 
$\mathfrak{m} \in [1,\infty)$, 
$\mathfrak{c}, \mathfrak{l} \in [0,\infty)$,
$\mu\colon \U \rightarrow (\R\backslash \{0\})$ 
satisfy for every $v,w \in \bfH_{\gamma-\beta}$ that 
\begin{align}
\label{eq:deriv_est_B_and_F}
   \max\!\big\{\|F|_{\bfH_0}\|_{C_{\mathrm{b}}^4(\bfH_0, \bfH_0)},\|B\|_{C_{\mathrm{b}}^4(\bfH_0, L_2(U,\bfH_0))}\big\}
& \leq \mathfrak{m},
\end{align}
\begin{equation}\label{eq:linGrowB}
  \textstyle\sum\limits_{u\in \U}\displaystyle |\mu_u|^{2} \left\| B(v)u \right\|_{\bfH_0}^2
 \leq 
 \mathfrak{c}^2
 \max\!\big\{
  1,
  \left\| v \right\|_{\bfH_{\gamma-\beta}}^2
 \big\},
\end{equation}
and
\begin{equation}\label{eq:LipB}
  \textstyle\sum\limits_{u\in \U}\displaystyle 
   \tfrac{\left\| (B(v) - B(w))u \right\|_{\bfH_0}^2}{|\mu_u|^{2}} 
 \leq 
 \mathfrak{l}^2 \left\| v - w \right\|_{\bfH_{\beta-\gamma}}^2,
\end{equation}
%
%
for every $I\subseteq \H$ let 
$P_I\colon{\cup_{r\in \R}H_r}\rightarrow{\cup_{r\in \R}H_r}$ 
and
$\bfP_I\colon{\cup_{r\in \R}\bfH_r}\rightarrow{\cup_{r\in \R}\bfH_r}$
satisfy 
for every $r\in \R$, $v\in H_r$, $w\in H_{r-\nicefrac{1}{2}}$ that 
$
  P_{I}(v) 
  =
  \sum_{h\in I}
    \langle 
     \left|\lambda_h \right|^{-r} h
      ,
      v
    \rangle_{H_r}
   \left| \lambda_h \right|^{-r} h
$
and 
$
  \bfP_{I}(v,w) 
  = 
  (P_I v , P_I w ), 
$ 
let
$\lfloor \cdot \rfloor_h \colon [0,\infty) \rightarrow \R$, $h\in [0,\infty)$, 
satisfy 
for every $h\in (0,\infty),$ $x\in [0,\infty)$ 
that
$
\lfloor x \rfloor_h 
=
\max( \{0, h, 2h, 3h, \ldots\} \cap [0,x] )$
and 
$\lfloor x \rfloor_0 =x$, 
and for every $I\subseteq \H$, $h\in [0,T]$ let
$ Y^{h,I}\colon [0,T] \times \Omega \rightarrow \bfP_I (\bfH_0)$
be an $(\mathbbm{F}_t)_{t\in[0,T]}$-predictable stochastic process which satisfies 
for every $t\in [0,T]$ that 
$\sup_{s\in[0,T]}\E\big[\| Y_{\round{s}{h}}^{h,I} \|_{\bfH_0}^2\big] < \infty$  
and
\begin{equation}
\begin{split}\label{eq:Y_mild}
& [Y_t^{h,I}]_{\P,\,\calB(\bfP_I(\bfH_0))} 
\\ & =
 [e^{t\bfA } \bfP_I \xi ]_{\P,\,\calB(\bfP_I(\bfH_0))} 
 +
 \int_{0}^{t} e^{ (t-\round{s}{h})\bfA} \bfP_I F ( Y^{h,I}_{\round{s}{h}} ) \,\mathrm{d}s
 +
 \int_{0}^{t} e^{ (t-\round{s}{h})\bfA } \bfP_I B( Y^{h,I}_{\round{s}{h}} ) \,\mathrm{d}W_s
.
\end{split}
\end{equation}
\end{setting}
\noindent{}
Note that the family of interpolation spaces $H_r$, $r\in\R$, introduced in Setting~\ref{setting:Part2} satisfies 
for every $r\in[0,\infty)$, $v\in H_r$ that $H_r=D((-A)^r)$ and $\|v\|_{H_r}=\|(-A)^rv\|_H$. 
Furthermore, observe 
that Setting~\ref{setting:Part2} ensures that for every $I\subseteq \H$, $t\in [0,T]$ it holds that \linebreak 
$\P\big[\int_0^t\big(\|e^{(t-s)\bfA}\bfP_IF(Y^{0,I}_s)\|_{\bfH_0}+\|e^{(t-s)\bfA}\bfP_IB(Y^{0,I}_s)\|_{L_2(U,\bfH_0)}^2\big)\,\mathrm ds<\infty\big]=1$ and
\begin{equation}
\begin{split}\label{eq:Y_mild2}
&
  [Y_t^{0,I}]_{\P,\,\calB(\bfP_I(\bfH_0))} 
\\ & =
 [e^{t\bfA } \bfP_I \xi ]_{\P,\,\calB(\bfP_I(\bfH_0))} 
 + \int_{0}^{t} e^{ (t-s)\bfA} \bfP_I F ( Y^{0,I}_{s} ) \,\mathrm{d}s
 +
 \int_{0}^{t} e^{ (t-s)\bfA } \bfP_I B ( Y^{0,I}_{s} ) \,\mathrm{d}W_s;
\end{split}
\end{equation}
cf., for example, Theorem~\ref{thm:existence} above, Lemma~\ref{lemma:semigroup} below, and 
Jacobe de Naurois et al. \cite[Remark 3.1]{JacobedeNauroisJentzenWelti:2021}  
for sufficient conditions which ensure the existence of such a process.
%
%
%
\subsection{Basic results for the linear wave equation}\label{ssec:wavesg}
The statement and the proof of the next result, Lemma~\ref{lem:Lambda_interpolation} below, can be found in,
e.g., Jacobe de Naurois et al.~\cite[Lemma 2.4]{JacobedeNauroisJentzenWelti:2015}.
\begin{lemma} \label{lem:Lambda_interpolation}
Assume Setting~\ref{setting:Part2}.
Then the $ \R $-Hilbert spaces $ ( \bfH_r , \langle \cdot , \cdot \rangle_{ \bfH_r } , \left\| \cdot\right\|_{ \bfH_r } )$, $ r \in \R $, are a family of interpolation spaces associated to $ \bfLambda $.
\end{lemma}
\noindent{}The next result, Lemma~\ref{lemma:semigroup} below, can be found in, e.g., 
Jacobe de Naurois et al.~\cite[Lemma 2.4]{JacobedeNauroisJentzenWelti:2021}.
A proof of Lemma~\ref{lemma:semigroup} can be found in, e.g., Lindgren~\cite[Section~5.3]{Lindgren:2012}.
\begin{lemma} \label{lemma:semigroup}
Assume Setting~\ref{setting:Part2} and let $ \bfS \colon [ 0 , \infty ) \to L( \bfH_0 ) $ be the function which
satisfies for every $ t \in [ 0 , \infty ) $, $ (v,w) \in \bfH_0 $ that
\begin{equation}
\bfS_t (v,w) = 
\twovector{ 
  \cos( t(-A)^{ \nicefrac{1}{2} }  ) v + (-A)^{ - \nicefrac{1}{2} } \sin( t(-A)^{ \nicefrac{1}{2} }  ) w, 
}{ 
  - (-A)^{ \nicefrac{1}{2} } \sin( t(-A)^{ \nicefrac{1}{2} }  ) v + \cos( t(-A)^{ \nicefrac{1}{2} }  ) w }.
\end{equation}
Then 
\begin{enumerate}[(i)]
 \item it holds that $ \bfS \colon [ 0 , \infty ) \to L( \bfH_0 ) $ is a strongly continuous semigroup of 
 bounded linear operators on $ \bfH_0 $ and
 \item it holds that $ \bfA \colon D ( \bfA ) \subseteq \bfH_0 \to \bfH_0 $ is the generator of $ \bfS $.
\end{enumerate}
\end{lemma}
\noindent{}The statement and the proof of the next result, Lemma~\ref{lemma:sg-estimate1} below, can be found in, e.g., Jacobe de Naurois et al.~\cite[Lemma 2.6]{JacobedeNauroisJentzenWelti:2015}.
\begin{lemma}\label{lemma:sg-estimate1}
Assume Setting~\ref{setting:Part2}.
Then 
$ 
 \sup_{t\in [0,\infty)} 
    \| e^{t\mathbf{A}} \|_{L(\mathbf{H}_0)}
  =1.
$
\end{lemma}
\noindent{}The next result, Lemma~\ref{lemma:sg-estimate2} below, 
provides another useful elementary estimate for the semigroup $(e^{t\bfA})_{t\in [0,\infty)}$ generated by the operator $\bfA\colon D(\bfA)\subseteq\bfH_0\to\bfH_0$ from Setting~\ref{setting:Part2}; 
cf.,~e.g., Kovacs et al.~\cite[Lemma 4.4]{KovacsEtAl:2013} for a similar result.
%
\begin{lemma}\label{lemma:sg-estimate2}
Assume Setting~\ref{setting:Part2} and let $\alpha \in [0,1], t\in (0,\infty)$.
Then 
\begin{equation}\label{eq:sg_estimate2}
    t^{-\alpha} 
    \| \bfLambda^{-\alpha} (\id_{\mathbf{H}_0} - e^{t\mathbf{A}}) \|_{L(\mathbf{H}_0)} 
  \leq 
  \sqrt{2}\bigg[ 
  \sup_{s \in (0,\infty)}
   (s^{-\alpha} |1-e^{\mathbf{i}s}|)\bigg]
  \leq
  2^{\nicefrac32}. 
\end{equation}
\end{lemma}
\begin{proof}[Proof of Lemma~\ref{lemma:sg-estimate2}.]
First, observe 
that for every $s\in(0,\infty)$ it holds that
\begin{equation}\label{eq:sg_estimate_proof_0}
\begin{split}
s^{-\alpha}|1-e^{\mathbf{i}s}|
&=
s^{-\alpha}\big[|1-\cos(s)|^2+|\sin(s)|^2\big]^{\nicefrac12}\\
&=
s^{-\alpha}\big[1-2\cos(s)+|\cos(s)|^2+|\sin(s)|^2\big]^{\nicefrac12}\\
&=
s^{-\alpha}\big[2-2\cos(s)\big]^{\nicefrac12}
=
\sqrt{2}\big(s^{-\alpha}|1-\cos(s)|^{\nicefrac12}\big).
\end{split}
\end{equation}
In addition, note 
that Lemma~\ref{lemma:semigroup} implies that for every $(v,w)\in \bfH_0$ it holds that
\begin{equation}
\begin{aligned}
&  \bfLambda^{-\alpha} (\id_{\mathbf{H}_0} - e^{t\bfA} ) (v,w)
\\ &
= \bfLambda^{-\alpha}\!\twovector{ 
  (\id_{H} - \cos( t(-A)^{ \nicefrac{1}{2} }  )) v 
  - (-A)^{ - \nicefrac{1}{2} } \sin( t(-A)^{ \nicefrac{1}{2} } ) w, 
  }{ 
  (-A)^{ \nicefrac{1}{2}  } \sin( t(-A)^{ \nicefrac{1}{2} }  ) v 
  + ( \id_{H_{-\nicefrac12}} - \cos( t(-A)^{ \nicefrac{1}{2} } )) w 
  }
\\ &
=  
  \twovector{ 
  (-A)^{-\nicefrac{\alpha}{2}} (\id_{H} - \cos( t(-A)^{ \nicefrac{1}{2} }  )) v 
  - (-A)^{ - \nicefrac{(1+\alpha)}{2} } \sin( t(-A)^{ \nicefrac{1}{2} } ) w, 
  }{ 
  (-A)^{ \nicefrac{(1-\alpha)}{2}  } \sin( t(-A)^{ \nicefrac{1}{2} }  ) v 
  + (-A)^{-\nicefrac{\alpha}{2}} ( \id_{H_{-\nicefrac12}} - \cos( t(-A)^{ \nicefrac{1}{2} } )) w 
  }
  .
\end{aligned}
\end{equation}
Hence, we obtain that for every $(v,w)\in \mathbf{H}_0$ it holds that
\begin{equation}\label{eq:sg_estimate_proof_1}
\begin{aligned}
&  t^{-\alpha}\| \bfLambda^{-\alpha} (\id_{\mathbf{H}_0} - e^{t\bf{A}} )(v,w)\|_{\mathbf{H}_0}
\\& =
t^{-\alpha}\left[
\Big\|(-A)^{-\nicefrac{\alpha}{2}} \big(\id_{H} - \cos( t(-A)^{ \nicefrac{1}{2} }  )\big) v 
  - (-A)^{ - \nicefrac{(1+\alpha)}{2} } \sin( t(-A)^{ \nicefrac{1}{2} } ) w\,\Big\|_{H}^2\right.
\\
&\quad +
\left.\Big\|(-A)^{ \nicefrac{(1-\alpha)}{2}  } \sin( t(-A)^{ \nicefrac{1}{2} }  ) v 
  + (-A)^{-\nicefrac{\alpha}{2}} \big( \id_{H_{-\nicefrac12}} - \cos( t(-A)^{ \nicefrac{1}{2} }) \big) w\, \Big\|_{H_{-\nicefrac12}}^2
\right]^{\nicefrac12}
\\& = 
  \left[\ 
    \sum_{h\in \H} 
        t^{-2\alpha} | \lambda_h |^{-\alpha}
        \left|
            \big(
                1
                -
                \cos(
                    t| \lambda_h |^{\nicefrac{1}{2}}
                ) 
            \big) 
            \left\langle h,v \right\rangle_{H}
            +
            \sin( 
                t | \lambda_h |^{\nicefrac{1}{2}}
            )
            \langle
            \left| \lambda_h \right|^{\nicefrac{1}{2}} h,w
            \rangle_{H_{-\nicefrac{1}{2}}}
        \right|^2
  \right.
\\ & \quad +
  \left.
    \sum_{h\in \H} 
      t^{-2\alpha} \left| \lambda_h \right|^{-\alpha}
        \left|
            \sin( 
               t | \lambda_h |^{\nicefrac{1}{2}}
            )
            \left\langle h,v \right\rangle_{H}
            +
            \big(
                1
                -
                \cos(
                    t | \lambda_h |^{\nicefrac{1}{2}}
                ) 
            \big)
            \langle
            \left| \lambda_h \right|^{\nicefrac{1}{2}}  h,w
            \rangle_{H_{-\nicefrac{1}{2}}}
       \right|^2
  \right]^{\nicefrac{1}{2}}. 
\\& \leq \sqrt{2}  
  \left[\ 
    \sum_{h\in \H} 
        t^{-2\alpha} | \lambda_h |^{-\alpha}
        \left(
            \big|
                1
                -
                \cos(
                    t| \lambda_h |^{\nicefrac{1}{2}} 
                ) 
            \big|^2
            +
            \big|\!
                \sin( 
                    t | \lambda_h |^{\nicefrac{1}{2}}
                )
            \big|^2
        \right)
        \left|\left\langle h,v \right\rangle_{H}\right|^2
    \right.
\\ & \quad 
    \left.
        + 
        \sum_{h\in \H} 
            t^{-2\alpha} | \lambda_h |^{-\alpha}
            \left(
                \big|
                    1
                    -
                    \cos(
                        t| \lambda_h |^{\nicefrac{1}{2}} 
                    ) 
                \big|^2
                +
                \big|\!
                    \sin( 
                        t | \lambda_h |^{\nicefrac{1}{2}}
                    )
                \big|^2
            \right)
            \big|
                \langle
                \left| \lambda_h \right|^{\nicefrac{1}{2}} h,w
                \rangle_{H_{-\nicefrac{1}{2}}}
            \big|^2
  \right]^{\nicefrac{1}{2}}.
\end{aligned}
\end{equation}
This shows that for every $(v,w)\in\bfH_0$ it holds that
\begin{equation}\label{eq:step_proof_wave_semigroup}
\begin{aligned}
&  t^{-\alpha}\| \bfLambda^{-\alpha} (\id_{\mathbf{H}_0} - e^{t\bf{A}} )(v,w)\|_{\mathbf{H}_0}
\\& \leq \sqrt{2}  
  \left[\ 
    \sup_{h\in \H}\Big\{ 
        t^{-2\alpha} | \lambda_h |^{-\alpha}
        \left(
            \big|
                1
                -
                \cos(
                    t| \lambda_h |^{\nicefrac{1}{2}} 
                ) 
            \big|^2
            +
            \big|\!
                \sin( 
                    t | \lambda_h |^{\nicefrac{1}{2}}
                )
            \big|^2
        \right) \Big\} \right]^{\nicefrac12}
\\ & \quad 
\cdot\left[
\bigg(\sum_{h\in\H}|\langle h,v\rangle_H|^2\bigg) 
+
\bigg(\sum_{h\in\H}\big|\langle |\lambda_h|^{\nicefrac12}h, w \rangle_{H_{-\nicefrac12}}\big|^2\bigg)
\right]^{\nicefrac12}
\\& = \sqrt{2}  
  \left[\ 
    \sup_{h\in \H}\left\{ 
        (t|\lambda_h |^{\nicefrac12})^{-\alpha}
        \left(
            \big|
                1
                -
                \cos(
                    t| \lambda_h |^{\nicefrac{1}{2}} 
                ) 
            \big|^2
            +
            \big|\!
                \sin( 
                    t | \lambda_h |^{\nicefrac{1}{2}}
                )
            \big|^2
        \right)^{\nicefrac12} \right\} \right]
\|(v,w)\|_{\bf H_0}.  
\end{aligned}
\end{equation}
Combining this and the fact that for every $s\in\R$ it holds that 
$|1-e^{\mathbf{i}s}|^{2}=2|1-\cos(s)|$ demonstrates that for every $(v,w)\in\bfH_0$ it holds that
\begin{equation}\label{eq:sg_estimate_proof_2}
\begin{aligned}
&  t^{-\alpha}\| \bfLambda^{-\alpha} (\id_{\mathbf{H}_0} - e^{t\bf{A}} )(v,w)\|_{\mathbf{H}_0}
  \\ & 
  \leq 
    \sqrt{2}
    \bigg[
    \sup_{s\in (0,\infty)} \!
    \Big\{
	s
	^{-\alpha} 
	\big( 
	  |1-\cos(s)|^{2}
	  +
	  |\sin(s)|^2
	\big)^{\!\nicefrac{1}{2}}
  \Big\}
  \bigg]
    \left\|
     (v,w)
    \right\|_{\bfH_0}
 \\ & 
    = \sqrt{2}
    \bigg[
    \sup_{s \in (0,\infty)} 
    (
    s^{-\alpha} |1-e^{\mathbf{i} s}|
    )
    \bigg]
    \left\|
     (v,w)
    \right\|_{\bfH_0}
    =
    2
  \bigg[
    \sup_{s\in (0,\infty)} \!
    \big(
     s^{-\alpha} |1-\cos(s)|^{\nicefrac{1}{2}}
    \big)
  \bigg]
    \left\|
     (v,w)
    \right\|_{\bfH_0}.
 \end{aligned}
\end{equation}
This and Lemma~\ref{lem:sup_estimates} establish \eqref{eq:sg_estimate2}. The proof of Lemma~\ref{lemma:sg-estimate2} is thus completed.
\end{proof}

\noindent
We end this subsection with a further auxiliary result for the semigroup $(e^{t\bf A})_{t\in[0,\infty)}$.
\begin{lemma}\label{lemma:sg-projection}
Assume Setting~\ref{setting:Part2} and let $I\subseteq\H$ be finite. Then it holds for every $t\in[0,\infty)$ that $\bfA\bfP_I|_{\bfH_0}\in L(\bfH_0)$ and
\begin{equation}\label{eq:sg_projection}
e^{t(\bfA\bfP_I|_{\bfH_0})}
=
e^{t\bfA}\bfP_I|_{\bfH_0}+\bfP_{\H\setminus I}|_{\bfH_0}.
\end{equation}
\end{lemma}
\begin{proof}[Proof of Lemma~\ref{lemma:sg-projection}.]
First, note that the finiteness of $I\subseteq\H$ ensures that
for every $x\in\bfH_0$ it holds that $\bfP_Ix\in \bfH_1=D(\bf A)$ and $\bfA\bfP_I|_{\bfH_0}\in L(\bfH_0)$.
This and 
Lemma~\ref{lemma:semigroup} imply that for every $s,t\in[0,\infty)$, $x\in\bfH_0$ it holds that
\begin{equation}\label{eq:sg_projection_proof_1}
\begin{split}
&(e^{s\bfA}\bfP_I+\bfP_{\H\setminus I})(e^{t\bfA}\bfP_I+\bfP_{\H\setminus I})x\\
&=
e^{(s+t)\bfA}\bfP_Ix+e^{s\bfA}\bfP_I\bfP_{\H\setminus I}x+e^{t\bfA}\bfP_I\bfP_{\H\setminus I}x+(\bfP_{\H\setminus I})^2x\\
&=
(e^{(s+t)\bfA}\bfP_I+\bfP_{\H\setminus I})x
\end{split}
\end{equation}
and
\begin{equation}\label{eq:sg_projection_proof_2}
\begin{split}
&\limsup_{h\searrow0}\left\|\tfrac1h\big[(e^{h\bfA}\bfP_I+\bfP_{\H\setminus I})x-x\big]-\bfA\bfP_Ix\right\|_{\bfH_0}\\
&=
\limsup_{h\searrow0}\left\|\tfrac1h\big[(e^{h\bfA}\bfP_I+\bfP_{\H\setminus I})x-(\bfP_I+\bfP_{\H\setminus I})x\big]-\bfA\bfP_Ix\right\|_{\bfH_0}\\
&=
\limsup_{h\searrow0}\left\|\tfrac1h\big[e^{h\bfA}\bfP_Ix-\bfP_Ix\big]-\bfA\bfP_Ix\right\|_{\bfH_0}
=
0.
\end{split}
\end{equation}
Moreover, observe that for every $x\in\bfH_0$ it holds that $(e^{0\bfA}\bfP_I+\bfP_{\H\setminus I})x=x$. Combining this, \eqref{eq:sg_projection_proof_1}, and \eqref{eq:sg_projection_proof_2} demonstrates that $\big(e^{t\bfA}\bfP_I|_{\bfH_0}+\bfP_{\H\setminus I}|_{\bfH_0}\big)_{t\geq0}$ is a strongly continuous semigroup of bounded linear operators on $\bfH_0$ with generator $\bfA\bfP_I|_{\bfH_0}\in L(\bfH_0)$. 
Hence, we obtain that $(e^{t(\bfA\bfP_I|_{\bfH_0})})_{t\geq0}=\big(e^{t\bfA}\bfP_I|_{\bfH_0}+\bfP_{\H\setminus I}|_{\bfH_0}\big)_{t\geq0}$. The proof of Lemma~\ref{lemma:sg-projection} is thus completed.
\end{proof}
\subsection{A priori bounds for the numerical approximations}\label{ssec:momentbounds}
\begin{lemma}\label{lemma:Y_unif_Lpbddness}
Assume Setting~\ref{setting:Part2}.
Then 
\begin{enumerate}[(i)]
 \item\label{it:Ysmooth} it holds for every $h,t\in [0,T]$, 
 $I\subseteq \H$ that $\P( Y_t^{h,I} \in \bfH_{\max\{ \rho, \gamma - \beta \}}) = 1$, 
 \item\label{it:Y_unif_L4bdd} it holds that 
 $ 
    \sup_{h,t\in [0,T],\, I\subseteq \H} \E \big[ \| Y_t^{h,I} \|_{\bfH_0}^6 \big] < \infty
 $, and 
 \item\label{it:Y_unif_L2bdd} it holds that 
 $ \sup_{h,t\in [0,T],\, I\subseteq \H} \E \big[ \| Y_t^{h,I} \|_{\bfH_{\max\{\rho, \gamma -\beta\}}}^2 \big] < \infty$.
\end{enumerate}
\end{lemma}
\begin{proof}[Proof of Lemma~\ref{lemma:Y_unif_Lpbddness}.] 
First, note that
\begin{equation}\label{eq:FLip_trivial}
\begin{aligned}
&
  \| F|_{\bfH_{ \rho}} \|_{\Lip(\bfH_{ \rho },\bfH_{\max\{ \rho, \gamma - \beta\}})} 
 \leq 
  \| \bfLambda^{\max\{\rho, \gamma -\beta\}-2(\gamma-\beta)} \|_{L(\bfH_0)} 
  \| F|_{\bfH_{ \rho}} \|_{\Lip(\bfH_{ \rho },\bfH_{2( \gamma - \beta)})} < \infty.
\end{aligned}
\end{equation}
The fact that $\E[\| \xi\|_{\bfH_{0}}^6 + \|\xi\|_{\bfH_{\rho}}^2]<\infty$ and Theorem~\ref{thm:existence} 
(with  $H\in\{\bfH_0,\bfH_\rho\}$, $S_t\in\{e^{t\bfA},e^{t\bfA}|_{\bfH_\rho}\}$, $p\in\{6,2\}$, 
$F\in\{\bfP_IF|_{\bfH_0},\bfP_IF|_{\bfH_\rho}\}$, $B\in\{\bfP_IB,\bfP_IB|_{\bfH_\rho}\}$, $\xi=\xi$ 
for $t\in[0,\infty)$, $I\subseteq \H$ in the notation of Theorem~\ref{thm:existence}) 
hence imply that there exist up to modifications unique $(\mathbbm{F}_t)_{t\in[0,T]}$-predictable stochastic processes 
$X^I\colon [0,T] \times \Omega \rightarrow \bfH_{\rho}$, $I\subseteq\H$, which satisfy for every $I\subseteq \H$, $t\in [0,T]$   
that 
$
\sup_{s\in [0,T]}  \E\big[ \| X_s^I \|_{\bfH_{0}}^6  +  \| X_s^I \|_{\bfH_{\rho}}^2 \big] < \infty 
$
and 
\begin{equation}
\begin{split}
& 
  [X_t^I]_{\P,\,\calB(\bfH_{\rho})} 
= 
 [e^{t \bfA} \bfP_I \xi]_{\P,\,\calB(\bfH_{\rho})} 
 + \int_{0}^{t} e^{(t-s)\bfA} \bfP_I F( X_s^I) \,\mathrm{d}s
 + \int_{0}^{t} e^{(t-s)\bfA} \bfP_I B( X_s^I) \,\mathrm{d}W_s.
\end{split}
\end{equation}
Combining this, \eqref{eq:Y_mild}, and Theorem~\ref{thm:existence} ensures that for every $I\subseteq \H$, $t\in [0,T]$ it 
holds that $\P(X_t^I = Y_t^{0,I})=1$. 
The fact that for every $I\subseteq\H$ it holds that
$
\sup_{t\in [0,T]}  \E\big[ \| X_t^I \|_{\bfH_{0}}^6  +  \| X_t^I \|_{\bfH_{\rho}}^2 \big] < \infty
$ 
hence implies that 
that for every $I\subseteq \H$ it holds that  
\begin{equation}
 \begin{aligned}\label{eq:Y0_Lpbddness}
  \sup_{t\in [0,T]}
	\E \big[ 
	  \|
	    Y^{0,I}_{t} 
	  \|_{\bfH_{0}}^6
	+
	  \|
	    Y^{0,I}_{t} 
	  \|_{\bfH_{\rho}}^2
	\big]
    < \infty.
 \end{aligned}
\end{equation}
Next note that Da Prato \& Zabczyk~\cite[Lemma 7.7]{DaPratoZabczyk:1992}, item~\eqref{it:Lip_est_standard2} in Lemma~\ref{lemma:Lip_est_standard}, Lemma~\ref{lemma:sg-estimate1}, and~\eqref{eq:Y_mild}
prove that for every $\delta,\theta\in [0,\infty)$, $p\in [2,\infty)$,
$h,t\in [0,T]$, $I\subseteq \H$ it holds that 
\begin{equation}\label{eq:Y_Lpest}
\begin{aligned}
& 
 \big( \E \big[ 
  \|
    Y_t^{h,I} 
  \|_{\bfH_{\theta}}^p
 \big] \big)^{\!\nicefrac{1}{p}}
\\ & 
  \leq 
  \big( \E \big[ 
  \| e^{t\bfA } \bfP_I \xi 
  \|_{\bfH_{\theta}}^p
  \big] \big)^{\!\nicefrac{1}{p}}
  +
  \int_{0}^{t}\!\!
    \Big( \E \Big[ \big\|
	  e^{(t-\round{s}{h})\bfA} \bfP_I F\big(Y^{h,I}_{\round{s}{h}}\big) 
    \big\|_{\bfH_{\theta}}^p
    \Big] \Big)^{\!\nicefrac{1}{p}}
  \,ds
\\ & \quad 
 +
 \sqrt{ 
  \tfrac{p(p-1)}{2}
 }
 \left(
 \int_{0}^{t} \!\!
    \Big( \E \Big[ \big\|
      e^{ (t-\round{s}{h})\bfA } \bfP_I 
      B\big(
	Y^{h,I}_{\round{s}{h}}
      \big)  
  \big\|_{L_2(U,\bfH_{\theta})}^p
  \Big] \Big)^{\!\nicefrac{2}{p}}
  \,ds
  \right)^{\!\!\nicefrac{1}{2}}
\\ &
  \leq
  \big( \E \big[ \| \xi 
  \|_{\bfH_{\theta}}^p
  \big] \big)^{\!\nicefrac{1}{p}}
  +  
  \| F|_{\bfH_{\delta}} \|_{\Lip(\bfH_{\delta},\bfH_{\theta})}
    \int_{0}^{t} \!\!
    \big(
	1
	+ 
	\big( \E \big[ \|
	  Y^{h,I}_{\round{s}{h}}
	\|_{\bfH_{\delta}}^p
	\big] \big)^{\!\nicefrac{1}{p}}
	\big)
    \,ds
\\ & \quad
 +
 \sqrt{ 
  \tfrac{p(p-1)}{2}
 }
 \| B|_{\bfH_{\delta}}\|_{\Lip(\bfH_{\delta},L_2(U,\bfH_{\theta}))}
 \left(
 \int_{0}^{t} \!\!
     \big( 
      1
      +
      \big( \E \big[ \|
	  Y^{h,I}_{\round{s}{h}} 
	\|_{\bfH_{\delta}}^p
      \big] \big)^{\!\nicefrac{1}{p}}
      \big)^{2}
  \,ds
  \right)^{\!\!\nicefrac{1}{2}}
\\ & 
  \leq
  \big( \E \big[ \| \xi 
  \|_{\bfH_{\theta}}^p
  \big] \big)^{\!\nicefrac{1}{p}}
  +
  \bigg(
      \sqrt{t}
  +
  \left(
  \int_{0}^{t} \!\!
	\big( \E \big[ \|
	  Y^{h,I}_{\round{s}{h}} 
	\|_{\bfH_{ \delta}}^p
	\big] \big)^{\!\nicefrac{2}{p}}
    \,ds
    \right)^{\!\!\nicefrac{1}{2}}
  \bigg)
\\ & \quad
 \cdot
  \Big(
    \sqrt{t}
    \| F|_{\bfH_{ \delta}} \|_{\Lip(\bfH_{ \delta },\bfH_{\theta})}
    +
    \sqrt{\tfrac{p(p-1)}{2}}
    \| B|_{\bfH_{\delta}} \|_{\Lip(\bfH_{ \delta },L_2(U,\bfH_{\theta}))}
  \Big)
 .
\end{aligned}
\end{equation}
Moreover, note that 
\begin{equation}\label{eq:BLip_trivial}\begin{aligned}
&  
  \| B|_{\bfH_{\rho}} \|_{\Lip(\bfH_{ \rho },L_2(U,\bfH_{\max\{\rho,\gamma - \beta\}}))}
\\ &
  \leq 
    \max\!\big\{ 
      \| B |_{\bfH_{\rho}} \|_{\Lip(\bfH_{ \rho },L_2(U,\bfH_{\rho}))}
      ,\,
      \| \bfLambda^{-\beta} \|_{L_2(\bfH_0)} 
      \| B|_{\bfH_{\rho}} \|_{\Lip(\bfH_{ \rho },L(U,\bfH_{\gamma}))}
    \big\} < \infty.
\end{aligned}
\end{equation}
This,~\eqref{eq:FLip_trivial}, \eqref{eq:Y0_Lpbddness}, and~\eqref{eq:Y_Lpest} (with $p=2$, $\delta =\rho$, 
$\theta = \max\{ \rho, \gamma - \beta\}$ in the notation of~\eqref{eq:Y_Lpest}) 
ensure that for every $I\subseteq \H$ it holds that 
\begin{equation}\label{eq:Y0_L2bddness}
 \sup_{t\in [0,T]} \E \big[ \| Y_t^{0,I} \|_{\bfH_{\max\{\rho,\gamma-\beta\}}}^2 \big] < \infty.
\end{equation} 
In addition, note that~\eqref{eq:Y_Lpest} (with $p=6$, $\delta =0$, 
$\theta = 0$ in the notation of~\eqref{eq:Y_Lpest})
implies that for every $h\in (0,T]$, $I\subseteq \H$, $k\in \N_0\cap [0,T/h)$ it holds that  
\begin{equation}\label{eq:Y_Lpest2}
\begin{aligned}
& 
\sup_{t\in (kh,(k+1)h]\cap[0,T]}
 \big( \E \big[ 
  \|
    Y_t^{h,I} 
  \|_{\bfH_{0}}^6
 \big] \big)^{\!\nicefrac{1}{6}}
\\ & 
  \leq
  \big( \E \big[ \| \xi 
  \|_{\bfH_{0}}^6
  \big] \big)^{\!\nicefrac{1}{6}}
  +
  \sqrt{(k+1)h}
      \Big( 
	1
	+
	\sup_{j\in \{0,1,\ldots,k\}}
	\big( \E \big[ \|
	  Y^{h,I}_{jh} 
	\|_{\bfH_{0}}^6
	\big] \big)^{\!\nicefrac{1}{6}}
      \Big)
\\ & \quad
 \cdot
 \left(
    \sqrt{(k+1)h}
    \| F|_{\bfH_{ 0}} \|_{\Lip(\bfH_{ 0 },\bfH_{0})}
    +
    \sqrt{15}
    \| B|_{\bfH_{0}} \|_{\Lip(\bfH_{ 0 },L_2(U,\bfH_{0}))}
 \right)
 .
\end{aligned}
\end{equation}
Hence, we obtain that for every $h\in (0,T]$, $I \subseteq \H$ it holds that
\begin{equation}\label{eq:Yh_L4bddness}
\sup_{t\in [0,T] } \E \big[ \| Y_{t}^{h,I} \|_{\bfH_0}^6 \big] <\infty. 
\end{equation}
Moreover, observe that~\eqref{eq:Y_Lpest} (with $p=2$, $\delta =\rho$, 
$\theta = \rho$ in the notation of~\eqref{eq:Y_Lpest})
implies that for every $h\in (0,T]$, $I\subseteq \H$, $k\in \N_0\cap [0,T/h)$ it holds that  
\begin{equation}\label{eq:Y_Lpest3}
\begin{aligned}
& 
\sup_{t\in (kh,(k+1)h]\cap[0,T]}
 \big( \E \big[ 
  \|
    Y_t^{h,I} 
  \|_{\bfH_{\rho}}^2
 \big] \big)^{\!\nicefrac{1}{2}}
\\ & 
  \leq
  \big( \E \big[ \| \xi 
  \|_{\bfH_{\rho}}^2
  \big] \big)^{\!\nicefrac{1}{2}}
  +
  \sqrt{(k+1)h}
      \Big( 
	1
	+
	\sup_{j\in \{0,1,\ldots,k\}}
	\big( \E \big[ \|
	  Y^{h,I}_{jh} 
	\|_{\bfH_{\rho}}^2
	\big] \big)^{\!\nicefrac{1}{2}}
      \Big)
\\ & \quad
 \cdot
 \left(
    \sqrt{(k+1)h}
    \| F|_{\bfH_{ \rho}} \|_{\Lip(\bfH_{ \rho },\bfH_{\rho})}
    +
    \| B|_{\bfH_{\rho}} \|_{\Lip(\bfH_{ \rho },L_2(U,\bfH_{\rho}))}
 \right)
 .
\end{aligned}
\end{equation}
Hence, we obtain that for every $h\in (0,T]$, $I \subseteq \H$ it holds that
$\sup_{ t\in [0,T] } \E [ \| Y_{t}^{h,I} \|_{\bfH_{\rho}}^2 ]$ $<\infty$.
Combining this,~\eqref{eq:FLip_trivial},~\eqref{eq:Y_Lpest} (with $p=2$, $\delta = \rho $, 
$\theta = \max\{ \rho, \gamma - \beta\}$ in the notation of~\eqref{eq:Y_Lpest}), and~\eqref{eq:BLip_trivial} 
implies that
for every $h\in (0,T]$, $I\subseteq \H$ it holds that
\begin{equation}\label{eq:Yh_L2bddness}
\sup_{t\in [0,T] } \E \big[ \| Y_{t}^{h,I} \|_{\bfH_{\max\{\rho,\gamma - \beta\}}}^2 \big] <\infty
. 
\end{equation}
This and~\eqref{eq:Y0_L2bddness} establish item~\eqref{it:Ysmooth}.
Next note that~\eqref{eq:Y_Lpest} ensures that for every $p\in [2,\infty)$, 
$\delta\in [0,\infty)$, $I\subseteq \H$, $h,t\in [0,T]$ it holds that 
\begin{equation}\label{eq:pre-Gronwall}
\begin{aligned}
&
  \sup_{s\in [0,t]} 
    \big(\E \big[ \| Y_{s}^{h,I} \|_{\bfH_{\delta}}^p \big]\big)^{\!\nicefrac{2}{p}}
\\ & 
 \leq 
  2 \Big(     
     \big( 
      \E \big[ 
	 \| 
	  \xi 
	 \|_{\bfH_{\delta}}^p
       \big] 
     \big)^{\!\nicefrac{1}{p}}
    +
    T
    \| F|_{\bfH_{\delta}} \|_{\Lip(\bfH_{\delta},\bfH_{\delta})}
  +
    \sqrt{\tfrac{p(p-1)T}{2}}
    \| B|_{\bfH_{\delta}} \|_{\Lip(\bfH_{\delta},L_2(U,\bfH_{\delta}))}
 \Big)^{\!2}
\\ & 
  \quad +
  2
  \Big(
    \sqrt{T} \| F|_{\bfH_{\delta}} \|_{\Lip(\bfH_{\delta},\bfH_{\delta})}
    +
    \sqrt{\tfrac{p(p-1)}{2}}
    \| B|_{\bfH_{\delta}} \|_{\Lip(\bfH_{\delta},L_2(U,\bfH_{\delta}))}
 \Big)^{\!2}
\\ & \quad
 \cdot
 \int_{0}^{t} 
    \sup_{u\in [0,s]}
      \big( \E \big[ \|
	Y^{h,I}_{u} 
      \|_{\bfH_{\delta}}^p \big] \big)^{\!\nicefrac{2}{p}}
  \,ds
  .
\end{aligned}
\end{equation}
Gronwall's inequality, \eqref{eq:Y0_Lpbddness}, and~\eqref{eq:Yh_L4bddness} hence
imply that 
\begin{equation}
\begin{aligned}
&
  \sup_{h,t\in [0,T],\, I \subseteq \H} 
    \big(\E \big[ \| Y_{t}^{h,I} \|_{\bfH_{0}}^6\big]\big)^{\!\nicefrac{1}{3}}
\\ &
  \leq 
  2\Big( 
    \big( \E \big[ \| \xi \|_{\bfH_{0}}^6 \big] \big)^{\!\nicefrac{1}{6}}
    +
      T
      \| F|_{\bfH_0} \|_{\Lip(\bfH_{0},\bfH_{0})}
      +
      \sqrt{15T}
      \| B \|_{\Lip(\bfH_{0},L_2(U,\bfH_{0}))}
  \Big)^{2}
\\ & \quad \cdot 
  \exp\! \Big(
    2T
    \Big(
      \sqrt{T}
      \| F|_{\bfH_0} \|_{\Lip(\bfH_0,\bfH_{0})}
      +
      \sqrt{15}
      \| B \|_{\Lip(\bfH_{0},L_2(U,\bfH_{0}))}
    \Big)^{2}\,
  \Big)
  < \infty.
\end{aligned}
\end{equation}
This establishes item~\eqref{it:Y_unif_L4bdd}. 
In the next step we 
observe that Gronwall's inequality,~\eqref{eq:FLip_trivial}, \eqref{eq:Y0_Lpbddness}, \eqref{eq:BLip_trivial},~\eqref{eq:Yh_L2bddness}, and~\eqref{eq:pre-Gronwall}
imply that 
\begin{equation}
\begin{aligned}
&
  \sup_{h,t\in [0,T],\, I \subseteq \H} 
    \E \big[ \| Y_{t}^{h,I} \|_{\bfH_{\rho}}^2\big]
\\ &
  \leq 
  2\Big( 
    \big( \E \big[ \| \xi \|_{\bfH_{\rho}}^2 \big] \big)^{\!\nicefrac{1}{2}}
    +
      T
      \| F|_{\bfH_{\rho}} \|_{\Lip(\bfH_{\rho},\bfH_{\rho})}
      +
      \sqrt{T}
      \| B|_{\bfH_{\rho}} \|_{\Lip(\bfH_{\rho},L_2(U,\bfH_{\rho}))}
  \Big)^{2}
\\ & \quad \cdot 
  \exp\! \Big(
    2T
    \Big(
      \sqrt{T}
      \| F|_{\bfH_{\rho}} \|_{\Lip(\bfH_\rho,\bfH_{\rho})}
      +
      \| B|_{\bfH_{\rho}} \|_{\Lip(\bfH_{\rho},L_2(U,\bfH_{\rho}))}
    \Big)^{2}\,
  \Big)
  < \infty.
\end{aligned}
\end{equation}
Combining this,~\eqref{eq:FLip_trivial},~\eqref{eq:Y_Lpest} 
(with $p=2$, $\delta = \rho$, $\theta= \max\{\rho,\gamma - \beta\}$ in the notation of~\eqref{eq:Y_Lpest}), 
and~\eqref{eq:BLip_trivial} establishes item~\eqref{it:Y_unif_L2bdd}.
The proof of Lemma~\ref{lemma:Y_unif_Lpbddness} is thus completed.
\end{proof}
\subsection{Upper bounds for the strong approximation errors}\label{ssec:upperboundsstrong}
The statement and the proof of the next result, Proposition~\ref{prop:strong_Galerkin_conv} below, 
are minor modifications  of the statement and the proof of~Jacobe de Naurois et al.~\cite[Lemma 3.3]{JacobedeNauroisJentzenWelti:2015}.
\begin{proposition}\label{prop:strong_Galerkin_conv}
Assume Setting~\ref{setting:Part2} and let
$h\in [0,T]$, $I,J\subseteq \H$. Then it holds that 
 \begin{equation}\label{eq:strong_Galerkin_conv_Y}
 \begin{aligned}
 &
 \sup_{t\in [0,T]}
  \| 
    Y_t^{h,J} - Y_t^{h,I}
  \|_{\calL^2(\P;\bfH_0)}
\\   &
  \leq 
  \sqrt{2} \exp\! \Big(\!
    \big(
      T \| \bfP_{I\cap J} F|_{\bfH_0} \|_{\Lip(\bfH_0,\bfH_0)}
      +
      \sqrt{T} \left\|\bfP_{I\cap J} B \right\|_{\Lip(\bfH_0,\bfH_0)}
    \big)^2
    \Big)
\\ & \quad 
\cdot 
  \bigg[
  \sup_{t\in [0,T] } 
  \|
    \bfP_{I\backslash J} Y_t^{h,I} - \bfP_{J\backslash I} Y_t^{h,J} 
  \|_{\calL^2(\P;\bfH_0)}
  \bigg].
 \end{aligned}
 \end{equation}
\end{proposition}
\begin{proof}[Proof of Proposition~\ref{prop:strong_Galerkin_conv}.]
Note that item~\eqref{it:Y_unif_L4bdd} in Lemma~\ref{lemma:Y_unif_Lpbddness} implies that
\begin{equation}\label{eq:strong_Galerkin_conv_proof_1}
\sup_{s\in[0,T]}\|Y^{h,I}_{s}-Y^{h,J}_{s}\|_{\mathcal L^2(\P;\bfH_0)}\leq \sup_{s\in[0,T]}\|Y^{h,I}_{s}\|_{\mathcal L^2(\P;\bfH_0)}+\sup_{s\in[0,T]}\big\|Y^{h,J}_{s}\|_{\mathcal L^2(\P;\bfH_0)}<\infty.
\end{equation}
Moreover, observe that Lemma~\ref{lemma:sg-estimate1} ensures
 that for every $t\in(0,T]$, $s\in(0,t)$ it holds that
\begin{equation}
\begin{split}
&\big\|e^{(t-\lfloor s\rfloor_h)\bf A}\bfP_{I\cap J}\big(F(Y^{h,I}_{\lfloor s\rfloor_h})-F(Y^{h,J}_{\lfloor s\rfloor_h})\big)\big\|_{\mathcal L^2(\P;\bfH_0)}\\
&\leq \|\bfP_{I\cap J}F|_{\bfH_0}\|_{\mathrm{Lip}(\bfH_0,\bfH_0)}
\bigg[\sup_{u\in[0,s]}\|Y^{h,I}_{u}-Y^{h,J}_{u}\|_{\mathcal L^2(\P;\bfH_0)}\bigg]
\end{split}
\end{equation}
and 
\begin{equation}
\begin{split}
&\big\|e^{(t-\lfloor s\rfloor_h)\bf A}\bfP_{I\cap J}\big(B(Y^{h,I}_{\lfloor s\rfloor_h})-B(Y^{h,J}_{\lfloor s\rfloor_h})\big)\big\|_{\mathcal L^2(\P;L_2(U,\bfH_0))}\\
&\leq \|\bfP_{I\cap J}B\|_{\mathrm{Lip}(\bfH_0,L_2(U,\bfH_0))}
\bigg[\sup_{u\in[0,s]}\|Y^{h,I}_{u}-Y^{h,J}_{u}\|_{\mathcal L^2(\P;\bfH_0)}\bigg].
\end{split}
\end{equation}
Combining \eqref{eq:deriv_est_B_and_F}, 
\eqref{eq:Y_mild}, \eqref{eq:strong_Galerkin_conv_proof_1}, 
and Jentzen \& Kurniawan~\cite[Corollary 3.1]{JentzenKurniawan:2015} 
(with $H=\bfH_0$, $p=2$, $\vartheta=0$, 
$\mathbf y=\|\bfP_{I\cap J}F|_{\bfH_0}\|_{\mathrm{Lip}(\bfH_0,\bfH_0)}$, 
$\mathbf z=\|\bfP_{I\cap J}B\|_{\mathrm{Lip}(\bfH_0,L_2(U,\bfH_0))}$, 
$X_t=Y^{h,I}_t$, $\bar X_t=Y^{h,J}_t$, 
$Y^t_s=e^{(t-\lfloor s\rfloor_h)\bf A}\bfP_{I\cap J}F(Y^{h,I}_{\lfloor s\rfloor_h})$, 
$\bar Y^t_s=e^{(t-\lfloor s\rfloor_h)\bf A}\bfP_{I\cap J}F(Y^{h,J}_{\lfloor s\rfloor_h})$, 
$Z^t_s=e^{(t-\lfloor s\rfloor_h)\bf A}\bfP_{I\cap J}B(Y^{h,I}_{\lfloor s\rfloor_h})$, 
$\bar Z^t_s=e^{(t-\lfloor s\rfloor_h)\bf A}\bfP_{I\cap J}B(Y^{h,J}_{\lfloor s\rfloor_h})$ 
for $s\in[0,t]$, $t\in[0,T]$ 
in the notation of \cite[Corollary 3.1]{JentzenKurniawan:2015}) 
therefore establishes that
\begin{align}
& 
\sup_{t \in [0,T]} \| Y_t^{h,I} - Y_t^{h,J} \|_{\calL^2( \P; \bfH_0 )} 
\notag
\\ &
\leq \sqrt{2}
\exp\Big(\!
    \big( 
        T 
        \|\bfP_{I\cap J}F|_{\bfH_0}\|_{ \Lip( \bfH_0, \bfH_0 ) } 
        + 
        \sqrt{T} \|\bfP_{ I \cap J } B \|_{ \Lip( \bfH_0, L_2( U, \bfH_0 ) ) } 
    \big)^2
\Big)
\notag
\\ 
& \quad 
\cdot 
\bigg[
\sup_{ t \in [0,T] } 
\Big\| 
    [Y_t^{h,I}]_{\P,\,\calB(\bfH_0)}  
    - 
    \Big( 
        \int_{0}^{t} 
            e^{ (t-\round{s}{h})\bfA } 
            \bfP_{ I \cap J } 
            F(Y_{\round{s}{h}}^{h,I}) 
        \,\mathrm{d}s 
\notag
\\ & \qquad
        +
        \int_{0}^{t} 
            e^{ (t-\round{s}{h})\bfA } 
            \bfP_{ I \cap J } 
            B(Y_{\round{s}{h}}^{h,I}) 
        \,\mathrm{d}W_s 
    \Big)
+ 
\Big( 
    \int_{0}^{t} 
        e^{ (t-\round{s}{h})\bfA } 
        \bfP_{ I \cap J } 
        F(Y_{\round{s}{h}}^{h,J}) 
    \,\mathrm{d}s 
\notag
\\ & \qquad
    + 
    \int_{0}^{t} 
        e^{ (t-\round{s}{h})\bfA } 
        \bfP_{ I \cap J } 
        B(Y_{\round{s}{h}}^{h,J}) 
    \,\mathrm{d}W_s 
\Big) 
- 
[Y_t^{h,J}]_{\P,\,\calB(\bfH_0)}  
\Big\|_{ L^2( \P; \bfH_0 ) } 
\bigg]
\notag
\\ 
& = 
\sqrt{2} \exp\Big(\!
    \big( 
        T 
        \|\bfP_{I\cap J}F|_{\bfH_0}\|_{ \Lip( \bfH_0, \bfH_0 ) } 
        + 
        \sqrt{T} \|\bfP_{ I \cap J } B \|_{ \Lip( \bfH_0, L_2( U, \bfH_0 ) ) } 
    \big)^2
\Big)
\\ 
&\quad 
\cdot 
\bigg[
\sup_{ t \in [0,T] } 
\Big\| 
    [Y_t^{h,I}]_{\P,\,\calB(\bfH_0)} 
    - 
    \bfP_J \Big( 
        [e^{ t\bfA  } \bfP_I \xi]_{\P,\,\calB(\bfH_0)}  
\notag
\\ & \qquad 
        + 
        \int_{0}^{t} 
            e^{ (t-\round{s}{h})\bfA } \bfP_{ I } F(Y_{\round{s}{h}}^{h,I}) 
        \, \mathrm{d}s  
        + 
        \int_{0}^{t} 
            e^{ (t-\round{s}{h})\bfA } 
            \bfP_{ I } B(Y_{\round{s}{h}}^{h,I}) 
        \,\mathrm{d}W_s 
    \Big) 
\notag
\\ & \qquad 
        + 
    \bfP_I \Big( 
        [e^{ t\bfA  } \bfP_J \xi]_{\P,\,\calB(\bfH_0)}  
        + 
        \int_{0}^{t} 
            e^{ (t-\round{s}{h})\bfA } 
            \bfP_{ J } F(Y_{\round{s}{h}}^{h,J}) 
        \,\mathrm{d}s
\notag
\\
& \qquad  
        + 
        \int_{0}^{t} 
            e^{ (t-\round{s}{h})\bfA } 
            \bfP_{ J } B(Y_{\round{s}{h}}^{h,J}) 
        \,\mathrm{d}W_s 
    \Big) 
    - 
    [Y_t^{h,J}]_{\P,\,\calB(\bfH_0)} 
    \Big\|_{ L^2( \P; \bfH_0 ) } 
    \bigg]
\notag
\\ 
& = 
\sqrt{2} \exp\Big(\!
    \big( 
        T 
        \|\bfP_{I\cap J}F|_{\bfH_0}\|_{ \Lip( \bfH_0, \bfH_0 ) } 
        + 
        \sqrt{T} \|\bfP_{ I \cap J } B \|_{ \Lip( \bfH_0, L_2( U, \bfH_0 ) ) } 
    \big)^2
\Big)
\notag
\\
& \quad 
\cdot 
\bigg[
\sup_{ t \in [0,T] } 
\| 
    \bfP_{ I \backslash J } Y_t^{h,I} 
    - 
    \bfP_{ J \backslash I } Y_t^{h,J} 
\|_{ \calL^2( \P; \bfH_0 ) }
\bigg]
.
\notag
\end{align} 
The proof of Proposition~\ref{prop:strong_Galerkin_conv} is thus completed.
\end{proof}
\noindent{}The proof of the next result, Corollary~\ref{cor:strong_Galerkin_conv} below, 
is a minor modification of the third step in the proof of 
Jacobe de Naurois et al.~\cite[Theorem 3.6]{JacobedeNauroisJentzenWelti:2021}.
\begin{corollary}\label{cor:strong_Galerkin_conv}
Assume Setting~\ref{setting:Part2}, let $h\in [0,T]$, and let $I_n \subseteq \H$,
 $n\in \N$, satisfy that $\cup_{n\in \N} I_n = \H$ and $[\forall\, n\in \N \colon I_n \subseteq I_{n+1}]$. Then  
\begin{equation}\label{eq:strong_Galerkin_conv}
 \limsup_{n\rightarrow \infty}
  \sup_{t\in [0,T]}
    \E \big[ \|
      Y_t^{h,\H} - Y_t^{h,I_n}
    \|_{\bfH_0}^2
  \big] 
  =
  0.
\end{equation}
\end{corollary}
\begin{proof}[Proof of Corollary~\ref{cor:strong_Galerkin_conv}.]
Observe that Proposition~\ref{prop:strong_Galerkin_conv}, \eqref{eq:Y_mild}, Lemma~\ref{lemma:sg-estimate1}, 
Minkowski's integral inequality, 
and It\^o's isometry
imply that for every $n\in \N$ it holds that
\begin{equation} \label{eq:weakrates,13}
\begin{aligned}
&\sup_{t \in [0,T]} 
\| Y_t^{h,\H} - Y_t^{h,I_n} \|_{\calL^2( \P; \bfH_0 ) } 
\\
& 
\leq 
\sqrt{2} \exp\Big(\!
    \big( 
        T 
        \|\bfP_{ I_n } F|_{\bfH_0}\|_{ \Lip( \bfH_0, \bfH_0 ) } 
        + 
        \sqrt{T} \|\bfP_{ I_n } B \|_{ \Lip( \bfH_0, L_2( U, \bfH_0 ) ) } 
    \big)^2
\Big)
\\ & \quad 
\cdot
\bigg[
\sup_{ t \in [0,T] } 
\big\| 
    \bfP_{ \H \backslash I_n } Y_t^{h,\H} 
\big\|_{ \calL^2( \P; \bfH_0 ) } 
\bigg]
\\
& \leq 
\sqrt{2}\exp\Big(\!
    \big( 
        T 
        \|\bfP_{ I_n } F|_{\bfH_0}\|_{ \Lip( \bfH_0, \bfH_0 ) } 
        + 
        \sqrt{T} \|\bfP_{ I_n } B \|_{ \Lip( \bfH_0, L_2( U, \bfH_0 ) ) } 
    \big)^2
\Big)
\\
& \quad 
\cdot \bigg[ 
    \| \bfP_{ \H \backslash I_n } \xi \|_{ \calL^2( \P; \bfH_0 ) } + 
    \int_{0}^{T} 
            \big\| 
                \bfP_{ \H \backslash I_n } F( Y_{\round{s}{h}}^{h,\H} ) 
            \big\|_{ \calL^2( \P; \bfH_0 ) } \,ds
\\
&\qquad\ 
+ \Big(
    \int_{0}^{T} 
        \big\| 
            \bfP_{ \H \backslash I_n } B( Y_{\round{s}{h}}^{h,\H} ) 
        \big\|_{ \calL^2( \P; L_2( U, \bfH_0 ) ) }^2 
    \,ds 
\Big)^{ \nicefrac{1}{2} } \bigg].
\end{aligned}
\end{equation}
This, 
item~\eqref{it:Lip_est_standard2} in Lemma~\ref{lemma:Lip_est_standard},
item~\eqref{it:Y_unif_L4bdd}
 in Lemma~\ref{lemma:Y_unif_Lpbddness}, 
and Lebesgue's theorem of dominated convergence ensure that 
\begin{equation} \label{eq:weakrates,17}
\limsup_{ n \to \infty } 
\sup_{t \in [0,T]} 
    \| 
        Y_t^{h,\H} - Y_t^{h,I_n} 
    \|_{\calL^2( \P; \bfH_0 ) } = 0.
\end{equation}
The proof of Corollary~\ref{cor:strong_Galerkin_conv} is thus completed.
\end{proof}
\subsection{Upper bounds for the weak approximation errors}
\label{ssec:upperboundsweak}
Theorem~\ref{prop:fin_dim_weak_error}  below constitutes our main result as it establishes suitable upper bounds for weak approximation errors related to temporal discretisations of hyperbolic SPDEs with multiplicative noise. 
In particular, it allows us to derive 
weak convergence rates for temporal discretisations of hyperbolic SPDEs in Corollary~\ref{cor:weak_conv}. 
The proof of Theorem~\ref{prop:fin_dim_weak_error} 
relies 
on Proposition~\ref{thm:mildIto} below,  
which states a special case of the mild It\^o formula in Da~Prato et al.~\cite[Corollary 1]{DaPratoJentzenRoeckner:2012}. 
\begin{proposition}\label{thm:mildIto}
Assume Setting~\ref{setting:Part1}, 
let
$ 
  \big( 
    H, 
    \left< \cdot , \cdot \right>_{ H }, 
    \left\| \cdot \right\|_{ H }
  \big) 
$ 
be a separable $\mathbbm{R}$-Hilbert space, 
let $S\colon[0,\infty)\rightarrow L(H)$ 
be a strongly continuous semigroup,  
let $\phi\in C^{2}(H,\R)$,  $\tau_0,\tau_1\in [0,T]$ satisfy $\tau_0<\tau_1$,
let 
$ 
  X \colon [\tau_0,\tau_1] \times \Omega
  \rightarrow H
$, 
$ 
  Y \colon [\tau_0,\tau_1] 
  \times \Omega
  \rightarrow H
$, 
$ 
  Z \colon [\tau_0,\tau_1] 
  \times \Omega
  \rightarrow L_2( U, H )
$
be $(\mathbbm F_t)_{t\in[\tau_0,\tau_1]}$-predictable 
stochastic processes which satisfy for every $t\in [\tau_0,\tau_1] $ that
$
    \P\big[
    \int_{ \tau_0 }^{ t }
    \big(\left\| S_{ t-s } Y_s
    \right\|_{ H }
    +
    \left\| S_{t-s} Z_s
    \right\|_{ L_2( U, H) }^2\big)
    \,\mathrm ds < \infty\big]=1
$  
and 
\begin{equation}
\label{eq:mildito}
 [ X_t ]_{\P,\mathcal{B}(H)} 
= 
  [ S_{ t - \tau_0 } \,
  X_{ \tau_0 } ]_{\P,\mathcal{B}(H)} 
  +
  \int_{ \tau_0 }^t
    S_{ t-s } \,
    Y_s
  \,\mathrm ds
  +
  \int_{ \tau_0 }^t
    S_{t-s} \,
    Z_s
  \,\mathrm dW_s.
\end{equation}
Then it holds for every $t\in [\tau_0,\tau_1]$ that
\begin{equation}
\label{eq:well1NON}
  \mathbbm{P}\!\left[
    \int_{ \tau_0 }^t \big(
    |
      \phi'( S_{ t-s } X_s )
      S_{ t-s } Y_s
    |
    +
    \|
      \phi'( S_{ t-s } X_s )
      S_{ t-s } Z_s
    \|_{ L_2(U, \R ) }^2 \big)
    \,\mathrm ds
    < \infty 
  \right] = 1,
\end{equation}
\begin{equation}
\label{eq:well2NON}
  \mathbbm{P}\!\left[
    \int_{ \tau_0 }^t \big(
    \|
      \phi''( S_{ t-s } X_s )
    \|_{ L^{(2)}( H, \R ) } \,
    \|
      S_{ t-s } Z_s
    \|_{ L_2(U, H ) }^2 \big)
    \,\mathrm ds
    < \infty 
  \right] = 1,
\end{equation}
and 
\begin{equation}\label{eq:itoformel_startNON}
\begin{aligned}
  [ \varphi( X_t ) ]_{\P,\mathcal{B}(\R)}
&=
  [\varphi( 
    S_{ t - \tau_0 } 
    X_{ \tau_0 } 
  )]_{\P,\mathcal{B}(\R)}
  +
  \int_{ \tau_0 }^t
  \Big(\phi'(  
    S_{ t-s} 
    X_s 
  ) \,
  S_{ t-s } \, 
  Y_s 
\\&\quad
    +
  \tfrac{1}{2}
  \sum_{ u\in \mathbbm{U} }
  \varphi''( 
    S_{ t-s } 
    X_s 
  )
  \left(
    S_{ t-s } 
    Z_s u,
    S_{ t-s } 
    Z_s u
  \right) \Big)\,\mathrm ds
  +
  \int_{ \tau_0 }^t
  \varphi'( 
    S_{ t-s } 
    X_s 
  ) \,
  S_{ t-s } \, 
  Z_s \,\mathrm dW_s.
\end{aligned}
\end{equation}
\end{proposition}
The key idea in the proof of the weak error estimate in Theorem~\ref{prop:fin_dim_weak_error} below is 
a specific 
decomposition 
of the error into terms that can be estimated using the mild It\^o formula in Proposition~\ref{thm:mildIto};  
see \eqref{eq:fin_dim_weak_error}--\eqref{eq:split_noise_term} below.
\begin{theorem}\label{prop:fin_dim_weak_error}
Assume Setting~\ref{setting:Part2}, let $h\in (0,T]$, for every finite $I\subseteq \H$ and every $x\in \bfP_I(\bfH_0)$
let 
$ X^{I,x}\colon [0,T] \times \Omega \rightarrow \bfP_I (\bfH_0 )$
be an $(\mathbbm{F}_t)_{t\in[0,T]}$-predictable stochastic process which satisfies for every $t\in [0,T]$
that 
$\sup_{s\in[0,T]}\E\big[\| X^{I,x}_s \|_{\bfH_0}^2\big] < \infty$  
and 
\begin{equation}\label{eq:fin_dim_weak_error_1}
\begin{aligned}
 \big[X^{I,x}_t\big]_{\P,\,\calB(\bfP_I(\bfH_0))} 
& =
 \left[e^{t\bfA } x \right]_{\P,\,\calB(\bfP_I(\bfH_0))}  
 +
 \int_{0}^{t} e^{ (t-s)\bfA} \bfP_I F \big( X^{I,x}_{s} \big) \,\mathrm{d}s
\\ & \quad 
 +
 \int_{0}^{t} e^{ (t-s)\bfA } \bfP_I B\big( X^{I,x}_{s} \big) \,\mathrm{d}W_s,
\end{aligned}
\end{equation}
and for every finite $I\subseteq \H$ let $v^I\colon [0,T]\times \bfP_I (\bfH_0) \rightarrow \R$ be the function 
which satisfies 
for every $t\in [0,T], x\in  \bfP_I (\bfH_0)$ that 
$v^I(t,x) = \E[\phi(X^{I,x}_{T-t})]$. 
Then 
\begin{enumerate}[(i)]
 \item \label{item:v_diff} it holds for every $t\in [0,T]$ and every finite $I \subseteq \H$ that
 $(\bfP_I(\bfH_0)\ni x \mapsto v^{I}(t,x) \in \R) \in C^{4}(\bfP_I ( \bfH_0), \R)$,
 \item \label{item:v_unif_bddness}
 it holds that
\begin{equation}
  \sup_{\stackrel{I\subseteq \H,}{ I \textnormal{ is finite}}}
  \max_{k\in \{1,2,3,4\}} 
  \sup_{t\in [0,T]} 
  \sup_{x\in \bfP_I(\bfH_0)} 
    \| (\tfrac{\partial^k}{\partial x^k} v^I)(t,x)\|_{L^{(k)}(\bfP_I(\bfH_0),\R)}
  < \infty,
\end{equation}
 and
\item\label{it:fin_dim_weak_error} it holds for every finite $I\subseteq \H$ that 
\begin{equation}\label{eq:fin_dim_weak_error_prop}
\begin{aligned}
&
\big|
  \E \big[ \phi\big( Y_T^{0,I} \big) \big]
  -
  \E \big[ \phi\big( Y_T^{h,I} \big) \big]
\big|
 \leq
 6 
  \max\{T,T^{2-2(\gamma-\beta)}\} \,h^{2(\gamma-\beta)} 
\\ & \cdot 
 \bigg[ \max_{k\in \{1,2,3,4\}} 
  \sup_{t\in [0,T]} \sup_{x \in \bfP_I(\bfH_0)} 
    \| (\tfrac{\partial^k}{\partial x^k} v^I)(t,x)\|_{L^{(k)}(\bfP_I(\bfH_0),\R)}
 \bigg]
%
%
\\ & 
\cdot 
\bigg[ 
      \| F|_{\bfH_{\rho}} \|_{\Lip(\bfH_{\rho},\bfH_{2(\gamma-\beta)})} 
      +
      3\,\mathfrak{m}^4 
      +
      \| F \|_{\Lip(\bfH_{\beta-\gamma},\bfH_0)} 
\\ & \quad
      +
      2\Big[\textstyle\sum\limits_{h\in\H}\displaystyle|\lambda_h|^{-\beta}\Big]
      \| B|_{\bfH_{\rho}} \|_{\Lip(\bfH_{\rho},L(U,\bfH_{\gamma}))}^2
      + 
      \mathfrak{l}\,\mathfrak{c}
\bigg]
\\ & \cdot 
	\Big[\max\big\{\|\bfLambda^{\rho-\max\{\rho,\gamma-\beta\}}\|_{L(\bfH_0)}^2,\|\bfLambda^{\gamma-\beta-\max\{\rho,\gamma-\beta\}}\|_{L(\bfH_0)}^2\big\}\Big] 
\\ & \cdot 
   \bigg[
      1
      +
      \sup_{t\in [0,T]} 
	\E \Big[ 
	  \| Y_t^{h,I} \|_{\bfH_{\max\{\rho,\gamma-\beta\}}}^2 
	  +
	  \| Y_t^{h,I} \|_{\bfH_{0}}^4 
	\Big]
   \bigg]
   < \infty.
\end{aligned}
\end{equation}
\end{enumerate}
\end{theorem}
\begin{proof}[Proof of Theorem~\ref{prop:fin_dim_weak_error}.]
Throughout this proof let $\delta \colon [0,\infty) \rightarrow [0,h)$ be
the function which satisfies for every $x\in [0,\infty)$ that 
$\delta(x) = x- \lfloor x \rfloor_{h}$, for every $p\in [1,\infty)$ 
and every $\R$-Hilbert space 
$ ( W , \langle \cdot , \cdot \rangle_{ W } , \left\| \cdot\right\|_{ W } ) $
let $(L_p(U,W), \left\| \cdot \right\|_{L_p(U,W)})$ 
be the $\R$-Banach space of Schatten-$p$ operators from $U$ to $W$ 
and  let $(L_p(W), \left\| \cdot \right\|_{L_p(W)})$ 
be the $\R$-Banach space of Schatten-$p$ operators from $W$ to $W$, 
for every finite $I\subseteq \H$ 
let $v_{1,0}^I\colon[0,T]\times \bfP_I(\bfH_0)\to \R$ be the function 
which satisfies for every $t \in [0,T]$, $x\in \bfP_I(\bfH_0)$
that 
$
  v^I_{1,0}(t,x)
  = 
  (\tfrac{\partial}{\partial t} 
  v^I)(t,x),
$ 
for every finite $I\subseteq \H$ and every $i\in \{1,2,3,4\}$ 
let $v_{0,i}^I \colon[0,T]\times \bfP_I(\bfH_0)\to L^{(i)}(\bfP_I(\bfH_0),\R)$ be the function 
which satisfies for every $t \in [0,T]$, $x_0,x_1,\ldots,x_{i}\in \bfP_I(\bfH_0)$
that 
\begin{equation}
  v^I_{0,i}(t,x_0)(x_1\ldots,x_{i})
  = 
  \big(\tfrac{\partial^i}{\partial x_0^i} 
  v^I(t,x_0)\big)(x_1,\ldots,x_i),
\end{equation}
for every finite $I\subseteq \H$ let
$\phi^I, \psi^I\colon [0,T]\times \bfP_I ( \bfH_0) \rightarrow \R$
be the functions which satisfy for every $t\in [0,T]$, $x\in \bfP_I ( \bfH_0)$ that
$\phi^I(t,x) = v^I_{0,1}(t,x)(\bfP_I F(x))$ and 
$\psi^I(t,x) = \sum_{u\in \U} v^I_{0,2}(t,x)$\linebreak$(\bfP_I B(x)u,\bfP_I B(x)u)$, 
and for every finite $I\subseteq \H$ and every $i\in \{1,2\}$ let
$\phi^I_{0,i}, \psi^I_{0,i} \colon [0,T] \times \bfP_I ( \bfH_0) \rightarrow L^{(i)}(\bfH_0,\R)$
be the functions which satisfy for every $t\in [0,T]$, $x\in  \bfP_I ( \bfH_0)$ that 
$\phi^I_{0,i}(t,x) = (\tfrac{\partial^i}{\partial x^i} \phi^I)(t,x)$ and 
$\psi^I_{0,i}(t,x) = (\tfrac{\partial^i}{\partial x^i} \psi^I)(t,x)$. 
Observe that for every finite $I\subseteq\H$ and every $t\in[0,\infty)$ it holds that $\bfA\bfP_I|_{\bfH_0}\in L(\bfH_0)$ and
\begin{equation}\label{eq:sg_projection}
e^{t(\bfA\bfP_I|_{\bfH_0})}
=
e^{t\bfA}\bfP_I|_{\bfH_0}+\bfP_{\H\setminus I}|_{\bfH_0}.
\end{equation}
Lemma~\ref{lemma:sg-estimate1} and \eqref{eq:deriv_est_B_and_F} therefore
ensure that
\begin{multline}\label{eq:deriv_est_B_and_F_projected}
  \sup_{\stackrel{I\subseteq \H,}{ I \textnormal{ is finite}}} 
  \bigg[
    \sup_{t\in [0,T]} 
    \|
      e^{t(\bfA\bfP_I|_{\bfH_0})}
    \big\|_{L(\bfH_0)}
    +
    \| 
      \bfP_IF|_{\bfH_0}
    \|_{C^4_{\mathrm{b}}(\bfH_0, \bfH_0)}+
    \|
      \bfP_IB
    \|_{C^4_{\mathrm{b}}(\bfH_0, L_2(U, \bfH_0))}
  \bigg]
 \\ 
  <
  \infty.
\end{multline}
Combining this, \eqref{eq:fin_dim_weak_error_1}, \eqref{eq:sg_projection}, 
item~\eqref{it:kolmogorov_diffble} in Lemma~\ref{lem:kolmogorov} 
(with $H=\bfH_0$, $k=4$, $A=\bfA\bfP_I|_{\bfH_0}$, $F=\bfP_IF|_{\bfH_0}$, $B=\bfP_IB$, $\varphi=\varphi$, $v^{A,F,B,\varphi}(t,x)=v^I(t,x)$ 
for $t\in[0,T]$, $x\in\bfP_I(\bfH_0)$, $I\in\{J\subseteq\H\colon J\text{ is a finite set}\}$  
in the notation of item~\eqref{it:kolmogorov_diffble} in Lemma~\ref{lem:kolmogorov}),  
and item \eqref{it:kolmogorov_deriv_est} in Lemma~\ref{lem:kolmogorov} 
(with $H=\bfH_0$ in the notation of item~\eqref{it:kolmogorov_deriv_est} in Lemma~\ref{lem:kolmogorov})
establishes items~\eqref{item:v_diff} and \eqref{item:v_unif_bddness}. 
It thus remains to prove item~\eqref{it:fin_dim_weak_error}. 
For this, observe that 
for every finite $I\subseteq\H$ it holds that $\E[\varphi(Y^{h,I}_T)]=\E[v^I(T,Y^{h,I}_T)]$ and 
\begin{equation}\label{eq:weak_error_conditional_exp}
\begin{split}
\E\big[\varphi\big(Y^{0,I}_T\big)\big]
&=
\E\big[\E\big(\varphi\big(Y^{0,I}_T\big)\big|\mathbbm{F}_0\big)\big]
=\E\big[v^I\big(0,\bfP_I\xi\big)\big]
=\E\big[v^I\big(0,Y^{h,I}_0\big)\big].
\end{split}
\end{equation}
Moreover, note that for every finite $I\subseteq\H$ it holds that $(e^{t\bfA}|_{\bfP_I(\bfH_0)})_{t\in[0,\infty)}\subseteq L(\bfP_I(\bfH_0))$ is a strongly continuous semigroup with generator $\bfA|_{\bfP_I(\bfH_0)}\in L(\bfP_I(\bfH_0))$. 
This and \eqref{eq:Y_mild} 
imply that for every finite $I\subseteq \H$ and every $t\in [0,T]$ it holds that 
\color{black}
\begin{equation}
\begin{aligned}\label{eq:Y_strong}
\big[ Y_t^{h,I} \big]_{\P,\,\calB(\bfP_I(\bfH_0))}
& = 
 \big[\bfP_I \xi\big]_{\P,\,\calB(\bfP_I(\bfH_0))}
 + 
 \int_{0}^{t} \big( 
  \bfA  Y_s^{h,I} + e^{\delta(s)\bfA} \bfP_I F\big( Y_{\round{s}{h}}^{h,I} \big) 
 \big) \,\mathrm{d}s 
\\ &
\quad  +
 \int_{0}^{t} e^{\delta(s)\bfA} \bfP_I B\big( Y_{\round{s}{h}}^{h,I} \big) \,\mathrm{d}W_s 
 .
\end{aligned}
\end{equation}
In addition, observe that item~\eqref{it:kolmogorov_deriv} in Lemma~\ref{lem:kolmogorov}
ensures that for every finite $I\subseteq\H$ and every $t\in[0,T]$, $x\in\bfP_I(\bfH_0)$ it holds that 
\begin{equation}\label{eq:vIC2}
v^I\in C^{1,2}\big([0,T]\times\bfP_I(\bfH_0),\R\big)
\end{equation}
and
\begin{equation}\label{eq:vIKolmogorov}
v^I_{1,0}(t,x)
=
-v^I_{0,1}(t,x)\big(\bfA x+\bfP_I F(x)\big)
-\tfrac12\sum_{u\in\U}v^I_{0,2}(t,x)\big(\bfP_I B(x)u,\bfP_I B(x)u\big).
\end{equation}
Combining item~\eqref{item:v_unif_bddness}, the fact that for every finite $I\subseteq\H$ it holds that $\bfA|_{\bfP_I(\bfH_0)}\in L(\bfP_I(\bfH_0))$, item~\eqref{it:Lip_est_standard2} in
Lemma~\ref{lemma:Lip_est_standard}, Lemma~\ref{lemma:sg-estimate1}, and item~\eqref{it:Y_unif_L4bdd} in  
Lemma~\ref{lemma:Y_unif_Lpbddness} therefore proves that for every finite $I\subseteq\H$ it holds that
\begin{equation}\label{eq:weak_error_integrability}
\begin{split}
&\E \bigg[
   \int_{0}^{T} 
    \big|v^I_{1,0} \big(t,Y_t^{h,I}\big)\big| 
   \,dt 
   +
   \int_{0}^{T}
    \big|v^I_{0,1}\big(t,Y_t^{h,I}\big) \big(\bfA Y_t^{h,I} + e^{\delta(t)\bfA} \bfP_I F\big(Y_{\round{t}{h}}^{h,I} \big)\big) \big|
   \,dt
\\ 
&\quad
   +
   \tfrac12 \sum_{u\in \U}
    \int_{0}^{T} 
     \big| v^I_{0,2}\big(t,Y_t^{h,I}\big) \big(
	e^{\delta(t) \bfA } \bfP_I B\big(Y_{\round{t}{h}}^{h,I}\big) u
	,
	e^{\delta(t) \bfA } \bfP_I B\big(Y_{\round{t}{h}}^{h,I}\big) u
      \big)\big|
    \,dt
   \bigg]
   <\infty.
\end{split}
\end{equation}
Furthermore, note 
\color{black}
that It\^o's isometry, item~\eqref{item:v_unif_bddness}, item~\eqref{it:Lip_est_standard2} in
Lemma~\ref{lemma:Lip_est_standard}, Lemma~\ref{lemma:sg-estimate1}, and item~\eqref{it:Y_unif_L4bdd} in  
Lemma~\ref{lemma:Y_unif_Lpbddness} \color{black}
imply that for every finite $I\subseteq \H$ it holds that 
\begin{equation}
\E \!\left[ \Big| 
  \int_{0}^{T} 
    v^I_{0,1}\big(t,Y_t^{h,I}\big) 
    e^{\delta(t)\bfA} 
    \bfP_I B\big(Y_{\round{t}{h}}^{h,I} \big)
  \,\mathrm{d} W_t
\Big|^2\right]
<
\infty.
\end{equation}
Hence, we obtain that for every finite $I\subseteq \H$ it holds that 
\begin{equation}\label{eq:stoch_int_vanish_1}
 \E \!\left[
  \int_{0}^{T} 
    v^I_{0,1}\big(t,Y_t^{h,I}\big) 
    e^{\delta(t)\bfA} 
    \bfP_I B\big(Y_{\round{t}{h}}^{h,I} \big)
  \,\mathrm{d} W_t
  \right]
  = 
  0.
\end{equation}
The It\^o formula, \eqref{eq:weak_error_conditional_exp}--\eqref{eq:vIC2}, and \eqref{eq:weak_error_integrability} 
therefore imply that for every finite $I\subseteq \H$ it holds that 
\begin{equation}\label{eq:weak_error_step1}
\begin{aligned}
&
\big| 
  \E \big[
    \phi\big(Y_T^{h,I}\big) 
    -
    \phi\big( Y_T^{0,I} \big)
  \big]
\big|
\\ 
& 
=
\big| 
  \E \big[
    v^I\big(T,Y_T^{h,I}\big) 
    -
    v^I\big(0,Y_0^{h,I}\big)
  \big]
\big|
\\
& 
=
\bigg|
  \E \bigg[
   \int_{0}^{T} 
    v^I_{1,0} \big(t,Y_t^{h,I}\big) 
   \,\mathrm{d}t 
   +
   \int_{0}^{T}
    v^I_{0,1}\big(t,Y_t^{h,I}\big) \big(\bfA Y_t^{h,I} + e^{\delta(t)\bfA} \bfP_I F\big(Y_{\round{t}{h}}^{h,I} \big)\big) 
   \,\mathrm{d}t
\\ 
& \qquad 
   +
   \tfrac12 \sum_{u\in \U}
    \int_{0}^{T} 
      v^I_{0,2}\big(t,Y_t^{h,I}\big) \big(
	e^{\delta(t) \bfA } \bfP_I B\big(Y_{\round{t}{h}}^{h,I}\big) u
	,
	e^{\delta(t) \bfA } \bfP_I B\big(Y_{\round{t}{h}}^{h,I}\big) u
      \big)
    \,\mathrm{d}t
   \bigg]
\bigg|
\end{aligned}
\end{equation}
Combining this and 
\eqref{eq:vIKolmogorov} 
demonstrates that for every finite $I\subseteq \H$ it holds that 
\begin{equation}\begin{aligned}\label{eq:fin_dim_weak_error}
&
\big| 
  \E \big[
    \phi\big(Y_T^{h,I}\big) 
    -
    \phi\big( Y_T^{0,I} \big)
  \big]
\big|
\\ 
& 
=
\bigg|
  \E \bigg[ 
   \int_{0}^{T}
    v^I_{0,1}\big(t,Y_t^{h,I}\big) \big( 
      e^{\delta(t)\bfA} \bfP_I F\big(Y_{\round{t}{h}}^{h,I} \big) - \bfP_I F\big(Y_{t}^{h,I}\big)
    \big) 
   \,\mathrm{d}t
\\ & 
  \quad 
    + \tfrac12 
    \sum_{u\in \U}
      \int_{0}^{T} 
	v^I_{0,2}\big(t,Y_t^{h,I})
	\big(
	  e^{\delta(t)\bfA } \bfP_I B(Y_{\round{t}{h}}^{h,I} ) u
	  ,
	  e^{\delta(t)\bfA } \bfP_I B(Y_{\round{t}{h}}^{h,I} ) u
	\big)
      \,\mathrm{d}t
\\ & \quad 
    - \tfrac12  
    \sum_{u\in \U}
      \int_{0}^{T} 
	v^I_{0,2}\big(t,Y_t^{h,I}\big)
	\big(
	  \bfP_I B\big(Y_{t}^{h,I} \big) u
	  ,
	  \bfP_I B\big(Y_{t}^{h,I} \big) u
	\big)
      \,\mathrm{d}t
  \bigg]
\bigg|.
\end{aligned}
\end{equation}
Moreover, observe that the triangle inequality ensures that for every finite $I\subseteq \H$ it holds that
\begin{align}\label{eq:split_det_term}
&\notag
  \left|
    \E\! \left[
      \int_{0}^{T}
	v^I_{0,1}\big(t,Y_t^{h,I}\big) 
	\big(e^{\delta(t)\bfA} \bfP_I F\big( Y_{\round{t}{h}}^{h,I} \big) - \bfP_I F\big( Y_{t}^{h,I} \big)\big) 
      \,\mathrm{d}t
    \right]
  \right|
\\ & \leq 
\notag
\left|
  \E\! \left[
    \int_{0}^{T}
      v^I_{0,1}\big(t,Y_t^{h,I}\big)
      \big(
	\big( e^{\delta(t)\bfA} - \id_{\mathbf{H}_0} \big) 
	\bfP_I F\big(Y_{\round{t}{h}}^{h,I}\big) 
      \big)
    \,\mathrm{d}t
  \right]
\right|
\\ & \quad 
  +
  \left|
    \E\! \left[
	\int_{0}^{T} 
 	  \left(
	    v^I_{0,1}\big(t,Y_t^{h,I}\big) \big( \bfP_I F\big(Y_{\round{t}{h}}^{h,I}\big) \big)
	    -
	    v^I_{0,1}\big(t, e^{\delta(t) \bfA } Y_{\round{t}{h}}^{h,I} \big) 
	    \big( \bfP_I F\big( Y_{\round{t}{h}}^{h,I} \big) \big)
 	  \right)
	\,\mathrm{d}t
    \right]
  \right|
\\ & \notag\quad 
  +
  \left|
    \E\! \left[
	\int_{0}^{T} 
	    v^I_{0,1}\big(t, e^{\delta(t) \bfA } Y_{\round{t}{h}}^{h,I}\big)
	    \big(
	      \bfP_I F\big( Y_{\round{t}{h}}^{h,I} \big)
	      -
	      \bfP_I F\big(e^{\delta(t) \bfA } Y_{\round{t}{h}}^{h,I} \big)
	    \big)
	\,\mathrm{d}t
    \right]
  \right|
\\ & \notag\quad 
  +
  \left|
    \E\! \left[
	\int_{0}^{T}
 	  \left( 
	    v^I_{0,1}\big(t, e^{\delta(t) \bfA } Y_{\round{t}{h}}^{h,I}\big)
	    \big( \bfP_I F\big(e^{\delta(t) \bfA } Y_{\round{t}{h}}^{h,I} \big) \big)
	    -
	    v^I_{0,1}\big(t,Y_t^{h,I}\big)
	    \big( \bfP_I F\big(Y_t^{h,I}\big) \big)
 	  \right)
	\,\mathrm{d}t
    \right]
  \right|.
 \end{align}
Furthermore, note that the triangle inequality implies that for every finite $I\subseteq \H$ it holds that
\begin{align}\label{eq:split_noise_term}
& \notag
  \Bigg|
    \E \Bigg[
      \sum_{u\in \U}
	\int_{0}^{T} 
	  v^I_{0,2}\big(t,Y_t^{h,I}\big)\big(
	    e^{\delta(t)\bfA } \bfP_I B\big(Y_{\round{t}{h}}^{h,I} \big) u
	    ,
	    e^{\delta(t)\bfA } \bfP_I B\big(Y_{\round{t}{h}}^{h,I} \big) u
	  \big)
	\,\mathrm{d}t
\\ & \quad \notag 
	  -
      \sum_{u\in \U}
	\int_{0}^{T} 
	  v^I_{0,2}\big(t,Y_t^{h,I}\big)\big(
	     \bfP_I B\big(Y_{t}^{h,I} \big) u
	    ,
	    \bfP_I B\big(Y_{t}^{h,I} \big) u
	  \big)
	\,\mathrm{d}t
    \Bigg]
  \Bigg|
\\ & \leq \notag
\Bigg|
  \E\! \Bigg[
    \sum_{u\in \U}
      \int_{0}^{T} 
	v^I_{0,2}\big(t,Y_t^{h,I}\big)
	\\ & \qquad \quad \notag
	\big(
	  \big( e^{\delta(t)\bfA } - \id_{\mathbf{H}_0}\big) \bfP_I B\big(Y_{\round{t}{h}}^{h,I} \big) u
	  ,
	  \big( e^{\delta(t)\bfA } + \id_{\mathbf{H}_0}\big) \bfP_I B\big(Y_{\round{t}{h}}^{h,I} \big) u
	\big)
      \,\mathrm{d}t
  \Bigg]
\Bigg|
\\ & \quad 
  +
  \Bigg|
    \E \Bigg[
      \sum_{u\in \U}
	\int_{0}^{T} 
	  v^I_{0,2}\big(t,Y_t^{h,I}\big)\big(
	    \bfP_I B\big(Y_{\round{t}{h}}^{h,I} \big) u
	    ,
	    \bfP_I B\big(Y_{\round{t}{h}}^{h,I} \big) u
	  \big)
	\,\mathrm{d}t
\\ & \qquad \notag 
	  -
      \sum_{u\in \U}
	\int_{0}^{T} 
	  v^I_{0,2}\big(t, e^{\delta(t) \bfA } Y_{\round{t}{h}}^{h,I}\big) \big(
	    \bfP_I B\big( Y_{\round{t}{h}}^{h,I} \big) u
	    ,
	    \bfP_I B\big( Y_{\round{t}{h}}^{h,I} \big) u
	  \big)
	\,\mathrm{d}t
    \Bigg]
  \Bigg|
\\ & \quad \notag 
  +
  \Bigg|
    \E \Bigg[
      \sum_{u\in \U}
	\int_{0}^{T} 
	  v^I_{0,2}\big(t,e^{\delta(t) \bfA } Y_{\round{t}{h}}^{h,I}\big)
\\ & \qquad \quad \notag 
	  \big(
	    \big[
	      \bfP_I B\big( Y_{\round{t}{h}}^{h,I} \big)
	      -
	      \bfP_I B\big( e^{\delta(t)\bfA } Y_{\round{t}{h}}^{h,I} \big) 
	    \big]
	    u
	    ,
	    \big[
	      \bfP_I B\big( Y_{\round{t}{h}}^{h,I} \big)
	      +
	      \bfP_I B\big( e^{\delta(t)\bfA } Y_{\lfloor t \rfloor_h}^{h,I}\big) 
	    \big]
	    u
	  \big)
	\,\mathrm{d}t
    \Bigg]
  \Bigg|
\\ & \quad \notag 
  +
  \Bigg|
    \E \Bigg[
      \sum_{u\in \U}
	\int_{0}^{T} 
	  v^I_{0,2}\big(t, e^{\delta(t) \bfA } Y_{\round{t}{h}}^{h,I}\big)
	  \big(
	    \bfP_I B\big( e^{\delta(t) \bfA } Y_{\round{t}{h}}^{h,I} \big) u
	    ,
	    \bfP_I B\big(e^{\delta(t) \bfA } Y_{\round{t}{h}}^{h,I} \big) u
	  \big)
	\,\mathrm{d}t
\\ & \quad \qquad \notag 
	  -
      \sum_{u\in \U}
	\int_{0}^{T} 
	  v^I_{0,2}\big( t,Y_{t}^{h,I}\big)\big(
	    \bfP_I B\big(Y_{t}^{h,I} \big) u
	    ,
	    \bfP_I B\big(Y_{t}^{h,I} \big) u
	  \big)
	\,\mathrm{d}t
    \Bigg]
  \Bigg|. \notag
 \end{align}
%
In the next step we estimate the first term on the right-hand side of~\eqref{eq:split_det_term}.
Note that item~\eqref{it:Lip_est_standard1} in Lemma~\ref{lemma:Lip_est_standard} and Lemma~\ref{lemma:sg-estimate2} imply that for every finite $I\subseteq \H$ it holds that
\begin{equation}
\begin{split} 
&
\left|
  \E\! \left[
    \int_{0}^{T}
      v^I_{0,1}\big(t,Y_t^{h,I}\big) 
      \big(
	\big( e^{\delta(t)\bfA} - \id_{\mathbf{H}_0} \big)
	\bfP_I F\big(Y_{\round{t}{h}}^{h,I}\big) 
      \big)
    \,\mathrm{d}t
  \right]
\right|
\\ &
 \leq 
 T
 \bigg[\sup_{t\in [0,T],\, x\in \bfP_I (\bfH_0) } 
 \!\Big(\| v^I_{0,1}(t,x) \|_{L(\bfP_I(\bfH_0),\R)}
\big\| \bfLambda^{2(\beta-\gamma)} \big( e^{\delta(t)\bfA} - \id_{\mathbf{H}_0} \big) \big\|_{L(\bfH_0)}\Big)\bigg]
\\ & \quad \cdot 
 \| F|_{\bfH_{\rho}} \|_{\Lip(\bfH_{\rho},\bfH_{2(\gamma-\beta)})}
  \bigg[\sup_{t\in[0,T]} \E
  \big[\! 
    \max\!\big\{
      1 
      ,  
      \big\| Y_{\round{t}{h}}^{h,I} \big\|_{\bfH_{\rho}} 
    \big\}
 \big]\bigg]
\label{eq:det_term_1}\\ & 
 \leq 
 2^{\nicefrac32} 
 h^{2(\gamma-\beta)} T \,
 \| F|_{\bfH_{\rho}} \|_{\Lip(\bfH_{\rho},\bfH_{2(\gamma-\beta)})}
 \bigg[\sup_{t\in [0,T],\, x\in \bfP_I (\bfH_0) } \| v^I_{0,1}(t,x) \|_{L(\bfP_I(\bfH_0),\R)}\bigg] 
\\ & \quad \cdot 
 \bigg[\sup_{t\in[0,T]} \E
  \big[\! 
    \max\!\big\{
      1 
      ,  
      \big\| Y_{\round{t}{h}}^{h,I} \big\|_{\bfH_{\rho}} 
    \big\}
 \big]\bigg].
 \end{split}
 \end{equation}
%
In the next step we estimate the second term on the right-hand side of~\eqref{eq:split_det_term}. 
\color{black}
Note that item~\eqref{item:v_diff}, \eqref{eq:Y_mild}, and 
Proposition~\ref{thm:mildIto}
(with  $H=(\bfP_I(\bfH_0))^2$, 
$S_{r}= [(\bfP_I(\bfH_0))^2\ni(x_1,x_2)\mapsto (e^{r\bfA} x_1,x_2)\in (\bfP_I(\bfH_0))^2]$, 
$\varphi=[(\bfP_I(\bfH_0))^2\ni(x_1,x_2)\mapsto v_{0,1}^I(t,x_1)\,x_2 \in \R]$,  
$\tau_0=\round{t}{h}$, $\tau_1=t$, 
$X_s=(Y^{h,I}_s,\bfP_IF(Y^{h,I}_{\round{t}{h}}))$,
$Y_s=(e^{\delta(s)\bfA}\bfP_I F(Y^{h,I}_{\round{t}{h}}),0)$, 
$Z_s u=(e^{\delta(s)\bfA}\bfP_I B(Y^{h,I}_{\round{t}{h}})u,0)$ 
for $r\in[0,\infty)$, $s\in[\round{t}{h},t]$, $t\in(0,T]\setminus\{h,2h,3h,\ldots\}$,  $I\in\{J\subseteq\H\colon J\text{ is a finite set}\}$, $u\in U$ 
in the notation of Proposition~\ref{thm:mildIto}) imply 
that for every finite $I\subseteq \H$ and every $t\in [0,T]$ it holds that  
\begin{equation}\label{eq:mild_Ito1}
 \begin{aligned}
& 
  \big[
    v^I_{0,1}\big(t,Y_t^{h,I}\big) \big( \bfP_I F\big(Y_{\round{t}{h}}^{h,I}\big) \big)
    -
    v^I_{0,1}\big(t, e^{\delta(t) \bfA } Y_{\round{t}{h}}^{h,I} \big) 
    \big( \bfP_I F\big( Y_{\round{t}{h}}^{h,I} \big) \big)
  \big]_{\P,\,\calB(\R)} 
\\ &
 = 
  \int_{\round{t}{h}}^{t}
    v^I_{0,2}\big(t,e^{(t-s)\bfA } Y_s^{h,I} \big) 
    \big(
      \bfP_I F\big( Y_{\round{t}{h}}^{h,I} \big)
      ,
      e^{\delta(t)\bfA }
      \bfP_I F\big(Y_{\round{t}{h}}^{h,I}\big)
    \big)
  \,\mathrm{d}s
\\ & \quad 
  +
  \int_{\round{t}{h}}^{t}
    v^I_{0,2}\big(t,e^{(t-s)\bfA } Y_s^{h,I} \big) 
    \big(
      \bfP_I F\big( Y_{\round{t}{h}}^{h,I} \big)
      ,
      e^{\delta(t)\bfA }
      \bfP_I B\big(Y_{\round{t}{h}}^{h,I}\big)
    \big)
  \,\mathrm{d}W_s
\\ & \quad 
  +
  \tfrac{1}{2}
  \sum_{u\in \U} 
    \int_{\round{t}{h}}^{t}  
      v^I_{0,3}\big(t,e^{(t-s)\bfA } Y_s^{h,I} \big) 
\\ & \quad \qquad \qquad 
      \big(
	\bfP_I F\big( Y_{\round{t}{h}}^{h,I} \big) 
	,
	e^{\delta(t)\bfA} \bfP_I B\big(Y_{\round{t}{h}}^{h,I}\big) u
	,
	e^{\delta(t)\bfA} \bfP_I B\big(Y_{\round{t}{h}}^{h,I} \big)u
      \big)
    \,\mathrm{d}s
    .
 \end{aligned}
\end{equation} 
Moreover, observe that It\^o's isometry, item~\eqref{item:v_unif_bddness},
item~\eqref{it:Lip_est_standard1} in Lemma~\ref{lemma:Lip_est_standard}, Lemma~\ref{lemma:sg-estimate1}, and 
item~\eqref{it:Y_unif_L4bdd} in  
Lemma~\ref{lemma:Y_unif_Lpbddness}
ensure that for every finite $I\subseteq \H$ and every $t\in [0,T]$ it holds that  
\begin{equation}\label{eq:stoch_int_vanish_2}
 \E
	\bigg[
    \Big| 
  \int_{\round{t}{h}}^{t}
    v^I_{0,2}\big(t,e^{(t-s)\bfA } Y_s^{h,I} \big) 
    \big(
      \bfP_I F\big( Y_{\round{t}{h}}^{h,I} \big)
      ,
      e^{\delta(t)\bfA }
      \bfP_I B\big(Y_{\round{t}{h}}^{h,I}\big)
    \big)
  \,\mathrm{d}W_s
 \Big|^2 
  \bigg]
 <\infty.
\end{equation}
This, item~\eqref{it:Lip_est_standard1} in Lemma~\ref{lemma:Lip_est_standard}, Lemma~\ref{lemma:sg-estimate1},~\eqref{eq:deriv_est_B_and_F}, and~\eqref{eq:mild_Ito1}
imply that for every finite $I\subseteq \H$ it holds that
\begin{equation}
\begin{aligned}
&
  \left|
    \E\! \left[
	\int_{0}^{T} 
	      \Big(v^I_{0,1}\big(t,Y_t^{h,I}\big) 
	      \big( \bfP_I F\big(Y_{\round{t}{h}}^{h,I}\big) \big)
	      -
	      v^I_{0,1}\big(t, e^{\delta(t) \bfA } Y_{\round{t}{h}}^{h,I} \big)
	      \big( \bfP_I F\big( Y_{\round{t}{h}}^{h,I} \big) \big)\Big)
	\,\mathrm{d}t
    \right]
  \right|
\\ &
  \leq h  T
  \sup\!\Big\{
    \| v^I_{0,k}(t,x)\|_{L^{(k)}(\bfP_I(\bfH_0),\R)}
    \colon
    \substack{
      x\in \bfP_I (\bfH_0), \\
      t\in [0,T],\,k\in \{2,3\} 
    }
  \Big\}
  \| F|_{\bfP_I (\bfH_0)} \|_{\Lip(\bfP_I (\bfH_0),\bfH_0)}
\\ & \quad \cdot
  \big(
    \| F|_{\bfP_I (\bfH_0)} \|_{\Lip(\bfP_I (\bfH_0),\bfH_0)}
    +  
    \tfrac{1}{2}\| B|_{\bfP_I (\bfH_0)} \|_{\Lip(\bfP_I (\bfH_0),L_2(U,\bfH_0))}^2
  \big)
\label{eq:det_term_2}\\ & \quad 
  \cdot 
  \bigg[
  \sup_{t\in[0,T]} \E
  \big[\! 
    \max\!\big\{
      1 
      ,  
      \big\| Y_{\round{t}{h}}^{h,I} \big\|_{\bfH_0}^3 
    \big\}
  \big]
  \bigg]
\\ & 
  \leq \tfrac{3}{2} h  T \mathfrak{m}^3
  \sup\!\Big\{
    \| v^I_{0,k}(t,x)\|_{L^{(k)}(\bfP_I(\bfH_0),\R)}
    \colon
    \substack{
      x\in \bfP_I (\bfH_0), \\
      t\in [0,T],\,k\in \{2,3\} 
    }
  \Big\}
\\ & \quad \cdot
  \bigg[
  \sup_{t\in[0,T]} \E
  \big[\! 
    \max\!\big\{
      1 
      ,  
      \big\| Y_{\round{t}{h}}^{h,I} \big\|_{\bfH_0}^3 
    \big\}
  \big]
  \bigg].
\end{aligned}
\end{equation}
%
In the next step we estimate the third term on the right-hand side of~\eqref{eq:split_det_term}.
Note that Lemma~\ref{lemma:sg-estimate2} 
implies that for every finite $I\subseteq \H$ it holds that
\begin{equation}\label{eq:det_term_3}
\begin{split}
&  
  \left|
    \E\! \left[
	\int_{0}^{T} 
	    v^I_{0,1}\big(t, e^{\delta(t) \bfA } Y_{\round{t}{h}}^{h,I}\big)
	    \big(
	      \bfP_I F\big( Y_{\round{t}{h}}^{h,I} \big)
	      -
	      \bfP_I F\big(e^{\delta(t) \bfA } Y_{\round{t}{h}}^{h,I} \big)
	    \big)
	\,\mathrm{d}t
    \right]
  \right|
\\ & 
  \leq
  \bigg[
  \sup_{t\in [0,T],\,x\in \bfP_I (\bfH_0)} \| v^I_{0,1}(t,x) \|_{L(\bfP_I(\bfH_0),\R) }
  \bigg]
  \left\| F \right\|_{\Lip(\bfH_{\beta-\gamma},\bfH_0)} 
\\ & \quad \cdot
  \int_{0}^{T}
    \E 
    \Big[ 
      \big\| 
        \bfLambda^{\beta-\gamma} \big(\!\id_{\mathbf{H}_0} - e^{\delta(t)\bf A} \big) Y_{\round{t}{h}}^{h,I} 
      \big\|_{\bfH_0}
    \Big]
  \,dt
\\ & 
  \leq
  2^{\nicefrac32} 
  h^{2(\gamma-\beta)}  
  T  \,
  \left\| F \right\|_{\Lip(\bfH_{\beta-\gamma},\bfH_0)}
  \bigg[\sup_{t\in [0,T],\,x\in \bfP_I (\bfH_0)} \| v^I_{0,1}(t,x) \|_{L(\bfP_I(\bfH_0),\R) } \bigg]
\\ & \quad \cdot  
  \bigg[\sup_{t\in [0,T]}
  \E
  \big[
    \big\| Y_{\round{t}{h}}^{h,I} \big\|_{\bfH_{\gamma-\beta}}
  \big]\bigg]
  . 
 \end{split}
 \end{equation}
In the next step we estimate the fourth term on the right-hand side of~\eqref{eq:split_det_term}.
Note that item~\eqref{item:v_diff} implies that for every finite $I\subseteq\H$ and every $t\in [0,T]$ it holds that
\begin{equation}\label{eq:phiI_C2}
\big(\bfP_I ( \bfH_0) \ni x \mapsto \phi^I(t,x) \in \R\big)\in C^2(\bfP_I(\bfH_0),\R).
\end{equation} 
In addition, observe that item~\eqref{item:v_diff} ensures that for every finite $I\subseteq \H$ and every
$t\in [0,T]$, $x, v_1, v_2 \in \bfP_I ( \bfH_0)$ it holds that 
\begin{equation}
 \begin{aligned}
  \phi^I_{0,1}(t,x)(v_1) 
  =
  v^I_{0,2}(t,x)(\bfP_I F(x),v_1)
  +
  v^I_{0,1}(t,x)(\bfP_I F'(x)(v_1))
 \end{aligned}
\end{equation}
and 
\begin{equation}
\begin{aligned}
 \phi^I_{0,2}(t,x)(v_1,v_2)
& =
 v^I_{0,3}(t,x)(\bfP_I F(x),v_1,v_2)
 +
 v^I_{0,2}(t,x)(\bfP_I F'(x)(v_2),v_1)
\\ & \quad 
 +
 v^I_{0,2}(t,x)(\bfP_I F'(x)(v_1),v_2)
 +
 v^I_{0,1}(t,x)(\bfP_I F''(x)(v_1,v_2)).
\end{aligned}
\end{equation}
This and item~\eqref{it:Lip_est_standard1} in Lemma~\ref{lemma:Lip_est_standard} imply that for every finite $I\subseteq \H$ 
and every  $t\in[0,T]$, $x\in\bfP_I(\bfH_0)$ 
it holds that
\begin{equation}\label{eq:Dphi_est1}
\begin{aligned}
 \| \phi^I_{0,1}(t,x) \|_{L(\bfP_I(\bfH_0),\R)}
  & \leq   
  2
  \sup\!\Big\{
    \| v^I_{0,k}(s,y)\|_{L^{(k)}(\bfP_I(\bfH_0),\R))}
    \colon
    \substack{
      y\in \bfP_I (\bfH_0), \\
      s\in [0,T],\,k\in \{1,2\} 
    }
  \Big\}
\\ & \quad \cdot 
  \| F|_{\bfP_I ( \bfH_0)} \|_{C^1_{\mathrm{b}}(\bfP_I (\bfH_0),\bfH_0)}
  \max\{ 1, \left\| x \right\|_{\bfH_0} \},
\end{aligned}
\end{equation}
and
\begin{equation}\label{eq:Dphi_est2}
\begin{aligned}
 \| \phi^I_{0,2}(t,x) \|_{L^{(2)}(\bfP_I(\bfH_0),\R)}
 & 
\leq 
  4 
  \sup\!\Big\{
    \| v^I_{0,k}(s,y)\|_{L^{(k)}(\bfP_I(\bfH_0),\R))}
    \colon
    \substack{
      y\in \bfP_I (\bfH_0), \\
      s\in [0,T],\,k\in \{1,2,3\} 
    }
  \Big\}
\\ & \quad  \cdot
  \| F|_{\bfP_I ( \bfH_0)} \|_{C^2_{\mathrm{b}}(\bfP_I (\bfH_0),\bfH_0)}
  \max\{ 1, \left\| x \right\|_{\bfH_0} \}
  .
\end{aligned}
\end{equation}
Next observe that 
\eqref{eq:Y_mild}, \eqref{eq:phiI_C2}, and 
Proposition~\ref{thm:mildIto}
(with $H=\bfP_I(\bfH_0)$, 
$S_{r}=e^{r\bfA}|_{\bfP_I(\bfH_0)}$, 
$\varphi=\varphi^I(t,\cdot)$,  
$\tau_0=\round{t}{h}$, $\tau_1=t$, 
$X_s=Y^{h,I}_s$, 
$Y_s=e^{\delta(s)\bfA}\bfP_I F(Y^{h,I}_{\round{t}{h}})$, $Z_s=e^{\delta(s)\bfA}\bfP_I B(Y^{h,I}_{\round{t}{h}})$  
for $r\in[0,\infty)$, $s\in[\round{t}{h},t]$, $t\in(0,T]\setminus\{h,2h,3h,\ldots\}$, $I\in\{J\subseteq\H\colon J\text{ is a finite set}\}$ 
in the notation of Proposition~\ref{thm:mildIto}) 
imply that for every finite $I\subseteq \H$ and every $t\in [0,T]$ it holds that
\begin{equation}\label{eq:Ito_phi}
 \begin{aligned}
&
  \big[
    v^I_{0,1}\big(t, e^{\delta(t) \bfA } Y_{\round{t}{h}}^{h,I}\big)
    \big( \bfP_I F\big(e^{\delta(t) \bfA } Y_{\round{t}{h}}^{h,I} \big) \big)
    -
    v^I_{0,1}\big(t,Y_t^{h,I}\big)
    \big( \bfP_I F\big(Y_t^{h,I}\big) \big)
  \big]_{\P,\,\calB(\R)}
\\ &   
  =
  - \int_{\round{t}{h}}^{t}   
    \phi^I_{0,1}\big(t,e^{(t-s)\bfA } Y_s^{h,I} \big) 
    \big( 
      e^{\delta(t)\bfA }
      \bfP_I F\big(Y_{\round{t}{h}}^{h,I}\big)
    \big)
  \,\mathrm{d}s
\\ & \quad
  -
  \int_{\round{t}{h}}^{t}
    \phi^I_{0,1}\big(t,e^{(t-s)\bfA } Y_s^{h,I} \big) 
    \big( 
      e^{\delta(t)\bfA }
      \bfP_I B\big(Y_{\round{t}{h}}^{h,I}\big)
    \big) 
  \,\mathrm{d}W_s
\\ & \quad 
  -
  \tfrac{1}{2}
  \sum_{u\in \U} 
    \int_{\round{t}{h}}^{t}
      \phi^I_{0,2}\big(t,e^{(t-s)\bfA}Y_s^{h,I}\big)
      \big( 
	e^{\delta(t)\bfA} \bfP_I B\big(Y_{\round{t}{h}}^{h,I}\big)u
	,
	e^{\delta(t)\bfA} \bfP_I B\big(Y_{\round{t}{h}}^{h,I}\big)u
      \big) 
    \,\mathrm{d}s.
 \end{aligned}
\end{equation}
Moreover, note that It\^o's isometry, item~\eqref{item:v_unif_bddness}, item~\eqref{it:Lip_est_standard1} in 
Lemma~\ref{lemma:Lip_est_standard}, Lemma~\ref{lemma:sg-estimate1},  
item~\eqref{it:Y_unif_L4bdd} in 
Lemma~\ref{lemma:Y_unif_Lpbddness}, 
and \eqref{eq:Dphi_est1} 
assure that for every finite $I\subseteq \H$ and every $t\in [0,T]$ it holds that
\begin{equation}\label{eq:stoch_int_vanish_3}
\E
  \bigg[
  \Big|
\int_{\round{t}{h}}^{t}
    \phi^I_{0,1}\big(t,e^{(t-s)\bfA } Y_s^{h,I} \big) 
    \big( 
      e^{\delta(t)\bfA }
      \bfP_I B\big(Y_{\round{t}{h}}^{h,I}\big)
    \big) 
  \,\mathrm{d}W_s
\Big|^2 
   \bigg]
  <\infty.
\end{equation}
Combining this, 
item~\eqref{it:Lip_est_standard2} in
Lemma~\ref{lemma:Lip_est_standard}, 
Lemma~\ref{lemma:sg-estimate1},~\eqref{eq:deriv_est_B_and_F},~\eqref{eq:Dphi_est1},~\eqref{eq:Dphi_est2}, and~\eqref{eq:Ito_phi} implies that for every finite $I\subseteq \H$ it holds that
\begin{equation}\label{eq:det_term_4}
\begin{aligned}
& 
  \left|
    \E\! \left[
	\int_{0}^{T}
 	  \left( 
	    v^I_{0,1}\big(t, e^{\delta(t) \bfA } Y_{\round{t}{h}}^{h,I}\big)
	    \big( \bfP_I F\big(e^{\delta(t) \bfA } Y_{\round{t}{h}}^{h,I} \big)\big)
	    -
	    v^I_{0,1}\big(t,Y_t^{h,I}\big)
	    \big(\bfP_I F\big(Y_t^{h,I}\big)\big)
 	  \right)
	\,\mathrm{d}t
    \right]
  \right|
\\ &    
  \leq 
  4 h T  
  \sup\!\Big\{
    \| v^I_{0,k}(t,x)\|_{L^{(k)}(\bfP_I(\bfH_0),\R))}
    \colon
    \substack{
      x\in \bfP_I (\bfH_0), \\
      t\in [0,T],\,k\in \{1,2,3\} 
    }
  \Big\}
  \| F|_{\bfP_I ( \bfH_0)} \|_{C^2_{\mathrm{b}}(\bfP_I ( \bfH_0),\bfH_0)} 
\\  & \quad \cdot 
  \left[     
    \| F|_{\bfP_I ( \bfH_0)} \|_{\Lip(\bfP_I (\bfH_0),\bfH_0)}
    +
    \tfrac12\| B|_{\bfP_I ( \bfH_0)} \|_{\Lip(\bfP_I (\bfH_0),L_2(U,\bfH_0))}^2
  \right]
\\ & \quad \cdot
  \bigg[ 
  \sup_{t\in[0,T]} \E
  \big[\! 
    \max\!\big\{
      1 
      ,  
      \big\| Y_{\round{t}{h}}^{h,I} \big\|_{\bfH_0}^3 
    \big\}
 \big]
 \bigg] 
\\    &   
  \leq 
  6 h T \mathfrak{m}^3 
  \sup\!\Big\{
    \| v^I_{0,k}(t,x)\|_{L^{(k)}(\bfP_I(\bfH_0),\R))}
    \colon
    \substack{
      x\in \bfP_I (\bfH_0), \\
      t\in [0,T],\,k\in \{1,2,3\} 
    }
  \Big\}
\\   & \quad \cdot 
  \bigg[  
  \sup_{t\in[0,T]} \E
  \big[\! 
    \max\!\big\{
      1 
      ,  
      \big\| Y_{\round{t}{h}}^{h,I} \big\|_{\bfH_0}^3 
    \big\}
 \big]
 \bigg]
 .
 \end{aligned}
 \end{equation}
%
In the next step we estimate the first term on the right-hand side of~\eqref{eq:split_noise_term}. 
Note that Lemma~\ref{lemma:Schatten_hoelder} implies that for every finite $I\subseteq \H$ it holds that
\begin{align}\label{eq:noise_term_1_aux}
&\notag
 \left|
  \E\! \left[
    \sum_{u\in \U}
      \int_{0}^{T} 
	v^I_{0,2}\big(t,Y_t^{h,I}\big)\Big(\big( e^{\delta(t)\bfA } - \id_{\mathbf{H}_0}\big) \bfP_I B\big(Y_{\round{t}{h}}^{h,I} \big) u
	  ,
	  \big( e^{\delta(t)\bfA } + \id_{\mathbf{H}_0}\big) \bfP_I B\big(Y_{\round{t}{h}}^{h,I} \big) u
	\Big)
      \,\mathrm{d}t
  \right]
\right|
\\ &
  \leq 
 \E \bigg[
  \int_{0}^{T}
  \big\| v^I_{0,2}\big(t,Y_t^{h,I}\big) \big\|_{L^{(2)}(\bfP_I ( \bfH_0),\R)}
  \big\| 
    \big(\!\id_{\mathbf{H}_0} + e^{\delta(t)\bfA} \big)\bfP_I B\big(Y_{\round{t}{h}}^{h,I}\big)
  \big\|_{L_{\nicefrac{2\beta}{\gamma}}(U,\bfH_0)}
\\ &\notag \qquad \quad 
  \cdot
  \big\| 
    \big(\!\id_{\mathbf{H}_0} - e^{\delta(t)\bfA}\big)\bfP_I B\big(Y_{\round{t}{h}}^{h,I}\big)
  \big\|_{L_{\nicefrac{2\beta}{(2\beta - \gamma)}}(U,\bfH_0)}
  \,\mathrm{d}t
  \bigg].
\end{align}
Moreover, observe that $\bfLambda$ and $\bfP_I$ commute.
This, item~\eqref{it:Lip_est_standard1} in Lemma~\ref{lemma:Lip_est_standard}, and Lemma~\ref{lemma:sg-estimate1}
imply that for every finite $I\subseteq \H$ and every $t\in [0,T]$ it holds that
\begin{equation}\label{eq:1a}
 \begin{aligned}
 &
   \big\| 
    \big(\!\id_{\mathbf{H}_0} + e^{\delta(t)\bfA}\big)\bfP_I B\big(Y_{\round{t}{h}}^{h,I}\big)
   \big\|_{L_{\nicefrac{2\beta}{\gamma}}(U,\bfH_0)}
\\ &  \leq    
   \| \id_{\mathbf{H}_0} + e^{\delta(t)\bfA} \|_{L(\bfH_0)}
   \| \bfLambda^{-\gamma} \|_{L_{\nicefrac{2\beta}{\gamma}}(\bfH_0)}
   \big\| \bfLambda^{\gamma} \bfP_I B\big(Y_{\round{t}{h}}^{h,I}\big)\big\|_{L(U,\bfH_0)}
\\ & \leq     
   2
   \| \bfLambda^{-2\beta} \|_{L_1(\bfH_0)}^{\nicefrac{\gamma}{(2\beta)}}
   \| B|_{\bfH_{\rho}} \|_{\Lip(\bfH_{\rho},L(U,\bfH_{\gamma}))}
   \max\!\big\{ 1, \| Y_{\round{t}{h}}^{h,I} \|_{\bfH_{\rho}}\big\}.
 \end{aligned}
\end{equation}
In addition, note that item~\eqref{it:Lip_est_standard1} in Lemma~\ref{lemma:Lip_est_standard} and Lemma~\ref{lemma:sg-estimate2} imply that for every finite $I\subseteq \H$ and every $t\in [0,T]$ it holds that
\begin{equation}\label{eq:1b}
 \begin{aligned}
&
 \big\| 
  \big(\! \id_{\mathbf{H}_0} - e^{\delta(t)\bfA}\big) \bfP_I B\big(Y_{\round{t}{h}}^{h,I}\big)
 \big\|_{L_{\nicefrac{2\beta}{(2\beta - \gamma)}}(U,\bfH_0)}
\\  &
  \leq    
   \big\|
   \bfLambda^{2(\beta - \gamma)}\big(\!\id_{\mathbf{H}_0} - e^{\delta(t)\bfA}\big) 
   \big\|_{L(\bfH_0)}
   \| \bfLambda^{\gamma-2\beta} \|_{L_{\nicefrac{2\beta}{(2\beta - \gamma)}}(\bfH_0)}
   \big\| \bfLambda^{\gamma} \bfP_I B\big(Y_{\round{t}{h}}^{h,I}\big) \big\|_{L(U,\bfH_0)}
\\ &
  \leq
  2^{\nicefrac32} 
  h^{2(\gamma-\beta)}
  \| \bfLambda^{-2\beta} \|_{L_1(\bfH_0)}^{\nicefrac{(2\beta-\gamma)}{(2\beta)}}
  \| B|_{\bfH_{\rho}} \|_{\Lip(\bfH_{\rho},L(U,\bfH_{\gamma}))}
   \max\!\big\{ 1, \| Y_{\round{t}{h}}^{h,I} \|_{\bfH_{\rho}}\big\}.
 \end{aligned}
\end{equation}
Combining~\eqref{eq:noise_term_1_aux}, \eqref{eq:1a}, and \eqref{eq:1b} 
ensures that for every finite $I\subseteq \H$ it holds that
\begin{align}\label{eq:noise_term_1}
&\notag
 \left|
  \E\! \left[
    \sum_{u\in \U}
      \int_{0}^{T} 
	v^I_{0,2}\big(t,Y_t^{h,I}\big)\big(
	  \big( e^{\delta(t)\bfA } - \id_{\mathbf{H}_0}\big) \bfP_I B\big(Y_{\round{t}{h}}^{h,I} \big) u
	  ,
	  \big( e^{\delta(t)\bfA } + \id_{\mathbf{H}_0}\big) \bfP_I B\big(Y_{\round{t}{h}}^{h,I} \big) u
	\big)
      \,\mathrm{d}t
  \right]
\right|
\\ &\notag
  \leq  
  2^{\nicefrac52} 
  h^{2(\gamma-\beta)}  
  T\,
  \| \bfLambda^{-2\beta} \|_{L_1(\bfH_0)}
 \| 
  B|_{\bfH_{\rho}}
 \|_{\Lip(\bfH_{\rho},L(U,H_{\gamma}))}^2
 \\ & \quad \cdot
 \bigg[\sup_{t\in [0,T],\,x\in \bfP_I (\bfH_0) } \| v^I_{0,2}(t,x) \|_{L^{(2)}(\bfP_I(\bfH_0),\R)}\bigg]
  \bigg[\sup_{t\in[0,T]} \E
  \big[\! 
    \max\!\big\{
      1 
      ,  
      \big\| Y_{\round{t}{h}}^{h,I} \big\|_{\bfH_{\rho}}^2 
    \big\}
 \big]\bigg]
  .
\end{align}
In the next step we estimate the second term on the right-hand side of~\eqref{eq:split_noise_term}.  
Note that item~\eqref{item:v_diff}, \eqref{eq:Y_mild}, and 
Proposition~\ref{thm:mildIto}
(with 
$H=(\bfP_I(\bfH_0))^3$, 
$S_{r}=$ $[(\bfP_I(\bfH_0))^3\ni(x_1,x_2,x_3)\mapsto (e^{r\bfA} x_1,x_2,x_3)\in (\bfP_I(\bfH_0))^3]$, 
$\varphi=[(\bfP_I(\bfH_0))^3\ni(x_1,x_2,x_3)\mapsto v_{0,2}^I(t,x_1)(x_2,x_3) \in \R]$,  
$\tau_0=\round{t}{h}$, $\tau_1=t$, 
$X_s=(Y^{h,I}_s,\bfP_IB(Y^{h,I}_{\round{t}{h}})u,\bfP_IB(Y^{h,I}_{\round{t}{h}})u)$,  
$Y_s=(e^{\delta(s)\bfA}\bfP_I F(Y^{h,I}_{\round{t}{h}}),0,0)$, 
$Z_s u'=(e^{\delta(s)\bfA}\bfP_I B(Y^{h,I}_{\round{t}{h}})u',0,0)$ 
for  $r\in [0,\infty)$, $s\in[\round{t}{h},t]$, $t\in(0,T]\setminus\{h,2h,3h,\ldots\}$, $I\in\{J\subseteq\H\colon J\text{ is a finite set}\}$, $u,u'\in U$ in the notation of Proposition~\ref{thm:mildIto})
ensure that for every finite $I\subseteq \H$ and every $t\in [0,T]$, $u\in \U$ it holds that
\begin{equation}\label{eq:B_mild_Ito_simple}
 \begin{aligned}
&
\big[
  v^I_{0,2}\big(t,Y_t^{h,I}\big)
  \big(
    \bfP_I B\big(Y_{\round{t}{h}}^{h,I} \big) u
    ,
    \bfP_I B\big(Y_{\round{t}{h}}^{h,I} \big) u
  \big)
\\ & \quad 
  -
  v^I_{0,2}\big(t, e^{\delta(t) \bfA } Y_{\round{t}{h}}^{h,I}\big) 
  \big(
    \bfP_I B\big( Y_{\round{t}{h}}^{h,I} \big) u
    ,
    \bfP_I B\big( Y_{\round{t}{h}}^{h,I} \big) u
  \big)
\big]_{\P,\calB(\R)}
\\ & 
%
  = 
  \int_{\round{t}{h}}^{t}
    v^I_{0,3}\big(t,e^{(t-s)\bfA } Y_s^{h,I} \big) 
    \big(
      \bfP_I B\big( Y_{\round{t}{h}}^{h,I} \big) u
      ,
      \bfP_I B\big( Y_{\round{t}{h}}^{h,I} \big) u
      ,
      e^{\delta(t)\bfA }
      \bfP_I F\big(Y_{\round{t}{h}}^{h,I}\big)
    \big)
  \,\mathrm{d}s
\\ & 
  +
  \int_{\round{t}{h}}^{t}
    v^I_{0,3}\big(t,e^{(t-s)\bfA } Y_s^{h,I}\big) 
    \big(
      \bfP_I B\big( Y_{\round{t}{h}}^{h,I} \big) u
      ,
      \bfP_I B\big( Y_{\round{t}{h}}^{h,I} \big) u
      ,
      e^{\delta(t)\bfA }
      \bfP_I B\big(Y_{\round{t}{h}}^{h,I}\big)
    \big)
  \,\mathrm{d}W_s
\\ & 
  +
  \tfrac{1}{2}
  \sum_{u'\in \U} 
    \int_{\round{t}{h}}^{t}  
      v^I_{0,4}\big(t,e^{(t-s)\bfA } Y_s^{h,I} \big) 
\\ & \quad \qquad
      \big(
	\bfP_I B\big( Y_{\round{t}{h}}^{h,I} \big) u
	,
	\bfP_I B\big( Y_{\round{t}{h}}^{h,I} \big) u
	,
	e^{\delta(t)\bfA} \bfP_I B\big(Y_{\round{t}{h}}^{h,I}\big) u'
	,
	e^{\delta(t)\bfA} \bfP_I B\big(Y_{\round{t}{h}}^{h,I} \big)u'
      \big)
    \,\mathrm{d}s.
\end{aligned}
\end{equation}
Moreover, observe that It\^o's isometry, item~\eqref{item:v_unif_bddness}, item~\eqref{it:Lip_est_standard1} in Lemma~\ref{lemma:Lip_est_standard}, Lemma~\ref{lemma:sg-estimate1}, and
item~\eqref{it:Y_unif_L4bdd} in Lemma~\ref{lemma:Y_unif_Lpbddness} 
ensure that for every finite $I\subseteq \H$ and every $t\in [0,T]$, $u\in\U$ it holds that
\begin{multline}\label{eq:stoch_int_vanish_4}
  \E
    \bigg[    
    \Big|
    \int_{\round{t}{h}}^{t}
      v^I_{0,3}\big(t,e^{(t-s)\bfA } Y_s^{h,I}\big) 
      \big(
	\bfP_I B\big( Y_{\round{t}{h}}^{h,I} \big) u
	,
	\bfP_I B\big( Y_{\round{t}{h}}^{h,I} \big) u
	,
	e^{\delta(t)\bfA }
	\bfP_I B\big(Y_{\round{t}{h}}^{h,I}\big)
      \big)
    \,\mathrm{d}W_s
  \Big|^2 
   \bigg]
\\  < \infty.
\end{multline}
Combining this, item~\eqref{it:Lip_est_standard1} in Lemma~\ref{lemma:Lip_est_standard}, Lemma~\ref{lemma:sg-estimate1},~\eqref{eq:deriv_est_B_and_F}, and~\eqref{eq:B_mild_Ito_simple} 
implies that for every finite $I\subseteq \H$ it holds that
\begin{align}\label{eq:noise_term_2}
&  \notag
  \Bigg|
    \E \Bigg[
      \sum_{u\in \U}
	\int_{0}^{T} 
	  \Big(
	    v^I_{0,2} \big(t,Y_t^{h,I}\big)\big(
	    \bfP_I B\big(Y_{\round{t}{h}}^{h,I} \big) u
	    ,
	    \bfP_I B\big(Y_{\round{t}{h}}^{h,I} \big) u
	  \big)
\\\notag & 
\phantom{  
  \Bigg|
    \E \Bigg[
      \sum_{u\in \U}
	\int_{0}^{T} 
}
	  -
	  v^I_{0,2}\big(t, e^{\delta(t) \bfA } Y_{\round{t}{h}}^{h,I}\big) \big(
	    \bfP_I B\big( Y_{\round{t}{h}}^{h,I} \big) u
	    ,
	    \bfP_I B\big( Y_{\round{t}{h}}^{h,I} \big) u
	  \big)
	\Big)
	\,\mathrm{d}t
    \Bigg]
  \Bigg|
\\\notag &= 
  \Bigg|
      \sum_{u\in \U}
	\int_{0}^{T} 
	  \E\Big[
	    v^I_{0,2} \big(t,Y_t^{h,I}\big)\big(
	    \bfP_I B\big(Y_{\round{t}{h}}^{h,I} \big) u
	    ,
	    \bfP_I B\big(Y_{\round{t}{h}}^{h,I} \big) u
	  \big)
\\ & 
\phantom{  
  \Bigg|
    \E \Bigg[
      \sum_{u\in \U}
	\int_{0}^{T} 
}
	  -
	  v^I_{0,2}\big(t, e^{\delta(t) \bfA } Y_{\round{t}{h}}^{h,I}\big) \big(
	    \bfP_I B\big( Y_{\round{t}{h}}^{h,I} \big) u
	    ,
	    \bfP_I B\big( Y_{\round{t}{h}}^{h,I} \big) u
	  \big)
	\Big]
	\,\mathrm{d}t
  \Bigg|
\\&
\notag
  \leq 
  h T
  \sup\!\Big\{
    \| v^I_{0,k}(t,x)\|_{L^{(k)}(\bfP_I(\bfH_0),\R))}
    \colon
    \substack{
      x\in \bfP_I (\bfH_0), \\
      t\in [0,T],\,k\in \{3,4\} 
    }
  \Big\}
  \| B|_{\bfP_I ( \bfH_0)} \|_{\Lip(\bfP_I ( \bfH_0),L_2(U,\bfH_0))}^2   
\\ &
\notag
  \quad 
  \cdot  
  \left(
    \| F|_{\bfP_I ( \bfH_0)} \|_{\Lip(\bfP_I (\bfH_0),\bfH_0)}
    +
    \tfrac{1}{2}\| B|_{\bfP_I ( \bfH_0)} \|_{\Lip(\bfP_I ( \bfH_0),L_2(U,\bfH_0))}^2  
  \right)
\\ & \quad 
\notag
  \cdot
  \bigg[
  \sup_{t\in[0,T]} \E
  \big[\! 
    \max\!\big\{
      1 
      ,  
      \big\| Y_{\round{t}{h}}^{h,I} \big\|_{\bfH_0}^4 
    \big\}
 \big]
 \bigg]
\\ & \notag\leq 
  \tfrac{3}{2} h T \mathfrak{m}^4
  \sup\!\Big\{
    \| v^I_{0,k}(t,x)\|_{L^{(k)}(\bfP_I(\bfH_0),\R))}
    \colon
    \substack{
      x\in \bfP_I (\bfH_0), \\
      t\in [0,T],\,k\in \{3,4\} 
    }
  \Big\}
\\\notag & \quad \cdot
  \bigg[  
  \sup_{t\in[0,T]} \E
  \big[\! 
    \max\!\big\{
      1 
      ,  
      \big\| Y_{\round{t}{h}}^{h,I} \big\|_{\bfH_0}^4 
    \big\}
 \big]
 \bigg].
 \end{align}
In the next step we estimate the third term on the right-hand side of~\eqref{eq:split_noise_term}.
Note that the Cauchy-Schwarz inequality, Lemma~\ref{lemma:sg-estimate2},~\eqref{eq:linGrowB},
and~\eqref{eq:LipB} imply that for every finite $I\subseteq \H$ it holds that
\begin{align}\label{eq:noise_term_3}
 \notag
 &
  \Bigg|
    \E \Bigg[
      \sum_{u\in \U}
	\int_{0}^{T} 
	  v^I_{0,2}\big(t,e^{\delta(t) \bfA } Y_{\round{t}{h}}^{h,I}\big)
\\ &\notag \quad 
	  \big(
	    \big[
	      \bfP_I B\big( Y_{\round{t}{h}}^{h,I} \big)
	      -
	      \bfP_I B\big( e^{\delta(t)\bfA } Y_{\round{t}{h}}^{h,I} \big) 
	    \big]
	    u
	    ,
	    \big[
	      \bfP_I B\big( Y_{\round{t}{h}}^{h,I} \big)
	      +
	      \bfP_I B\big( e^{\delta(t)\bfA } Y_{\lfloor t \rfloor_h}^{h,I}\big) 
	    \big]
	    u
	  \big)
	\,\mathrm{d}t
    \Bigg]
  \Bigg|
\\  &\notag
  \leq      
  \bigg[
  \sup_{t\in [0,T],\,x\in \bfP_I (\bfH_0)}
  \| v^I_{0,2}(t,x) \|_{L^{(2)}(\bfP_I ( \bfH_0),\R)}
  \bigg]
\\ &\notag \quad \cdot
  \E \Bigg[
  \int_{0}^{T}\!\!
  \bigg(
    \sum_{u\in \U}
    |\mu_u|^{-2}
    \big\|
      \big[
	B\big( Y_{\round{t}{h}}^{h,I} \big)
	-
	B\big( e^{\delta(t)\bfA } Y_{\round{t}{h}}^{h,I} \big) 
      \big]
      u
    \big\|_{\bfH_0}^2
  \bigg)^{\!\!\nicefrac{1}{2}}
\\ & \quad \qquad \cdot 
  \bigg(
    \sum_{u\in \U}
    |\mu_u|^{2}
    \big\|
      \big[
	B\big( Y_{\round{t}{h}}^{h,I} \big)
	+
	B\big( e^{\delta(t)\bfA } Y_{\round{t}{h}}^{h,I} \big) 
      \big]
      u
    \big\|_{\bfH_0}^2
  \bigg)^{\!\!\nicefrac{1}{2}}
    \,\mathrm{d}t
  \Bigg]
\\ & \notag
  \leq 
    2\,\mathfrak{l}\,\mathfrak{c}
    \bigg[
  \sup_{t\in [0,T],\,x\in \bfP_I (\bfH_0)}
  \| v^I_{0,2}(t,x) \|_{L^{(2)}(\bfP_I ( \bfH_0),\R)}
  \bigg]
\\ & \notag\quad \cdot
  \int_{0}^{T}
    \big\|
	\bfLambda^{2(\beta-\gamma)}
	\big(\!
	  \id_{\mathbf{H}_0} 
	  -
	  e^{\delta(t)\bfA }
	\big)
    \big\|_{L(\bfH_0)}
    \E 
    \big[ \!
    \max\!\big\{
      1
      ,
      \|
	Y_{\round{t}{h}}^{h,I}
      \|_{\bfH_{\gamma-\beta}}^2
      \big\}
    \big]
    \,dt
\\  &\notag 
\begin{aligned}
& \leq    
    2^{\nicefrac 52}
    h^{2(\gamma-\beta)} T\, 
    \mathfrak{l}\, \mathfrak{c} 
  \bigg[\sup_{t\in [0,T],\,x\in \bfP_I (\bfH_0)}
  \| v^I_{0,2}(t,x) \|_{L^{(2)}(\bfP_I ( \bfH_0),\R)}\bigg]
 \\ &\notag \quad \cdot
  \bigg[\sup_{t\in [0,T]}
  \E 
  \big[\!
    \max\!\big\{
      1
      ,
      \big\|
	Y_{\round{t}{h}}^{h,I}
      \big\|_{\bfH_{\gamma-\beta}}^2
    \big\}
  \big]\bigg]
    .
\end{aligned}
\end{align}
In the next step we estimate the final term on the right-hand side of~\eqref{eq:split_noise_term}. 
Note that item~\eqref{item:v_diff} implies
that for every finite $I\subseteq\H$ and every $t\in [0,T]$ it holds that 
\begin{equation}\label{eq:psiI_C2}
\big(\bfP_I (\bfH_0) \ni x \mapsto \psi^I(t,x) \in \R\big)\in C^2(\bfP_I(\bfH_0),\R).
\end{equation} 
In addition, observe that item~\eqref{item:v_diff} ensures that for every finite $I\subseteq \H$ and every $t\in [0,T]$, $x,v_1,v_2\in \bfP_I (\bfH_0)$ it holds that
\begin{equation}\label{eq:deriv1_phi}
 \begin{aligned}
  \psi^I_{0,1}(t,x)(v_1) 
  &
  =
  \textstyle\!\sum\limits_{u\in \U}\displaystyle 
    v^I_{0,3}(t,x)\left(
      \bfP_I B(x)u,
      \bfP_I B(x)u,
      v_1
    \right)
\\ & \quad   +
  2\textstyle\!\sum\limits_{u\in \U}\displaystyle
    v^I_{0,2}(t,x)\left(
      \bfP_I B'(x)(v_1) u,
      \bfP_I B(x) u
    \right),
 \end{aligned}
\end{equation}
and
 \begin{align} \label{eq:deriv2_phi}
\notag
  \psi^I_{0,2}(t,x) (v_1,v_2)
 &=
  \textstyle\!\sum\limits_{u\in \U}\displaystyle
    v^I_{0,4}(t,x)\left(
      \bfP_I B(x)u,
      \bfP_I B(x)u,
      v_1,
      v_2
    \right)
\\ & \quad \notag
  +
  2\textstyle\!\sum\limits_{u\in \U}\displaystyle
    v^I_{0,3}(t,x)\left(
      \bfP_I B'(x)(v_2) u,
      \bfP_I B(x) u,
      v_1
    \right)
\\ & \quad 
  +
  2\textstyle\!\sum\limits_{u\in \U}\displaystyle
    v^I_{0,3}(t,x)\left(
      \bfP_I B'(x)(v_1) u,
      \bfP_I B(x) u,
      v_2
    \right)
\\ & \quad \notag
  +
  2\textstyle\!\sum\limits_{u\in \U}\displaystyle
    v^I_{0,2}(t,x)\left(
      \bfP_I B''(x)(v_1,v_2) u,
      \bfP_I B(x) u
    \right)
\\ & \quad \notag
  +
  2\textstyle\!\sum\limits_{u\in \U}\displaystyle
    v^I_{0,2}(t,x)\left(
      \bfP_I B'(x)(v_1) u,
      \bfP_I B'(x)(v_2) u
    \right)
    .
 \end{align}
Combining this and item~\eqref{it:Lip_est_standard1} in Lemma~\ref{lemma:Lip_est_standard} 
shows that for every finite $I\subseteq \H$ 
and every $t\in [0,T]$, $x\in \bfP_I (\bfH_0)$ it holds that
\begin{equation}\label{eq:Dpsi_est1}
\begin{aligned}
\| \psi^I_{0,1} (t,x)\|_{ L(\bfP_I(\bfH_0),\R)}  
  & \leq
  3   
  \sup\!\Big\{
    \| v^I_{0,k}(s,y)\|_{L^{(k)}(\bfP_I(\bfH_0),\R))}
    \colon
    \substack{
      y\in \bfP_I (\bfH_0), \\
      s\in [0,T],\,k\in \{2,3\} 
    }
  \Big\}
\\ & \quad \cdot
  \| B|_{\bfP_I ( \bfH_0)} \|_{C^1_{\mathrm{b}}(\bfP_I (\bfH_0),L_2(U,H_0))}^2 
  \max\!\big\{ 1,\left\| x \right\|_{\bfH_0}^2 \!\big\}
\end{aligned}
\end{equation}
and
\begin{equation}\label{eq:Dpsi_est2}
\begin{aligned}
\| \psi^I_{0,2} (t,x) \|_{L^{(2)}(\bfP_I(\bfH_0),\R)}
& 
  \leq
  9   
  \sup\!\Big\{
    \| v^I_{0,k}(s,y)\|_{L^{(k)}(\bfP_I(\bfH_0),\R))}
    \colon
    \substack{
      y\in \bfP_I (\bfH_0), \\
      s\in [0,T],\,k\in \{2,3,4\} 
    }
  \Big\}
\\ & \quad \cdot
  \| B|_{\bfP_I ( \bfH_0)} \|_{C^2_{\mathrm{b}}(\bfP_I (\bfH_0),L_2(U,H_0))}^2
  \max\!\big\{ 1,\left\| x \right\|_{\bfH_0}^2\! \big\} 
  .
\end{aligned}
\end{equation} 
In addition, observe that
\eqref{eq:Y_mild}, \eqref{eq:psiI_C2}, and 
Proposition~\ref{thm:mildIto}
(with \
$H=\bfP_I(\bfH_0)$, 
$S_{r}=e^{r\bfA}|_{\bfP_I(\bfH_0)}$, 
$\varphi=\psi^I(t,\cdot)$,  
$\tau_0=\round{t}{h}$, $\tau_1=t$, 
$X_s=Y^{h,I}_s$, 
$Y_s=e^{\delta(s)\bfA}\bfP_I F(Y^{h,I}_{\round{t}{h}})$, $Z_s=e^{\delta(s)\bfA}\bfP_I B(Y^{h,I}_{\round{t}{h}})$ 
for $r\in [0,\infty)$, $s\in[\round{t}{h},t]$, $t\in(0,T]\setminus\{h,2h,3h,\ldots\}$, $I\in\{J\subseteq\H\colon\!\!$ $ J\text{ is a finite set}\}$ 
in the notation of Proposition~\ref{thm:mildIto}) 
imply that for every finite $I\subseteq \H$ and every $t\in [0,T]$ it holds that 
\begin{align}\label{eq:mild_Ito2}
& \notag 
\big[ 
  \psi^I\big(t,Y_{t}^{h,I}\big)
  -
  \psi^I\big(t,e^{\delta(t)\bfA }Y_{\round{t}{h}}^{h,I} \big)
\big]_{\P,\calB(\R)}
  =
  \int_{\round{t}{h}}^{t}
    \psi^I_{0,1 }\big(t,e^{(t-s)\bfA } Y_s^{h,I} \big) 
    \big( 
      e^{\delta(t)\bfA }
      F\big(Y_{\round{t}{h}}^{h,I}\big)
    \big)
  \,\mathrm{d}s
\\ & \quad 
  +
  \int_{\round{t}{h}}^{t}
    \psi^I_{0,1 }\big(t,e^{(t-s)\bfA } Y_s^{h,I} \big) 
    \big( 
      e^{\delta(t)\bfA }
      \bfP_I B\big(Y_{\round{t}{h}}^{h,I}\big)
    \big)
  \,\mathrm{d}W_s
\\ & \quad \notag 
  +
  \tfrac{1}{2}
  \sum_{u\in \U} 
    \int_{\round{t}{h}}^{t}
      \psi^I_{0,2 }\big(t,e^{(t-s)\bfA}Y_s^{h,I}\big)
      \big( 
	e^{\delta(t)\bfA} \bfP_I B\big(Y_{\round{t}{h}}^{h,I}\big)u
	,
	e^{\delta(t)\bfA} \bfP_I B\big(Y_{\round{t}{h}}^{h,I}\big)u
      \big)
    \,\mathrm{d}s.
 \end{align}
Next note that It\^o's isometry, item~\eqref{item:v_unif_bddness},
item~\eqref{it:Lip_est_standard1} in Lemma~\ref{lemma:Lip_est_standard}, Lemma~\ref{lemma:sg-estimate1}, 
item~\eqref{it:Y_unif_L4bdd} in Lemma~\ref{lemma:Y_unif_Lpbddness}, and \eqref{eq:Dpsi_est1} 
prove that for every finite $I\subseteq \H$ and every $t\in [0,T]$ it holds that
\begin{equation}\label{eq:stoch_int_vanish_5}
 \E
   \bigg[
   \Big|
   \int_{\round{t}{h}}^{t}
    \psi^I_{0,1 }\big(t,e^{(t-s)\bfA } Y_s^{h,I} \big) 
    \big( 
      e^{\delta(t)\bfA }
      \bfP_I B\big(Y_{\round{t}{h}}^{h,I}\big)
    \big)
  \,\mathrm{d}W_s
 \Big|^2
   \bigg]
 <\infty.
\end{equation}
Combining this, item~\eqref{it:Lip_est_standard1} in Lemma~\ref{lemma:Lip_est_standard}, Lemma~\ref{lemma:sg-estimate1}, \eqref{eq:deriv_est_B_and_F},~\eqref{eq:Dpsi_est1},~\eqref{eq:Dpsi_est2}, and~\eqref{eq:mild_Ito2}
ensures that for every finite $I\subseteq \H$ and every $t\in [0,T]$ it holds that
 \begin{align}
  \notag &  
  \big| 
    \E \big[
      \psi^I\big(t,Y_{t}^{h,I}\big)
      -
      \psi^I\big(t,e^{\delta(t)\bfA }Y_{\round{t}{h}}^{h,I} \big)
    \big]
  \big|
\\\notag  &
  \leq 
  h 
  \sup\!\Big\{
    \| v^I_{0,k}(s,x)\|_{L^{(k)}(\bfP_I(\bfH_0),\R))}
    \colon
    \substack{
      x\in \bfP_I (\bfH_0), \\
      s\in [0,T],\,k\in \{2,3,4\} 
    }
  \Big\}
  \| B|_{\bfP_I ( \bfH_0)} \|_{C^2_{\mathrm{b}}(\bfP_I (\bfH_0),L_2(U,\bfH_0))}^2 
\\\notag  & \quad \cdot 
  \left( 
    3 \| F|_{\bfP_I ( \bfH_0)} \|_{\Lip(\bfP_I (\bfH_0),\bfH_0)}
    +
    \tfrac{9}{2} \| B|_{\bfP_I ( \bfH_0)} \|_{\Lip(\bfP_I (\bfH_0),L_2(U,\bfH_0))}^2
  \right)
\\ & \quad \cdot
  \bigg[
  \sup_{s\in[0,T]} \E
  \big[\! 
    \max\!\big\{
      1 
      ,  
      \big\| Y_{\round{s}{h}}^{h,I} \big\|_{\bfH_0}^4
    \big\}
 \big]
 \bigg]
\\ \notag & \leq 
  \tfrac{15}2 h \,\mathfrak{m}^4  
  \sup\!\Big\{
    \| v^I_{0,k}(s,x)\|_{L^{(k)}(\bfP_I(\bfH_0),\R))}
    \colon
    \substack{
      x\in \bfP_I (\bfH_0), \\
      s\in [0,T],\,k\in \{2,3,4\} 
    }
  \Big\}
\\ \notag & \quad \cdot
  \bigg[
  \sup_{s\in[0,T]} \E
  \big[\! 
    \max\!\big\{
      1 
      ,  
      \big\| Y_{\round{s}{h}}^{h,I} \big\|_{\bfH_0}^4
    \big\}
 \big]
 \bigg].
 \end{align}
This implies that for every finite $I\subseteq \H$ it holds that
\begin{align}\label{eq:noise_term_4}
& \notag
  \Bigg|
    \E \Bigg[
      \sum_{u\in \U}
	\int_{0}^{T}
	  v^I_{0,2}\big(t, e^{\delta(t) \bfA } Y_{\round{t}{h}}^{h,I}\big) 
	  \big(
	    \bfP_I B\big( e^{\delta(t) \bfA } Y_{\round{t}{h}}^{h,I} \big) u
	    ,
	    \bfP_I B\big( e^{\delta(t) \bfA } Y_{\round{t}{h}}^{h,I} \big) u
	  \big)
	 \,\mathrm{d}t  
\\ & \quad \notag
	  - 
      \sum_{u\in \U}
	\int_{0}^{T}
	  v^I_{0,2}\big(t,Y_t^{h,I}\big)\big(
	    \bfP_I B\big(Y_{t}^{h,I} \big) u
	    ,
	    \bfP_I B\big(Y_{t}^{h,I} \big) u
	  \big)
	\,\mathrm{d}t
    \Bigg]
  \Bigg|
\\  &
 =
   \left|
    \int_{0}^{T}
    \E \big[
      \psi^I\big(t,Y_{t}^{h,I}\big)
      -
      \psi^I\big(t,e^{\delta(t)\bfA }Y_{\round{t}{h}}^{h,I} \big)
    \big]
    \,dt 
  \right|
\\ &  \notag  \leq 
 \tfrac{15}2 h T \mathfrak{m}^4
  \sup\!\Big\{
    \| v^I_{0,k}(t,x)\|_{L^{(k)}(\bfP_I(\bfH_0),\R))}
    \colon
    \substack{
      x\in \bfP_I (\bfH_0), \\
      t\in [0,T],\,k\in \{2,3,4\} 
    }
  \Big\}
\\ & \notag \quad 
  \cdot
  \bigg[
  \sup_{t\in[0,T]} \E
  \big[\! 
    \max\!\big\{
      1 
      ,  
      \big\| Y_{\round{t}{h}}^{h,I} \big\|_{\bfH_0}^4
    \big\}
 \big]
 \bigg]
.
\end{align}
%
%
Combining~\eqref{eq:fin_dim_weak_error},~\eqref{eq:split_det_term},~\eqref{eq:split_noise_term},~\eqref{eq:det_term_1},~\eqref{eq:det_term_2},~\eqref{eq:det_term_3},~\eqref{eq:det_term_4},~\eqref{eq:noise_term_1},~\eqref{eq:noise_term_2},~\eqref{eq:noise_term_3}, and~\eqref{eq:noise_term_4} ensures that for every finite $I\subseteq \H$ it holds that
\begin{align}
\notag
& 
\big| 
  \E \big[ \phi\big( Y_T^{0,I} \big) \big]
  -
  \E \big[ \phi\big( Y_T^{h,I} \big) \big] 
\big|
  \leq
  \max\{ h, h^{2(\gamma-\beta)} \} T 
\\  & 
\notag
  \cdot
  \sup\!\Big\{
    \| v^I_{0,k}(t,x)\|_{L^{(k)}(\bfP_I(\bfH_0),\R))}
    \colon
    \substack{
      x\in \bfP_I (\bfH_0),\, \\t\in [0,T],\,
      k\in \{1,2,3,4\}
    }
  \Big\}
\\ & \quad 
\notag 
%
%
 \cdot \bigg(
 2^{\nicefrac32} 
 \| F|_{\bfH_{\rho}} \|_{\Lip(\bfH_{\rho},\bfH_{2(\gamma-\beta)})} 
 \bigg[
 \sup_{t\in[0,T]} \E
  \big[\! 
    \max\!\big\{
      1 
      ,  
      \| Y_t^{h,I} \|_{\bfH_{\rho}} 
    \big\}
 \big]
 \bigg]
%
%
\\ & \qquad 
  + 
  17 \,\mathfrak{m}^4 
  \bigg[
  \sup_{t\in[0,T]} \E
  \big[\! 
    \max\!\big\{
      1 
      ,  
      \| Y_t^{h,I} \|_{\bfH_0}^4
    \big\}
 \big]
 \bigg]
%
%
\\ & \qquad 
\notag
  +
  2^{\nicefrac32}
  \left\| F \right\|_{\Lip(\bfH_{\beta-\gamma},\bfH_0)} 
  \bigg[
  \sup_{t\in [0,T]}
  \E \big[
    \|
      Y_{t}^{h,I}
    \|_{\bfH_{\gamma-\beta}}
  \big]
  \bigg]
%
\\ &  \qquad 
\notag 
  + 
  2^{\nicefrac52}
 \big\| \bfLambda^{-2\beta} \big\|_{L_1(\bfH_0)} 
 \| 
  B|_{\bfH_{\rho}}
 \|_{\Lip(\bfH_{\rho},L(U,H_{\gamma}))}^2
 \bigg[
 \sup_{t\in[0,T]} \E
  \big[\! 
    \max\!\big\{
      1 
      ,  
      \| Y_t^{h,I} \|_{\bfH_{\rho}}^2
    \big\}
 \big]
 \bigg]
%
%
\\ & \qquad 
\notag
  +  
    2^{\nicefrac52}
    \mathfrak{l}\,\mathfrak{c} 
  \bigg[
  \sup_{t\in [0,T]}
  \E 
  \big[\!
    \max\!\big\{ 
      1,
    \|
      Y_{t}^{h,I}
    \|_{\bfH_{\gamma-\beta}}^2
    \big\}
  \big]
  \bigg]
  \bigg).
\end{align}
Combining this, the estimates $\max\{h,h^{2(\gamma-\beta)}\}\leq\max\{T^{1-2(\gamma-\beta)},1\}h^{2(\gamma-\beta)}$ and $2^{\nicefrac52}\leq 6$, 
the fact that $\| \bfLambda^{-2\beta} \|_{L_1(\bfH_0)}=2\big[\sum_{h\in\H}|\lambda_h|^{-\beta}\big]$, 
and items~\eqref{it:Y_unif_L4bdd}--\eqref{it:Y_unif_L2bdd} in Lemma~\ref{lemma:Y_unif_Lpbddness}
establishes item~\eqref{it:fin_dim_weak_error}. 
The proof of Theorem~\ref{prop:fin_dim_weak_error} is thus completed.
\end{proof} 
\begin{corollary}\label{cor:weak_conv}
Assume Setting~\ref{setting:Part2}.
 Then 
 \begin{equation}\label{eq:weak_conv}
  \sup_{h\in (0,T]} \!
    \Big(
    h^{2(\beta-\gamma)}
    \big|
      \E \big[ \phi \big(Y_T^{0,\H} \big) \big]
      - 
      \E \big[ \phi \big(Y_T^{h,\H} \big) \big]
    \big|    
    \Big)
   <\infty.
 \end{equation}
\end{corollary}
\begin{proof}[Proof of Corollary~\ref{cor:weak_conv}.]
Throughout this proof let $I_n\subseteq \H$, $n\in \N$, 
be a non-decreasing sequence of finite sets which satisfies that
$\cup_{n\in \N} I_n = \H$. 
Observe that the triangle inequality implies that for every $n\in \N$, $h\in (0,T]$ it holds that
\begin{equation}\label{eq:weak_conv1}
 \begin{aligned}
 &
 \big| 
  \E \big[ \phi \big(Y_T^{0,\H} \big) \big] - \E \big[ \phi \big(Y_T^{h,\H}\big) \big]
 \big|
\\ & \leq 
 \big|
   \E \big[ \phi \big(Y_T^{0,\H}\big) \big] - \E \big[ \phi\big(Y^{0,I_n}_T\big)\big] 
 \big|
 +
 \big|
  \E \big[ \phi\big(Y^{0,I_n}_T\big)\big]  - \E \big[ \phi\big(Y^{h,I_n}_T\big)\big] 
 \big|
\\ & \quad 
 +
 \big|
  \E \big[ \phi\big(Y^{h,I_n}_T\big)\big]  - \E  \big[ \phi\big(Y^{h,\H}_T\big)\big] 
 \big|
\\ 
 & \leq 
 \bigg[ 
    \sup_{x\in \bfH_0} \| \phi'(x)\|_{L(\bfH_0,\R)}
 \bigg]
 \big( 
  \| Y_T^{0,\H} - Y^{0,I_n}_T \|_{\calL^2(\P;\bfH_0)}
  +
  \| Y_T^{h,\H} - Y^{h,I_n}_T \|_{\calL^2(\P;\bfH_0)}
 \big)
\\ & \quad
 +
 \big|
  \E \big[ \phi\big(Y^{0,I_n}_T\big) \big] - \E \big[ \phi\big(Y^{h,I_n}_T\big)\big] 
 \big|. 
 \end{aligned}
\end{equation}
Furthermore, note that Corollary~\ref{cor:strong_Galerkin_conv} ensures that for every $h\in (0,T]$ it holds that 
\begin{equation}\label{eq:weak_conv2}
 \limsup_{n\rightarrow \infty} 
 \big( 
  \| Y^{0,\H}_T - Y^{0,I_n}_T \|_{\calL^2(\P;\bfH_0)} 
  +
  \| Y_T^{h,\H} - Y^{h,I_n}_T \|_{\calL^2(\P;\bfH_0)}
 \big)
 =
 0.
\end{equation}
Moreover, observe that items~\eqref{it:Y_unif_L4bdd}--\eqref{it:Y_unif_L2bdd} in Lemma~\ref{lemma:Y_unif_Lpbddness} and items~\eqref{item:v_unif_bddness}--\eqref{it:fin_dim_weak_error} in Theorem~\ref{prop:fin_dim_weak_error} imply that
\begin{equation}
 \sup_{h \in (0,T]}\! 
 \left(
 h^{2(\beta-\gamma)}
 \bigg[
 \limsup_{n\rightarrow \infty}
 \big|
  \E \big[ \phi\big(Y^{0,I_n}_T\big) \big] - \E \big[ \phi\big(Y^{h,I_n}_T\big)\big] 
 \big|
 \bigg]
 \right)
 < \infty.
\end{equation}
This,~\eqref{eq:weak_conv1}, and~\eqref{eq:weak_conv2} imply~\eqref{eq:weak_conv}.
The proof of Corollary~\ref{cor:weak_conv} is thus completed.
\end{proof}
\section[Weak convergence rates for the hyperbolic Anderson model]{Weak convergence rates for temporal numerical approximations of the hyperbolic Anderson model}\label{sec:hyperbolic_anderson}
%
%
%
%
%
\subsection{Setting}\label{ssec:setting_hyperbolic_anderson}
Throughout this section we shall frequently use the following setting.
\begin{setting}\label{setting:Part3}
For every measure space $(\Omega,\calF,\mu)$, 
every measurable space $(S,\Sigma)$, 
every set 
$\mathcal O$, 
and every function $f\colon\mathcal O\rightarrow S$ let $[f]_{\mu,\Sigma}$ be the set given by
\begin{equation}\label{eq:setting_sec4_1}
\begin{aligned}
[ f ]_{\mu,\Sigma} 
 =
 \Big\{ g \colon \Omega \rightarrow S \colon 
 \left[
 \substack{ 
 [\exists\, A \in \calF \colon ( \mu(A) = 0 
 \text{ and }
 \{ \omega \in \Omega\cap\mathcal O \colon f(\omega) \neq g(\omega)\} 
 \subseteq A )]\\
 \text{ and }[\forall\, A\in \Sigma\colon g^{-1}(A)\in \calF]
 }
 \right]
 \Big\},
\end{aligned}
\end{equation}
let $\lambda\colon \calB((0,1))\rightarrow [0,1]$ be
the Lebesgue-Borel measure on $(0,1)$, 
for every 
$r\in [0,\infty)$, $p\in (1,\infty)$ let $(W^{r,p}((0,1),\R), \left\| \cdot \right\|_{W^{r,p}((0,1),\R)})$ be the 
Sobolev-Slobodeckij space
with smoothness parameter $r$ and integrability parameter $p$ 
of equivalence classes of 
$\calB((0,1))/\calB(\R)$-measurable functions, 
for every $r\in [0,2]$, $p\in (1,\infty)$ let 
$(\calW^{r,p}((0,1),\R), \left\|\cdot\right\|_{\calW^{r,p}((0,1),\R)})$ 
be the $\R$-Banach space which satisfies that
\begin{equation}
\calW^{r,p}((0,1),\R) 
 =
\begin{cases}
 W^{r,p}((0,1),\R) & \colon r\leq\nicefrac{1}{p}
 \\
 \left\{ 
  f \in W^{r,p}((0,1),\R) 
  \colon 
  \left[
  \substack{ 
  \exists\, g\in C([0,1],\R)\colon 
  (g|_{(0,1)}\in f \\
  \text{ and }g(0)=g(1)=0) 
  }
  \right]
 \right\} 
  &
 \colon r >\nicefrac{1}{p}
\end{cases}
\end{equation} 
and $[\forall \,v\in\calW^{r,p}((0,1),\R)\colon \|v\|_{\calW^{r,p}((0,1),\R)}=\|v\|_{W^{r,p}((0,1),\R)}]$, 
for every $p\in (1,\infty)$ let 
$(L^p(\lambda;\R),\left\|\cdot\right\|_{L^p(\lambda;\R)})$ be the $\R$-Banach space which satisfies that $(L^p(\lambda;\R),\left\|\cdot\right\|_{L^p(\lambda;\R)})=$ $(W^{0,p}((0,1),\R),\left\|\cdot\right\|_{W^{0,p}((0,1),\R)})$
and let 
$A_p\colon D(A_p) \subseteq L^p(\lambda; \R) \rightarrow L^p(\lambda; \R)$ be 
the linear operator which satisfies that $D(A_p) = \calW^{2,p}((0,1),\R)$ and 
$[\forall\, h\in D(A_p)\colon A_p(h)=\Delta h]$, for every $p\in (1,\infty)$
let $ ( V_{r,p} , \left\| \cdot \right\|_{V_{r,p}})$, $ r \in \R $, be
a family of interpolation spaces associated to $- A_p $, and
for every $\delta\in (0,1)$
let $(C^{\delta}([0,1],\R), \left\| \cdot \right\|_{C^{\delta}([0,1],\R)})$ be 
the space of $\delta$-H\"older continuous functions from $[0,1]$ to $\R$.
\end{setting}
%
\noindent
Note that  for every $p\in(1,\infty)$ it holds that $A_p$ is the Dirichlet Laplacian on $L^p(\lambda; \R)$. 
The relationship between the spaces $\mathcal{W}^{2r,p}((0,1),\R)$ and $V_{r,p}$, $r\in (0,1)$, $p\in (1,\infty)$, is discussed in Lemma~\ref{lemma:eqD(A)W} below. 
\subsection{Preparatory lemmas}\label{ssec:prepHA}
Various results closely related to Lemmas~\ref{lemma:eqD(A)W}--\ref{lemma:multiplication_HS} below are
available in the literature; see, e.g., Lemarie-Rieusset \& Gala~\cite[Lemma 1]{LemarieRieussetGala:2006} for a result closely related to
Lemmas~\ref{lemma:multiplication_on_Hoelder1} and~\ref{lemma:multiplication_on_Hoelder2} below. We provide these lemmas in the exact form that we need.
\begin{lemma}\label{lemma:eqD(A)W}
Assume Setting~\ref{setting:Part3}. Then 
\begin{enumerate}[(i)]
 \item\label{item:Sobolev_Hilbert1} it holds for every $r\in (0,1)\backslash \{ \nicefrac{1}{4} \}$ that 
 $V_{r,2} \subseteq \calW^{2r,2}((0,1),\R)$ continuously, 
 \item\label{item:Sobolev_Hilbert2} it holds for every $r\in (0,1)\backslash \{ \nicefrac{1}{4} \}$ that 
 $\calW^{2r,2}((0,1),\R) \subseteq V_{r,2}$ continuously, and
 \item\label{item:Sobolev_Banach1} it holds for every $p\in (1,\infty)$,
 $r\in (0,1)$,
 $s\in [0,r)$ 
 that
 $V_{r,p} \subseteq \calW^{2s,p}((0,1),\R)$
 continuously.
\end{enumerate}
\end{lemma}
\begin{proof}[Proof of Lemma~\ref{lemma:eqD(A)W}.]
First, note that, e.g., 
Triebel~\cite[Theorem~1.15.3, Definition~2.3.1/1, item~(d) in Theorem~2.3.2, Definition~4.2.1/1,
Definition~4.3.3/2, 
equation~(7) in Theorem~4.3.3, item~(b) in Theorem~4.9.1, and item~(b) in Theorem~5.5.1]{Triebel:1978} 
(with $k=1$, $B_1=\id_{\mathbbm C^{\{0,1\}}}$, $m_1=0$, $m=2$, $p=2$, $\theta=r$ for $r\in(0,1)\setminus\{\nicefrac14\}$ in the notation of 
\cite[Definition~4.3.3/2 and 
equation~(7) in Theorem~4.3.3]{Triebel:1978})  
implies that for every $r\in (0,1)\backslash \{ \nicefrac{1}{4} \}$ it holds that 
$
 V_{r,2}
 \subseteq 
 \calW^{2r,2}((0,1),\R)
 \subseteq V_{r,2} 
$
continuously 
(cf.\ Triebel~\cite[Definition~2.3.1/1 and Definition~4.2.1/1]{Triebel:1978} for a definition of $W^{r,p}((0,1),\C)$, $r\in [0,\infty)$, $p\in (1,\infty)$, and cf.\ Triebel~\cite[Section~4.2.4, Remark~2 in Section~4.4.1, and Remark~2 in Section~4.4.2]{Triebel:1978} for equivalent definitions of $W^{r,p}((0,1),\C)$, $r\in [0,\infty)$, $p\in (1,\infty)$). 
This proves items~\eqref{item:Sobolev_Hilbert1} and \eqref{item:Sobolev_Hilbert2}.
Next observe that, e.g., 
Triebel~\cite[Theorem~1.15.3, 
Definition~4.2.1/1, Definition~4.3.3/2, 
equation~(7) 
in Theorem~4.3.3, items~(a)--(b) in Theorem~4.6.1, item~(b) in Theorem~4.9.1, item~(c) in Theorem~5.4.4/1, and item~(b) in Theorem~5.5.1]{Triebel:1978} 
(with $k=1$, $B_1=\id_{\mathbbm C^{\{0,1\}}}$, $m_1=0$, $m=2$, 
$p=p$, $\theta=r-(\nicefrac\eps2)\one_{\{\nicefrac1{(2p)}\}}(r)$ 
for $p\in(1,\infty)$, $\eps\in(0,r]$, $r\in(0,1)$
in the notation of \cite[Definition~4.3.3/2 and  
equation~(7) 
in Theorem~4.3.3]{Triebel:1978})  
ensures that for every $p\in (1,\infty)$, $r\in (0,1)$, $\eps\in (0,r]$ it holds that 
$
V_{r,p}\subseteq \calW^{2(r-\eps)}((0,1),\R) $
continuously. This establishes item~\eqref{item:Sobolev_Banach1}. 
The proof of Lemma~\ref{lemma:eqD(A)W} is thus completed.
\end{proof}
%
%
\begin{lemma}\label{lemma:multiplication_on_Hoelder1}
Assume Setting~\ref{setting:Part3} and let $r\in [0,\nicefrac{1}{2})\backslash\{\nicefrac{1}{4}\}$, 
$\delta\in (2r,1)$.
Then 
\begin{enumerate}[(i)]
 \item\label{item:inV} it holds for every $f\in V_{r,2}$, $v\in C^{\delta}([0,1],\R)$
 that $fv \in V_{r,2}$ and 
 \item it holds that \label{item:multest1} 
\begin{equation}\label{eq:Afv_est1}
\begin{aligned}
& 
 \sup\!\Big\{
 \tfrac{
  \| 
    fv 
  \|_{ V_{r,2}}
 }
 {
  \| 
    f 
  \|_{V_{r,2}}
  \| v \|_{C^{\delta}([0,1],\R)}
 }
 \colon
 \substack{f\in V_{ r,2}\backslash \{0\},\,\\v\in C^{\delta}([0,1],\R)\backslash \{0\}}
 \Big\}
\\& \leq
 \tfrac{\sqrt{3}}
 {\sqrt{\delta - 2r}}
 \bigg[
 \sup_{w\in V_{ r,2}\backslash \{0\}}
 \tfrac{ 
  \| 
    w 
  \|_{ V_{r,2}}
 }
 {\| w \|_{W^{2r,2}((0,1),\R)}}
 \bigg]
 \bigg[
 \sup_{w\in V_{r,2}\backslash \{0\}}
 \tfrac{\| w \|_{W^{2r,2}((0,1),\R)}}
 { 
  \| 
    w 
  \|_{V_{r,2}}
 }
 \bigg]
 < \infty.
\end{aligned}
\end{equation}
\end{enumerate}
\end{lemma}

\begin{proof}[Proof of Lemma~\ref{lemma:multiplication_on_Hoelder1}.]
To prove items (i) and (ii) we distinguish between the case $r=0$ and the case $r>0$. 
We first prove items (i) and (ii) in the case $r=0$. 
Observe that the fact that for every $w\in C([0,1],\R)$ it holds that 
\begin{equation} 
\| w \|_{C^{\delta}([0,1],\R)} = \sup_{x\in [0,1]} |w(x)| + \sup_{x,y\in [0,1],\, x\neq y} 
\Big( 
    \tfrac{|w(x)-w(y)|}{|x-y|^{\delta}}
\Big)
\end{equation}
establishes items (i) and (ii) in the case $r=0$.
Next we prove items (i) and (ii) in the case $r>0$.
Note that item~\eqref{item:Sobolev_Hilbert1} in Lemma~\ref{lemma:eqD(A)W} and (23) in Jentzen \& R\"ockner~\cite{JentzenRoeckner:2012} imply that 
for every 
$f\in V_{r,2}$, $v\in C^{\delta}([0,1],\R)$ it holds that 
$f\in \calW^{2r,2}((0,1),\R)$ and $fv\in \calW^{2r,2}((0,1),\R)$. 
Combining this and item~\eqref{item:Sobolev_Hilbert2} in Lemma~\ref{lemma:eqD(A)W} establishes item~\eqref{item:inV} 
in the case $r>0$. 
Moreover, observe that (23) in Jentzen~\& R\"ockner~\cite{JentzenRoeckner:2012} assures that
\begin{equation}\label{eq:JR}
 \sup\!
 \Big\{
 \tfrac{\| f v \|_{W^{2r,2}((0,1),\R)}}
 {
 \| f \|_{W^{2r,2}((0,1),\R)}
 \| v \|_{C^{\delta}([0,1],\R)} 
 }
 \colon
 \substack{f\in W^{2r,2}((0,1),\R)\backslash \{0\},\,  \\ v\in C^{\delta}([0,1],\R)\backslash \{0\}}
 \Big\}
 \leq 
 \tfrac{\sqrt{3}}
 {\sqrt{\delta - 2r}}.
\end{equation}
Combining this and items~\eqref{item:Sobolev_Hilbert1}--\eqref{item:Sobolev_Hilbert2} in Lemma~\ref{lemma:eqD(A)W} establishes item~\eqref{item:multest1}  
in the case $r>0$.  
The proof of Lemma~\ref{lemma:multiplication_on_Hoelder1} is thus completed.
\end{proof}
\begin{lemma}\label{lemma:multiplication_on_Hoelder2}
Assume Setting~\ref{setting:Part3} and let $r\in (0,\nicefrac{1}{2})\backslash\{\nicefrac{1}{4}\}$, 
$\delta\in (2r,1)$. 
Then 
\begin{equation}\label{eq:Afv_est2}
\begin{aligned}
& \sup\!
\Big\{ 
 \tfrac{
  \| 
    fv 
  \|_{V_{-r,2}}
 }
 {
  \| 
    f 
  \|_{V_{-r,2}}
  \| v \|_{C^{\delta}([0,1],\R)}
 }
 \colon
  \substack{f\in L^2(\lambda;\R)\backslash \{0\},\,\\v\in C^{\delta}([0,1],\R)\backslash \{0\}}
  \Big\}
\\ & \leq
 \tfrac{\sqrt{3}}
 {\sqrt{\delta - 2r}}
 \bigg[
 \sup_{w\in V_{ r,2}\backslash \{0\}}
 \tfrac{ 
  \| 
     w
  \|_{V_{r,2}}
 }
 {\| w \|_{W^{2r,2}((0,1),\R)}}
 \bigg]
 \bigg[
 \sup_{w\in V_{ r, 2}\backslash\{0\}}
 \tfrac{\| w \|_{W^{2r,2}((0,1),\R)}}{ \| w \|_{V_{r,2}}}
 \bigg]
 < \infty.
\end{aligned}
\end{equation}
\end{lemma}

\begin{proof}[Proof of Lemma~\ref{lemma:multiplication_on_Hoelder2}]
First, note that Lemma~\ref{lemma:multiplication_on_Hoelder1}
proves that for every $u\in L^2(\lambda;\R)$, $v\in C^{\delta}([0,1],\R)$ it holds that $v(-A_2)^{-r} u \in V_{r,2}$. 
The fact that for every $f\in L^2(\lambda; \R )$, $v\in C^{\delta}([0,1],\R)$ it holds 
that $fv\in L^2(\lambda; \R )$ and the self-adjointness of $L^2(\lambda;\R)\ni v\mapsto  (-A_2)^{-r}v\in L^2(\lambda;\R)$ therefore imply that  
for every $f\in L^2(\lambda;\R)$, $v\in C^{\delta}([0,1],\R)$ it holds that 
\begin{equation}\label{eq:dual}
\begin{aligned}
 \|
  fv
 \|_{V_{-r,2}}
&  = 
 \| 
  (-A_2)^{-r} (fv)
  \|_{L^2(\lambda;\R)} 
 =
 \sup_{u\in L^2(\lambda; \R )\backslash \{0\}}\!
 \tfrac{
  \langle
      (-A_2)^{-r} u,
      fv
  \rangle_{L^2(\lambda; \R )}
 }{
  \| u \|_{L^2(\lambda; \R )}
 }
 \\
 & = 
 \sup_{u\in L^2(\lambda; \R )\backslash \{0\}}\!
 \tfrac{
  \langle
      (-A_2)^{r}
      ( v (-A_2)^{-r} u ),
      (-A_2)^{-r} f
  \rangle_{L^2(\lambda; \R )}
 }{
  \| u \|_{L^2(\lambda; \R )}
 }
 \\
 & 
 \leq 
 \bigg[
 \sup_{u\in L^2(\lambda; \R )\backslash\{0\}}
 \tfrac{
    \|
       v (-A_2)^{-r} u 
    \|_{V_{r,2}}
  }{
    \| u \|_{L^2(\lambda; \R )}
  }
  \bigg]
  \left\|
    f
  \right\|_{V_{-r,2}}.
\end{aligned}
\end{equation}
Combining this 
and 
Lemma~\ref{lemma:multiplication_on_Hoelder1} \color{black}
(with $v=v$, $f=(-A_2)^{-r}u$ in the notation of Lemma~\ref{lemma:multiplication_on_Hoelder1})
establishes~\eqref{eq:Afv_est2}.
The proof of Lemma~\ref{lemma:multiplication_on_Hoelder2} is thus completed.
\end{proof}
\begin{lemma}\label{lemma:multiplication_HS}
Assume Setting~\ref{setting:Part3}, 
for every $\R$-Hilbert space $(H,\langle\cdot,\cdot\rangle_H,\left\|\cdot\right\|_H)$ let $(L_2(H),$\linebreak$\langle\cdot,\cdot\rangle_{L_2(H)},\left\|\cdot\right\|_{L_2(H)})$ be the $\R$-Hilbert space of Hilbert-Schmidt operators from $H$ to $H$, 
let 
$r \in (-\nicefrac{1}{4},\nicefrac{1}{4})$, and for every 
$m\in V_{\max\{0,r\},2}$ let $M_m\colon D(M_m) \rightarrow L^2(\lambda; \R )$
be the linear operator which satisfies that
$ 
 D(M_m) = \{ h \in L^2( \lambda; \R) \colon mh \in L^2( \lambda; \R)\}
$
and $[\forall\, h\in D(M_m) \colon M_m h = mh]$.
Then  
\begin{enumerate}[(i)]
 \item\label{item:inV2} it holds for every $m\in V_{\max\{0,r\},2}$, $h\in L^2(\lambda; \R )$ 
 that $(-A_2)^{-\nicefrac{1}{2}}h\in D(M_m)$ and $M_m (-A_2)^{-\nicefrac{1}{2}}h \in V_{\max\{0,r\},2}$ and
\item\label{item:HSest} it holds that
$
  \sup_{m\in V_{\max\{0,r\},2}\backslash \{0\}} 
  \Big[
  \frac{ 
    \| 
      (-A_2)^{r} M_m (-A_2)^{-\nicefrac{1}{2}}
    \|_{L_2(L^2(\lambda; \R ))} 
  }
  {  \| m \|_{V_{r,2}} }
  \Big]
  <
  \infty
$.
\end{enumerate}
\end{lemma}
\begin{proof}[Proof of Lemma~\ref{lemma:multiplication_HS}.]
Throughout this proof let 
$\varrho=\max\{0,r\}$, 
$\eps\in (0,\frac14-\varrho)$, let 
$e_n\colon [0,1]\rightarrow \R$, $n\in \N$, be the functions which satisfy for every $ n \in \N $, $x\in [0,1]$ that 
$ e_n(x) = \sqrt{2} \sin( n \pi x )$, let $(\Omega,\calF,\P)$ be a probability space, let $\gamma_n\colon \Omega\rightarrow \R$, $n\in \N$, be independent standard Gaussian 
random variables, and 
let $K_{p}\in [1,\infty]$, $p\in [1,\infty)$,  
be the extended real numbers which satisfy for every $p\in [1,\infty)$ that
\begin{equation}\label{eq:KK}
 K_{p} 
 = 
 \sup\!\bigg\{
  \tfrac{ 
   \left(
    \E\left[ 
      \left|
	\sum_{k=1}^{n} \gamma_k x_k
      \right|^p
     \right]
    \right)^{\!\nicefrac{1}{p}}
  }{
  \left(
    \E\left[ 
      \left|
	\sum_{k=1}^{n} \gamma_k x_k
      \right|^2
     \right]
    \right)^{\!\nicefrac{1}{2}}
  }
  \colon 
  n\in \N, x_1,\ldots,x_n\in \R\backslash \{0\} 
 \bigg\}.
\end{equation}
Observe that the Khintchine inequalities imply that for every $p\in [1,\infty)$ it holds that $K_p<\infty$.
Moreover, note that item~\eqref{item:Sobolev_Hilbert1} in Lemma~\ref{lemma:eqD(A)W} and the fractional Sobolev inequalities prove that for every 
$h\in L^2(\lambda; \R )$, $\delta \in (0,\nicefrac{1}{2})$ it holds 
that there exists a $v\in  C^{\delta}([0,1],\R)$ such that $(-A_2)^{-\nicefrac{1}{2}}h = [ v ]_{\lambda,\calB(\R)}$. 
Lemma~\ref{lemma:multiplication_on_Hoelder1} and the fact 
that for every $f\in \calL^2(\lambda;\R)$, $v\in C^{\delta}([0,1],\R)$, $v_1,v_2\in [ v ]_{\lambda,\calB(\R)}$
it holds that 
\begin{equation} 
 [v_1 f]_{\lambda,\calB(\R)} = [v_2 f]_{\lambda,\calB(\R)} 
\end{equation}
hence 
imply that for every $h\in L^2(\lambda; \R )$, $m\in V_{\varrho,2}$
it holds that $M_m (-A_2)^{-\nicefrac{1}{2}}h \in V_{\varrho,2}$. 
This establishes item~\eqref{item:inV2}.
Furthermore, observe that for every $p\in (1,\infty)$ it holds that 
\begin{align}\label{eq:g_rad_HS3}
 & \notag
  \sup_{n\in \N}
    \bigg(
      \E \bigg[ 
	\Big\|
	\textstyle\!
	\sum\limits_{k=1}^{n} (k\pi)^{-1}
	\displaystyle
	\gamma_k e_k 
      \Big\|^2_{C^{2\varrho+\eps}([0,1], \R )}
     \bigg]
   \bigg)^{\nicefrac{1}{2}}
\leq 
  \sup\!\Big\{ 
    \tfrac{ \| v \|_{C^{2\varrho+\eps}([0,1],\R) }}
    {\| [v]_{\lambda,\calB(\R)} \|_{ V_{\varrho+\eps,p}}}
    \colon 
    \substack{ v\in C^{2}([0,1],\R)\backslash\{0\},\\  v(0)=v(1)=0} 
  \Big\}
\\ &  \cdot
  \sup_{n\in \N}
    \bigg(
      \E \bigg[ 
	\Big\|
	\textstyle\!
	\sum\limits_{k=1}^{n}
	\displaystyle
	(k\pi)^{-1+2(\varrho+\eps)} \gamma_k e_k 
      \Big\|^2_{\calL^p(\lambda; \R )}
     \bigg]
   \bigg)^{\nicefrac{1}{2}}.
  \end{align}
In addition, note that item~\eqref{item:Sobolev_Banach1} in Lemma~\ref{lemma:eqD(A)W} and the fractional Sobolev inequalities demonstrate that for every $p\in (\eps^{-1},\infty)$ it holds that 
\begin{equation} \label{eq:g_rad_HS4}
  \sup\!\Big\{ 
    \tfrac{ \| v \|_{C^{2\varrho+\eps}([0,1],\R) }}
    { \| [v]_{\lambda,\calB(\R)} \|_{V_{\varrho+\eps,p}}}
    \colon 
    \substack{ v\in C^{2}([0,1],\R)\backslash\{0\},\\  v(0)=v(1)=0} 
  \Big\}<\infty.
\end{equation}
Moreover, note that H\"older's inequality, Fubini's 
theorem, and~\eqref{eq:KK} imply that for every $p\in (\eps^{-1},\infty)$ it holds  that 
\begin{align}\label{eq:g_rad_HS5}
& 
    \sup_{n\in\N}\bigg(
    \E \bigg[ 
	\Big\|
	\textstyle\!
	\sum\limits_{k=1}^{n} 
	\displaystyle
	(k\pi)^{-1+2(\varrho+\eps)} \gamma_k e_k 
      \Big\|^2_{\calL^p(\lambda; \R )}
     \bigg]\bigg)^{\nicefrac12}
\notag
\\ & 
=
    \sup_{n\in\N}\Bigg(
    \E \Bigg[ 
	\bigg(
	\textstyle\!
	\int\limits_0^1
	\Big|
	\sum\limits_{k=1}^{n} 
	\displaystyle
	(k\pi)^{-1+2(\varrho+\eps)} \gamma_k e_k(x) 
     \Big|^{p}dx\bigg)^{\nicefrac2p}
     \Bigg]\Bigg)^{\nicefrac12}
\notag
\\ & 
\leq
    \sup_{n\in\N}
    \bigg(
      \E \bigg[\, 
	\textstyle\!
	\int\limits_0^1
	\Big|
	\sum\limits_{k=1}^{n} 
	\displaystyle
	(k\pi)^{-1+2(\varrho+\eps)} \gamma_k e_k(x) 
     \Big|^{p}dx
     \bigg]
   \bigg)^{\nicefrac{1}{p}} 
\notag
\\ & 
=
    \sup_{n\in\N}
    \bigg( 
	\textstyle\!
	\int\limits_0^1
	\E \bigg[
	\Big|
	\sum\limits_{k=1}^{n} 
	\displaystyle
	(k\pi)^{-1+2(\varrho+\eps)} \gamma_k e_k(x) 
     \Big|^{p}
     \bigg]dx
   \bigg)^{\nicefrac{1}{p}} 
\\ & 
\leq K_p
    \sup_{n\in\N}\Bigg( 
	\textstyle\!
	\int\limits_0^1
	\bigg(\E \bigg[
	\Big|
	\sum\limits_{k=1}^{n} 
	\displaystyle
	(k\pi)^{-1+2(\varrho+\eps)} \gamma_k e_k(x) 
     \Big|^{2}
     \bigg]\bigg)^{\nicefrac{p}2}dx
   \Bigg)^{\nicefrac{1}{p}} 
\notag
\\ & 
= K_p
    \sup_{n\in\N}\Bigg( 
	\textstyle\!
	\int\limits_0^1
	\bigg(
	\sum\limits_{k=1}^{n} 
	\displaystyle
	(k\pi)^{-2+4(\varrho+\eps)} |e_k(x)|^2 
     \bigg)^{\nicefrac{p}2}dx
   \Bigg)^{\nicefrac{1}{p}}      
\notag
\\ & 
  \leq \sqrt{2}K_{p}
  \sup_{n\in\N}\bigg(
	\textstyle\!
	\sum\limits_{k=1}^{n} 
	\displaystyle
	(k\pi)^{-2+4(\varrho+\eps)}
  \bigg)^{\nicefrac12}
  = \sqrt{2}K_{p}
  \bigg(
	\textstyle\!
	\sum\limits_{k=1}^{\infty} 
	\displaystyle
	(k\pi)^{-2+4(\varrho+\eps)}
  \bigg)^{\nicefrac12}
  <\infty.
\notag
 \end{align}
Next observe that
for every $m\in V_{\varrho,2}$ it holds that 
\begin{align}
&\big\| (-A_2)^{r} M_m (-A_2)^{-\nicefrac{1}{2}}\big\|_{L_2(L^2(\lambda; \R ))}^2 
  = \sup_{n\in \N}
  \bigg(
	\textstyle\!
	\sum\limits_{k=1}^{n} 
	\displaystyle
    \big\|
	  M_m (-A_2)^{-\nicefrac{1}{2}} [e_k]_{\lambda,\calB(\R)} 
      \big\|^2_{V_{r,2}}
    \bigg)    
\notag
\\ & \label{eq:g_rad_HS1}
  = \sup_{n\in \N}
    \E \bigg[ 
      \Big\|	
	\textstyle\!
	\sum\limits_{k=1}^{n} 
	\displaystyle
	  \gamma_k M_m (-A_2)^{-\nicefrac{1}{2}} [e_k]_{\lambda,\calB(\R)} 
      \Big\|^2_{V_{r,2}}
    \bigg]
\\ &
  = \sup_{n\in \N}
    \E \bigg[ 
      \Big\|
	m 
	\textstyle\!
	\sum\limits_{k=1}^{n}
	\displaystyle 
	(k\pi)^{-1} \gamma_k e_k
      \Big\|^2_{V_{r,2}}
    \bigg].
\notag
\end{align}
Moreover, note that Lemma~\ref{lemma:multiplication_on_Hoelder1}, Lemma~\ref{lemma:multiplication_on_Hoelder2},~\eqref{eq:g_rad_HS3},~\eqref{eq:g_rad_HS4}, and~\eqref{eq:g_rad_HS5} imply that
\begin{equation}\label{eq:g_rad_HS2}
\begin{aligned}
& 
\sup_{m\in V_{\varrho,2}\backslash \{0\}}
\Bigg\{ 
   \tfrac{
    \sup_{n\in \N}
      \left(
	\E \left[ 
	  \left\|
	    m \sum_{k=1}^{n} (k\pi)^{-1} \gamma_k e_k 
	  \right\|^2_{V_{r,2}}	
	\right]
    \right)^{\nicefrac{1}{2}}
   }
   {
    \| m \|_{V_{r,2}}
   }
\Bigg\}
\\& 
 \leq   
 \sqrt{\tfrac{3}{\eps}} 
 \bigg[ 
    \sup_{w\in V_{\varrho,2}\backslash \{0\}}
 \tfrac{ 
  \|  w \|_{V_{\varrho,2}}
 }
 {\| w \|_{W^{2\varrho,2}((0,1),\R)}}
 \bigg]
 \bigg[
 \sup_{w\in V_{\varrho, 2}\backslash\{0\}}
 \tfrac{\| w \|_{W^{2\varrho,2}((0,1),\R)}}
 { \|  w \|_{V_{\varrho,2}}}
 \bigg]
\\ & \quad 
 \cdot
 \sup_{n\in \N}
    \bigg(
      \E \bigg[ 
	\Big\|
	\textstyle\!
	\sum\limits_{k=1}^{n} 
	\displaystyle
	(k\pi)^{-1} \gamma_k e_k 
      \Big\|^2_{C^{2\varrho+\eps}([0,1], \R )}
     \bigg]
   \bigg)^{\nicefrac{1}{2}} 
 < \infty. 
\end{aligned}
\end{equation}
Combining this and~\eqref{eq:g_rad_HS1} establishes item~\eqref{item:HSest}.
The proof of Lemma~\ref{lemma:multiplication_HS} is thus completed.
\end{proof}
\subsection{The hyperbolic Anderson model}\label{ssec:HA}
\begin{corollary}\label{cor:hyperbolic_anderson}
For every pair of $\R$-Hilbert spaces $ ( \mathcal{V} , \langle \cdot , \cdot \rangle_{ \mathcal{V} } , \left\| \cdot\right\|_{ \mathcal{V}  } ) $ and 
$ ( \mathcal{W} , \langle \cdot , \cdot \rangle_{ \mathcal{W} } ,$ $ \left\| \cdot\right\|_{ \mathcal{W} } ) $ 
let $(L_2(\mathcal{V},\mathcal{W}), \langle \cdot,\cdot\rangle_{L_2(\mathcal{V},\mathcal{W})},\left\| \cdot \right\|_{L_2(\mathcal{V},\mathcal{W})} )$ be the $\R$-Hilbert space 
of Hilbert-Schmidt operators from $\mathcal{V}$ to $\mathcal{W}$,
for every measure space $(\Omega,\calF,\mu)$,
every measurable space $(S,\Sigma)$, 
every set 
$\mathcal O$, 
and every function $f\colon\mathcal O\rightarrow S$ let $[f]_{\mu,\Sigma}$ be the set given by
\begin{equation}
\begin{aligned}
[ f ]_{\mu,\Sigma} 
 =
 \Big\{ g \colon \Omega \rightarrow S \colon 
 \left[
 \substack{ 
 [\exists\, A \in \calF \colon ( \mu(A) = 0 
 \text{ and }
 \{ \omega \in \Omega\cap\mathcal O \colon f(\omega) \neq g(\omega)\} 
 \subseteq A )]\\
 \text{ and }[\forall\, A\in \Sigma\colon g^{-1}(A)\in \calF]
 }
 \right]
 \Big\},
\end{aligned}
\end{equation}
let $ T, \vartheta \in ( 0 , \infty ) $, $ b_0$, $b_1 \in \R$, 
let $\lambda\colon \calB((0,1))\rightarrow [0,1]$ be 
the Lebesgue-Borel measure on $(0,1)$, 
let $(H,\langle\cdot,\cdot\rangle_H,\left\|\cdot\right\|_H)$ be the $\R$-Hilbert space given by $(H,\langle\cdot,\cdot\rangle_H,\left\|\cdot\right\|_H)=(L^2(\lambda;\R),\langle\cdot,\cdot\rangle_{L^2(\lambda;\R)},
$\linebreak$
\left\|\cdot\right\|_{L^2(\lambda;\R)})$, 
let $ ( \Omega , \calF , \P, (\mathbbm{F}_t )_{ t \in [ 0 , T ] } ) $ be a filtered probability space which fulfills the usual conditions, 
let $ ( W_t )_{ t \in [ 0 , T ] } $ be an $\id_H$-cylindrical $ ( \mathbbm{F}_t )_{ t \in [ 0 , T ] } $-Wiener process,
for every $n\in\N$ let $e_n\in H$ satisfy 
$ e_n =[(\sqrt{2} \sin( n \pi x ))_{x\in (0,1)}]_{\lambda,\calB(\R)}$, 
let $A\colon D(A) \subseteq H \rightarrow H$
satisfy 
$D(A) = 
\left\{ h \in H \colon\sum_{n=1}^{\infty}|(n \pi)^2 \langle e_n, h \rangle_{H}|^2  < \infty \right\}
$ 
and 
$\big[\forall\, h\in D(A) \colon Ah=  \sum_{n=1}^{\infty} -(n\pi)^2 \langle e_n, h \rangle_{H}e_n\big]$,
let $ ( H_{r} , \langle \cdot, \cdot \rangle_{H_r}, \left\| \cdot \right\|_{H_{r}})$, $ r \in \R $, 
be a family of interpolation spaces associated to $- A $, 
for every $r\in\R$ let $ ( \bfH_r ,\langle \cdot , \cdot \rangle_{ \bfH_r } , \left\| \cdot\right\|_{ \bfH_r } )$ be the $ \R $-Hilbert space 
which satisfies  
$ 
( \bfH_r , \langle \cdot , \cdot \rangle_{ \bfH_r } ,$ $\left\| \cdot\right\|_{ \bfH_r } ) 
= $ $
\bigl( H_{ \nicefrac{r}{2} } \times H_{ \nicefrac{r}{2} - \nicefrac{1}{2} },$ $ 
\langle \cdot, \cdot \rangle_{ H_{ \nicefrac{r}{2} } \times H_{ \nicefrac{r}{2} - \nicefrac{1}{2} } }, 
\norm{ \cdot }_{ H_{ \nicefrac{r}{2} } \times H_{ \nicefrac{r}{2} - \nicefrac{1}{2} } } \bigl) 
$, 
let 
$ \bfA \colon D ( \bfA ) \subseteq \bfH_0
\to \bfH_0 $ 
satisfy 
$ D ( \bfA ) = \bfH_1 $ 
and 
$\left[ \forall\, (v,w) \in \bfH_1 \colon \bfA( v , w ) = ( w , \vartheta A v ) \right]$, 
let 
$ \phi \in C^4( \bfH_0, \R ) $ 
satisfy for every $k\in \{1,2,3,4\}$ that $\sup_{x\in \bfH_0}\| \phi^{(k)}(x) \|_{L^{(k)}(\bfH_0,\R)} < \infty$, 
let $ \xi \in \calL^6 ( \P \vert_{\mathbbm{F}_0} ; \bfH_{\nicefrac12} ) $, 
$B\colon H\rightarrow L_2(H,H_{-\nicefrac{1}{2}})$ 
satisfy 
for every 
$v\in \calL^2(\lambda;\R)$, $u\in C([0,1],\R)$
that  
$
  B([v]_{\lambda,\,\calB(\R)}) [u]_{\lambda,\,\calB(\R)} 
  = $ $[ ( (b_0 + b_1 v(x))u(x))_{x\in (0,1)} ]_{\lambda,\,\calB(\R)} 
$, 
let $\bfB \colon \bfH_0 \to L_2( H, \bfH_0 ) $
satisfy 
for every 
$ (v,w) \in \bfH_0 $, $u\in H$ that 
$ \bfB( v, w ) u  = \bigl( 0, B(v)u \bigr)$, 
let $ X\colon [0,T] \times \Omega \rightarrow  \bfH_0$
be an $(\mathbbm{F}_t)_{t\in[0,T]}$-predictable stochastic process which satisfies 
for every $t\in [0,T]$ that 
$\sup_{s\in[0,T]}\E\big[\| X_s \|_{\bfH_0}^2\big] < \infty$  
and 
\begin{align}\label{eq:X_hyperbolic_Anderson}
 [X_t]_{\P,\,\calB(\bfH_0)} & 
 =
 \left[e^{t\bfA } \xi\right]_{\P,\,\calB(\bfH_0)}
 + \int_{0}^{t} e^{ (t-s)\bfA} \bfB ( X_s  )\,\mathrm{d}W_s,
\end{align} 
for every $h\in (0,T]$ let $\lfloor \cdot \rfloor_h \colon [0,\infty) \rightarrow \R$ satisfy for every $x\in [0,\infty)$ that 
\begin{equation}
\lfloor x \rfloor_h 
=
\max( \{0, h, 2h, 3h, \ldots\} \cap [0,x] ),
\end{equation} 
and for every $h\in(0,T]$ let $Y^{h}\colon [0,T] \times \Omega \rightarrow \bfH_0$ 
be an $(\mathbbm{F}_t)_{t\in[0,T]}$-predictable stochastic processes which satisfies 
for every $t\in [0,T]$  that 
$\sup_{s\in[0,T]}\E\big[\| Y^h_s \|_{\bfH_0}^2\big] < \infty$  
and 
\begin{align}\label{eq:Y_mild_cor2}
 [Y_t^{h}]_{\P,\,\calB(\bfH_0)} 
 =
 \left[e^{t\bfA }\xi\right]_{\P,\,\calB(\bfH_0)}
 +
 \int_{0}^{t} e^{ (t-\round{s}{h})\bfA } \bfB\big( Y^{h}_{\round{s}{h}} \big) \,\mathrm{d}W_s.
\end{align}
Then it holds for every $\eps\in (0,\infty)$ that 
\begin{equation}\label{eq:weakrates}
  \sup_{h\in (0,T]} \!
  \Big(
    h^{\eps-1}
     \big|
      \E\big[ \phi\big(X_T\big) \big] 
      - 
      \E\big[ \phi\big(Y_T^h\big)\big] 
     \big|
  \Big)
   < \infty.
\end{equation}
\end{corollary}
\begin{proof}[Proof of Corollary~\ref{cor:hyperbolic_anderson}.]
Throughout this proof 
for every $\R$-Hilbert space $(V,\langle\cdot,\cdot\rangle_V,\left\|\cdot\right\|_V)$ let $(L_2(V),\langle\cdot,\cdot\rangle_{L_2(V)},\left\|\cdot\right\|_{L_2(V)})$ be the $\R$-Hilbert space of Hilbert-Schmidt operators from $V$ to $V$, 
for every pair of $ \R $-Banach spaces $ ( V, \left\| \cdot \right\|_{V} ) $ and $ ( W, \left\| \cdot \right\|_{W} ) $ 
let $(\Lip ( V , W ) , \left\| \cdot\right\|_{ \Lip ( V , W ) })$
be the $\R$-Banach space of Lipschitz continuous mappings from $V$ to $W$, 
for every $\ell \in \N$ and every pair of $\R$-Banach spaces $ ( V , \left\| \cdot\right\|_{ V  } ) $ and 
$ ( W , \left\| \cdot\right\|_{ W } ) $ let $(C_{\mathrm{b}}^{\ell}( V,W ), \left\| \cdot \right\|_{C_{\mathrm{b}}^{\ell}( V,W )})$ 
be the $\R$-Banach space of $\ell$-times continuously 
Fr\'echet differentiable functions 
from $V$ to $W$ with globally bounded derivatives, 
for every $v\in H$ let 
$M_v\colon D(M_v)\subseteq H\rightarrow H$ be the linear operator which satisfies that
$D(M_v) = \{ h \in H \colon vh \in H\}$ 
and $[\forall\, h\in D(M_v) \colon M_v h = vh]$, 
and let $\eps\in(0,\nicefrac23]$.
Observe that the fact that 
for every $\rho\in[0,\nicefrac14)$ it holds that $B\in\Lip(H,L_2(H,H_{\rho-\nicefrac12}))$
implies 
that for every $\rho\in[0,\nicefrac14)$ it holds that $\bfB\in\Lip(\bfH_0,L_2(H,\bfH_{2\rho}))$. Hence, we obtain that  $\bfB\in\Lip(\bfH_0,L_2(H,\bfH_{\nicefrac12-\nicefrac\eps4}))$ and 
\begin{equation}\label{eq:bfB_Lip1} 
\bfB|_{\bfH_{\nicefrac{1}{2} - \nicefrac{\eps}{4}}} \in \Lip( \bfH_{\nicefrac{1}{2} - \nicefrac{\eps}{4}}, L_2(H,\bfH_{\nicefrac{1}{2} - \nicefrac{\eps}{4}})).
\end{equation}
Moreover, note that de 
Naurois et al.~\cite[(3.74)--(3.75) in Section 3.3]{JacobedeNauroisJentzenWelti:2021} 
and H\"{o}lder's inequality ensure that for every $\rho\in(0,\nicefrac14)$, $v, w\in H_\rho$, $u\in H_1$ it holds that $B(v),B(w)\in L(H,H_{\rho-\nicefrac14})$ and
\begin{equation}
\begin{split}
&\|(B(v)-B(w))u\|_{H_{\rho-\nicefrac14}}\\
&=
\sup_{\psi\in H_1\setminus\{0\}}\frac{|\langle\psi,(B(v)-B(w))u\rangle_H|}{\|\psi\|_{H_{(\nicefrac14)-\rho}}}=
\sup_{\psi\in H_1\setminus\{0\}}\frac{\|\psi\,b_1(v-w)u\|_{L^1(\lambda;\R)}}{\|\psi\|_{H_{(\nicefrac14)-\rho}}}\\
&\leq
\sup_{\psi\in H_1\setminus\{0\}}\frac{|b_1|\|\psi\|_{L^{\nicefrac{1}{(2\rho)}}(\lambda;\R)}\|v-w\|_{L^{\nicefrac{2}{(1-4\rho)}}(\lambda;\R)}\|u\|_{L^2(\lambda;\R)}}{\|\psi\|_{H_{(\nicefrac14)-\rho}}}\\
&\leq
|b_1|\Bigg[\sup_{\psi\in H_1\setminus\{0\}}\frac{\|\psi\|_{L^{\nicefrac{1}{(2\rho)}}(\lambda;\R)}}{\|\psi\|_{H_{(\nicefrac14)-\rho}}}\Bigg]
\Bigg[\sup_{\zeta\in H_\rho\setminus\{0\}}\frac{\|\zeta\|_{L^{\nicefrac{2}{(1-4\rho)}}(\lambda;\R)}}{\|\zeta\|_{H_{\rho}}}\Bigg]\|v-w\|_{H_\rho}\|u\|_H
<\infty.
\end{split}
\end{equation}
This 
assures that for every $\rho\in(0,\nicefrac14)$ it holds that $B|_{H_\rho}\in\Lip(H_\rho,L(H,H_{\rho-\nicefrac14}))$
and 
$\bfB|_{\bfH_{2\rho}}\in\Lip(\bfH_{2\rho},L(H,\bfH_{2\rho+\nicefrac12}))$.
Therefore, we obtain that
\begin{equation}\label{eq:bfB_Lip2} 
\bfB|_{\bfH_{\nicefrac{1}{2} - \nicefrac{\eps}{4}}} \in \Lip( \bfH_{\nicefrac{1}{2} - \nicefrac{\eps}{4}}, L(H,\bfH_{1 - \nicefrac{\eps}{4}})).
\end{equation}
Furthermore, observe that for every $(v_1,v_2),(w_1,w_2)\in\bfH_0$, $u\in H_1$ it holds that
\begin{equation}\label{eq:BC4_1}
\begin{split}
&\bfB(v_1+w_1,v_2+w_2)u
=(0,B(v_1+w_1)u)
=(0,(b_0+b_1(v_1+w_1))u)\\
&=(0,(b_0+b_1v_1)u)+(0,b_1w_1u)
=(0,B(v_1)u)+(0,b_1w_1u)\\
&=\bfB(v_1,v_2)u+(0,b_1w_1u).
\end{split}
\end{equation}
Combining this, the fact that for every $(v_1,v_2)\in\bfH_0$ it holds that $\bfB(v_1,v_2)\in L_2(H,\bfH_0)$, and the fact that $H_1$ is a dense subset of $H$ 
implies that for every $(v_1,v_2),(w_1,w_2)\in\bfH_0$, $u\in H_1$ it holds that $\bfB\in C^1(\bfH_0,L_2(H,\bfH_0))$ and
\begin{equation}
\big[\big(\bfB^{(1)}(v_1,v_2)\big)(w_1,w_2)\big]u=(0,b_1w_1u).
\end{equation}
Hence, we obtain that for every $k\in\N$ it holds that 
\begin{equation}\label{eq:bfB_C4}
\bfB \in C^k_{\mathrm{b}}( \bfH_0, L_2(H, \bfH_0 )).
\end{equation}
Next observe that for every $(v_1,v_2)\in \bfH_{0}$ it holds that
\begin{equation}\label{eq:proofcorweakrates_1}
\begin{aligned}
&
\sum_{n\in \N} (n\pi)^{1-\eps} \left\| \bfB(v_1,v_2)e_n \right\|_{\bfH_0}^2
= 
\sum_{n\in \N} (n\pi)^{1-\eps} 
    \,\big\|
		(-A)^{-\nicefrac12}((b_0+b_1v_1)e_n) 
	\big\|_{H}^2
\\ 
& = 
  \sum_{n\in \N} (n\pi)^{1-\eps} 
   \,\big\| 
  	b_0(-A)^{-\nicefrac12}e_n+b_1(-A)^{-\nicefrac12}M_{v_1}e_n 
   \big\|_{H}^2
\\ 
& =  
\sum_{n\in\N} 
	\big\|
    b_0 (-A)^{-\nicefrac{(1+\eps)}{4}}e_n
    +
    b_1 
    (- A)^{-\nicefrac12}  M_{v_1} (-A)^{\nicefrac{(1-\eps)}{4}}e_n
   \big\|_H^2
\\ 
&  \leq  
 2|b_0|^2
   \,\big\| 
    (-A)^{-\nicefrac{(1+\eps)}{4}}
  \big\|_{L_2(H)}^2
  +
  2|b_1|^2 
  \left[\sum_{n\in\N}
   \big\|
    (-A)^{-\nicefrac12} M_{v_1} (- A)^{\nicefrac{(1-\eps)}{4}}e_n
   \big\|_{H}^2\right].
\end{aligned}
\end{equation}
Moreover, note that for every $(v_1,v_2),(w_1,w_2)\in\bfH_0$ it holds that 
\begin{equation}\label{eq:proofcorweakrates_2}
\begin{aligned}
& 
\sum_{n\in \N} (n\pi)^{\eps-1} \left\| [\bfB(v_1,v_2) - \bfB(w_1,w_2)]e_n \right\|_{\bfH_0}^2
\\& 
=
\sum_{n\in\N}(n\pi)^{\eps-1}
	\,\big\|
		(-A)^{-\nicefrac12}(b_1(v_1-w_1)e_n)
	\big\|_H^2
\\
& = 
|b_1|^2
\left[\sum_{n\in\N}
	\big\|
      (-A)^{-\nicefrac12}  M_{(v_1-w_1)} (- A)^{\nicefrac{(\eps-1)}{4}}e_n
    \big\|_{H}^2
\right].
\end{aligned}
\end{equation}
In addition, observe that item~\eqref{item:inV2} in Lemma~\ref{lemma:multiplication_HS} ensures that for every $r\in(-\nicefrac14,\nicefrac14)$, $v\in H_{\max\{0,r\}}$, $n\in\N$ it holds that $M_v(-A)^{-\nicefrac12}e_n\in H_{\max\{0,r\}}$. 
This and the fact that for every $v\in H$ it holds that $M_{v}\colon D(M_v)\subseteq H\to H$ is a symmetric linear operator imply that for every $r\in(-\nicefrac14,\nicefrac14)$, $v\in H_{\max\{0,r\}}$ it holds that
\begin{equation}
\begin{aligned}
&\sum_{n\in\N}
   \big\|
    (-A)^{-\nicefrac12} M_{v} (- A)^{r}e_n
  \big\|_{H}^2
=
\sum_{m,n\in\N}
   \big|\big\langle
       (-A)^{-\nicefrac12} M_{v} (- A)^{r}e_n,e_m
   \big\rangle_H\big|^2
\\
&=
\sum_{m,n\in\N}
    \big|\big\langle 
         e_n,(-A)^{r} M_{v} (- A)^{-\nicefrac12}e_m
   \big\rangle_H\big|^2  
= 
\sum_{m\in\N}
	\big\|
      (-A)^{r} M_{v} (- A)^{-\nicefrac12}e_m
   \big\|_H^2.
\end{aligned}
\end{equation}
Lemma~\ref{lemma:multiplication_HS}, \eqref{eq:proofcorweakrates_1}, \eqref{eq:proofcorweakrates_2}, and the fact that 
for every $r \in (\nicefrac{1}{4},\infty)$ it holds that 
$
\| A^{-r} \|_{L_2(H)} 
< \infty
$
therefore ensure that
\begin{equation} 
\begin{aligned}
&
\sup
 \!\bigg\{
 \tfrac{
  \sum_{n\in \N} (n\pi)^{1-\eps} \left\| \bfB(v_1,v_2)e_n \right\|_{\bfH_0}^2
 }
 {
   \max\{1,\left\| (v_1,v_2) \right\|_{\bfH_{\nicefrac{(1-\eps)}{2}}}^2\}
 }
 \colon 
 \substack{(v_1,v_2)\in \bfH_{\nicefrac{(1-\eps)}{2}}}
  \bigg\}
\\ &  \leq 
\sup
  \!\bigg\{
  \tfrac{
  2|b_0|^2
  \| 
    (-A)^{-\nicefrac{(1+\eps)}{4}}
  \|_{L_2(H)} ^2 
  +
  2|b_1|^2 
  \| 
    (-A)^{\nicefrac{(1-\eps)}{4}} M_{v_1} (- A)^{-\nicefrac12} 
  \|_{L_2(H)}^2
 }
 {
   \max\{1,\left\| (v_1,v_2) \right\|_{\bfH_{\nicefrac{(1-\eps)}{2}}}^2\}
 } 
 \colon
 \substack{(v_1,v_2)\in \bfH_{\nicefrac{(1-\eps)}{2}}}
  \bigg\}
 \\
 &<\infty 
\end{aligned}
\end{equation}
and 
\begin{equation}
\begin{aligned}
& 
\sup
  \!\bigg\{
  \tfrac{
  \sum_{n\in \N} (n\pi)^{\eps-1} \left\| [\bfB(v_1,v_2) - \bfB(w_1,w_2)]e_n \right\|_{\bfH_0}^2
 }
 {
   \left\| (v_1,v_2) - (w_1,w_2) \right\|_{\bfH_{\nicefrac{(\eps-1)}{2}}}^2
 }
 \colon
 \substack{
 (v_1,v_2),(w_1,w_2)\in \bfH_{\nicefrac{(1-\eps)}{2}},\\ (v_1,v_2)\neq (w_1,w_2)
 }
  \bigg\}
\\ & = 
\sup
  \!\bigg\{
  \tfrac{
  |b_1|^2 
  \| 
    (-A)^{\nicefrac{(\eps-1)}{4}} M_{(v_1-w_1)} (- A)^{-\nicefrac12} 
  \|_{L_2(H)}^2
 }
 {
   \left\| (v_1,v_2)-(w_1,w_2) \right\|_{\bfH_{\nicefrac{(\eps-1)}{2}}}^2
 }
 \colon
 \substack{
 (v_1,v_2),(w_1,w_2)\in \bfH_{\nicefrac{(1-\eps)}{2}},\\ (v_1,v_2)\neq (w_1,w_2)
 }
  \bigg\}
 < \infty.
\end{aligned}
\end{equation}
Combining this, \eqref{eq:bfB_Lip1}, \eqref{eq:bfB_Lip2}, \eqref{eq:bfB_C4}, and Corollary~\ref{cor:weak_conv} 
(with 
$U = H$, 
$\U = \{ e_n \}_{n\in \N}$, 
$T=T$, $(W_t)_{t\in[0,T]}=(W_t)_{t\in[0,T]}$, 
$\gamma = 1 - \nicefrac{\eps}{4}$, 
$\beta = \nicefrac{1}{2} + \nicefrac{\eps}{4}$, 
$\rho = \nicefrac{1}{2} - \nicefrac{\eps}{4}$, 
$H = H$, 
$\H = \{ e_n \}_{n\in \N}$,
$\left[\forall\,n\in \N \colon \lambda_{e_n} = -\vartheta (n\pi)^2\right]$, 
$A=\vartheta A$, 
$\left[\forall\,r\in\R:\left\|\cdot\right\|_{H_r}=\vartheta^r\left\|\cdot\right\|_{H_r}\right]$, 
$\bfA=\bfA$, 
$\phi = \phi$, $\xi = \xi$, 
$F=0$, $B=\bfB$, 
$\left[\forall\, n\in \N \colon \mu_{e_n} = (n\pi)^{\nicefrac{(1-\eps)}{2}}\right]$, 
$\mathfrak{m} = \max\{\vartheta^{-\nicefrac12},1\}\|\bfB\|_{C^4_{\mathrm{b}}(\bfH_0,L_2(H,\bfH_0))}$\linebreak$+1$, 
$\mathfrak{c}^2 = 
\max\{\vartheta^{\nicefrac{(\eps-3)}{2}},\vartheta^{\nicefrac{(\eps-1)}2}\}
\,\sup\!\Big\{
\tfrac{
   \sum_{n\in \N} (n\pi)^{1-\eps}\left\| \bfB(v_1,v_2)e_n \right\|_{\bfH_0}^2
 }
 {
   \max\{1,\left\| (v_1,v_2) \right\|_{\bfH_{(1-\eps)/2}}^2\}
 }
 \colon 
 \substack{(v_1,v_2)\in \bfH_{\nicefrac{(1-\eps)}{2}}}
 \Big\}
 $, 
$ \mathfrak{l}^2= $\linebreak$
\max\{\vartheta^{\nicefrac{-(1+\eps)}2},\vartheta^{\nicefrac{(1-\eps)}2}\}
\,\sup\!\Big\{
 \tfrac{
  \sum_{n\in \N} (n\pi)^{\eps-1}\left\| [\bfB(v_1,v_2) - \bfB(w_1,w_2)]e_n \right\|_{\bfH_0}^2
 }
 {
  \left\| (v_1,v_2) - (w_1,w_2) \right\|_{\bfH_{(\eps-1)/2}}^2
 }
\colon 
\substack{(v_1,v_2),(w_1,w_2)\in \bfH_{\nicefrac{(1-\eps)}2},\\ (v_1,v_2)\neq (w_1,w_2)} 
\Big\}
 $,  
$\big[\forall\,h$\linebreak$\in (0,T]\colon Y^{h,\H} = Y^h\big]$, $Y^{0,\H} = X$ 
in the 
notation of Corollary~\ref{cor:weak_conv})  
establishes~\eqref{eq:weakrates}. The proof of Corollary~\ref{cor:hyperbolic_anderson} is thus completed.
\end{proof}

In Corollary~\ref{cor:hyperbolic_anderson} the multiplicative noise term appearing in the 
hyperbolic SPDE  
\eqref{eq:X_hyperbolic_Anderson} is specified via the mapping $B\colon H\mapsto L_2(H,H_{-\nicefrac12})$, which is defined 
as a combination of 
a multiplication operator and a Nemytskii operator induced by an affine transformation. 
We expect that the approach 
to weak error analysis 
presented in this article can be extended to 
hyperbolic SPDEs 
formulated in a more general Banach space 
framework  
involving Banach space-valued stochastic integrals, so that more general classes of nonlinear Nemytskii operators can potentially be 
handled 
as well; compare, e.g.,  
\cite[Corollary~1]{CoxEtAl:2016}, \cite{HefterEtAl:2016}.

\subsection{Numerical simulations}\label{ssec:simulations}

Here we illustrate Corollary~\ref{cor:hyperbolic_anderson} 
with some numerical experiments. 
To this end, assume the setting specified in Corollary~\ref{cor:hyperbolic_anderson}, 
assume that $T=2$, $\vartheta=1$, $b_0=1$, $b_1=1$, assume that for every $\omega\in\Omega$ it holds that $\xi(\omega)=(e_1,0)$, 
let 
$J,M\in\N$, $\mathbf N\in2\N$, 
$P\colon H\to H$ satisfy for every $v\in H$ that $P(v)=\sum_{j=1}^J\langle e_j,v\rangle_H e_j$, 
let $\mathfrak W^{N,m}_n\colon\Omega\to P(H)$, $N,n,m\in\N$, be independent Gaussian random variables which satisfy for every $N,n,m\in\N$, $v,w\in P(H)$ that $\E[\mathfrak W^{N,m}_n]=0$ and $\E[\langle v,\mathfrak W^{N,m}_n\rangle_H\langle w,\mathfrak W^{N,m}_n\rangle_H]=(\nicefrac TN)\langle v,w\rangle_H$, 
and for every $N,m\in\N$ let $\Y^{N,m}=(\Y^{N,m,1},\Y^{N,m,2})\colon\{0,1,\ldots,N\}\times\Omega\to P(H)\times P(H)$ be the stochastic process which satisfies for every 
$n\in\{1,2,\ldots,N\}$  that $\Y^{N,m}_0=\xi$ 
and
\begin{equation}\label{eq:simulation_scheme}
\begin{aligned}
\Y^{N,m}_n = 
e^{(\nicefrac{T}{N})\bfA} 
\Bigl( \Y^{N,m,\;\!1}_{n-1}  ,\;
\Y^{N,m,\;\!2}_{n-1} + P\bigl[B\bigl(\Y^{N,m,\;\!1}_{n-1}\bigr)\;\!\mathfrak W^{N,m}_n\;\!  \bigr] \Bigr).
\end{aligned}
\end{equation} 
For every $N\in\{1,2,\ldots,\mathbf N/2\}$ we employ the random variable 
\begin{equation}\label{eq:simulation_error}
\begin{aligned}
\Biggl| 
\biggl[\frac1M\sum_{m=1}^M \varphi(\Y^{\mathbf N,m}_{\mathbf N})\biggr] - 
\biggl[\frac1M\sum_{m=1}^M \varphi(\Y^{N,m}_{N})\biggr]
\Biggr|
\end{aligned}
\end{equation} 
as a Monte Carlo approximation of the weak error $|\E[ \phi(X_T) ] - \E[ \phi(Y_T^{\nicefrac TN})] |$; compare \eqref{eq:weakrates} in Corollary~\ref{cor:hyperbolic_anderson} above.  
Note that this Monte Carlo approximation involves a spatial spectral Galerkin approximation using $J$ basis functions. 
For the implementation of the scheme~\eqref{eq:simulation_scheme} we apply the explicit representation of the propagator $e^{(\nicefrac{T}{N})\bfA}$ 
provided 
by Lemma~\ref{lemma:semigroup}.  
Moreover, in order to 
approximately compute the multiplicative noise term $P[B(\Y^{N,m,\;\!1}_{n-1})\mathfrak W^{N,m}_n]$ in 
\eqref{eq:simulation_scheme}, 
we use the fact that for every $v,u\in C([0,1],\R)$ it holds that
\begin{equation}
\begin{aligned}
P\bigl[B\big([v]_{\lambda,\,\calB(\R)}\big) [u]_{\lambda,\,\calB(\R)}\bigr]
&= 
\sum_{j=1}^J\bigl\langle e_j, B\big([v]_{\lambda,\,\calB(\R)}\big) [u]_{\lambda,\,\calB(\R)} \bigr\rangle_H e_j \\
&= 
\sum_{j=1}^J \biggl[\textstyle \int\limits_0^1 \displaystyle \sqrt{2}\sin(j\pi x) v(x) u(x) \,dx\biggr] e_j
\end{aligned}
\end{equation} 
and we employ 
for every $j\in\{1,2,\ldots,J\}$, $v,u\in C([0,1],\R)$ 
the approximation 
\begin{equation}\label{eq:aliasing}
\begin{aligned}
\int_0^1 \sqrt{2}\sin(j\pi x) v(x) u(x) \,dx
~\approx~ 
\frac{\sum_{k=0}^J \sqrt{2}\sin(j\pi \frac{k}{J+1}) v(\frac{k}{J+1}) u(\frac{k}{J+1})}{J+1}.
\end{aligned}
\end{equation}
Observe that for every $v,u\in C([0,1],\R)$ with $[v]_{\lambda,\,\calB(\R)}, [u]_{\lambda,\,\calB(\R)}\in P(H)$ the approximations   \eqref{eq:aliasing}, $j\in\{1,2,\ldots,J\}$, are readily implemented by means of the discrete sine transform; see the \textsc{Python} code in Listing~\ref{listing}. 

In Figure~\ref{figure} we present 
approximate simulations of the weak error estimator \eqref{eq:simulation_error} for different test functions $\varphi$ 
plotted against the used numbers of time steps $N\in\{2^3,2^4,\ldots,2^{11}\}$,
employing  $\mathbf N=2^{12}$ time steps for the reference value, $J=16$ spatial basis functions, and $M=5\cdot 10^5$ Monte Carlo runs. 
By using a 
relatively 
large number of Monte Carlo runs we take into account the fact that the Monte Carlo error tends to dominate 
the weak error caused be the temporal and spatial approximations. 
As the focus lies on temporal approximations, we content ourselves with a 
relatively 
small number of spatial basis functions.
Note that the simulation results presented in Figure~\ref{figure} are in accordance with Corollary~\ref{cor:hyperbolic_anderson}. 
In particular, the results in the case $\varphi(x)=\langle x^{(2)}, \pi e_1\rangle_{H_{-\nicefrac12}}$, $x=(x^{(1)},x^{(2)})\in\bfH_0$, suggest that the weak order  $1^-$ established in Corollary~\ref{cor:hyperbolic_anderson} is sharp. 
Moreover, the visible fluctuations of simulation values seem to be due to a dominance of the Monte Carlo error in the corresponding ranges.    
The \textsc{Python} code used to obtain the simulations is presented in Listing~\ref{listing}.
\wen

\begin{figure}[!ht]
\centering
{\includegraphics[width=.8\textwidth]{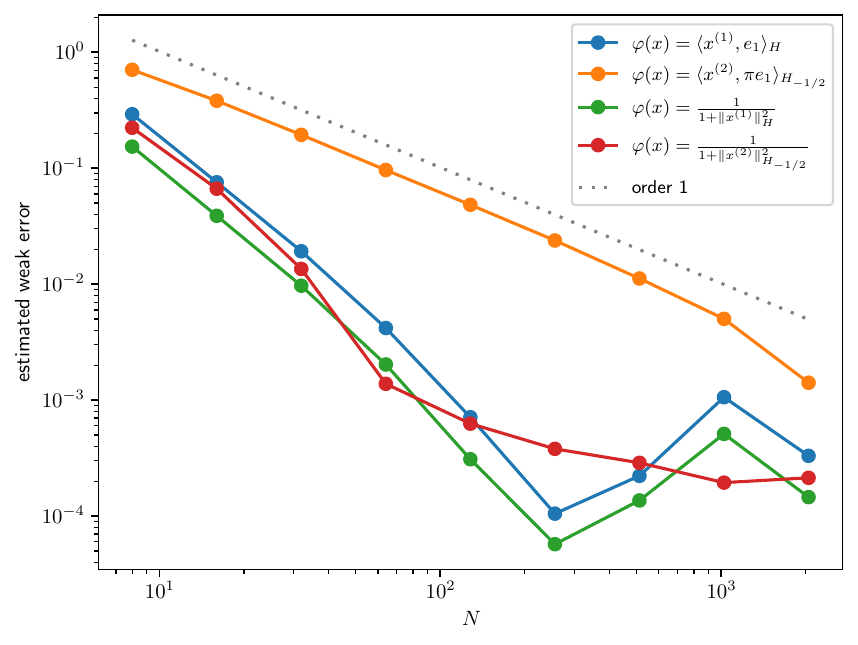}}
\caption{
Approximate simulations of the weak error estimator \eqref{eq:simulation_error} for different test functions $\varphi(x)$, $x=(x^{(1)},x^{(2)})\in\mathbf H_0$,  
plotted against the used numbers of time steps $N\in\{2^3,2^4,\ldots,2^{11}\}$,
employing  $\mathbf N=2^{12}$ time steps for the reference value, $J=16$ spatial basis functions, and $M=5\cdot 10^5$ Monte Carlo runs 
(see the text and Listing~\ref{listing} for details) 
\label{figure}
}
\end{figure}

\begin{lstlisting}[language=Python, caption=\textsc{Python} code used to create Figure~\ref{figure}, label=listing,
			captionpos=t
			]
import numpy as np
from scipy.fft import dst
import matplotlib.pyplot as plt  

# endpoint of time interval and numbers of time steps 
T= 2.; N_list = [2 ** n for n in range(3, 13)]

# number of spatial basis functions, number of Monte Carlo runs, and seed 
J = 16; M = 5e5; np.random.seed(123)  

# initial condition, given in terms of coefficients w.r.t. basis functions
xi = np.array([1] + (2 * J - 1) * [0], dtype = float).reshape(2, J)

# arrays of eigenvalues of (-A)^{1/2} and (-A)^{-1/2}
sqrt_ev = np.array([j * np.pi for j in range(1, J+1)])   
sqrt_ev_inv = sqrt_ev ** (-1)

# one-step propagation depending on step size h and array of coefficients 
# w.r.t. spatial basis functions of shape (2, J)
def propagate(h, coeff):     
    C = np.cos(h * sqrt_ev); S = np.sin(h * sqrt_ev)
    new_coeff = np.zeros([2, J])
    new_coeff[0,:] = C * coeff[0,:] + sqrt_ev_inv * S * coeff[1,:]
    new_coeff[1,:] = - sqrt_ev * S * coeff[0,:] + C * coeff[1,:]
    return new_coeff

# transformation of coefficients w.r.t. basis functions into function values
def coeff_to_fn(coeff):
    return 1 / np.sqrt(2) * dst(coeff, type = 1, norm = 'backward')

# transformation of function values into coefficients w.r.t. basis functions
def fn_to_coeff(values):
    return 1 / (np.sqrt(2) * (len(values) + 1)) \
            * dst(values, type = 1, norm = 'backward')

# simulation of approximate solution of the stochastic wave equation
# with N time steps and J spatial basis functions
def simulate_SWE(N):
    Y = np.zeros([N, 2, J])
    Y[0,:,:] = xi
    h = T / N
    for n in range(1, N):
        coeff = Y[n-1,:,:]
        y_1 = coeff_to_fn(coeff[0,:])
        dw = coeff_to_fn(np.sqrt(h) * np.random.normal(0, 1, size = J))
        BdW = fn_to_coeff(y_1 * dw)  
        coeff[1,:] += BdW
        Y[n,:,:] = propagate(h, coeff)
    return Y

# Monte Carlo approximation of E\varphi(Y^h_T) with step size h = T / N 
# using M samples and J spatial basis functions
def Monte_Carlo_EphiY(M, N):
    EphiY = np.zeros(4)   
    for m in range(M): 
        Y = simulate_SWE(N)                   
        Y_T = Y[-1,:,:]
        normsq = np.array([np.sum(Y_T[0,:] ** 2), 
                           np.sum((sqrt_ev_inv * Y_T[1,:]) ** 2)])
        EphiY += np.array([Y_T[0,0], Y_T[1,0], (1 + normsq[0]) ** (-1), 
                           (1 + normsq[1]) ** (-1)])
    EphiY *= 1/M * np.array([1, 1/np.pi, 1, 1])    
    return EphiY, Y  

# compute weak error estimates for different numbers of time steps
EphiY = np.zeros([len(N_list), 4])
error = np.zeros([len(N_list)-1, 4])
for i in range(len(N_list)):
    EphiY[i,:], Y = Monte_Carlo_EphiY(M, N_list[i])
for i in range(len(N_list)-1):
    error[i,:] = np.abs(EphiY[-1,:] - EphiY[i,:])      

## plot weak error estimates
plt.rcParams['text.usetex'] = True; plt.figure()
plt.xlabel('$N$'); plt.ylabel('estimated weak error')
plt.yscale('log'); plt.xscale('log')
order_line = 1.8 * error[0,1] * np.array([N_list[0] / N_list[i] 
                                          for i in range(0,len(N_list)-1)])
labels = [r'$\varphi(x) = \langle x^{(1)}, e_1 \rangle_H$', 
          r'$\varphi(x) = \langle x^{(2)}, \pi e_1 \rangle_{H_{-1/2}}$', 
          r'$\varphi(x) = \frac{1}{1+\|x^{(1)}\|_H^2}$', 
          r'$\varphi(x) = \frac{1}{1+\|x^{(2)}\|_{H_{-1/2}}^2}$'] 
for i in range(4):
    plt.plot(N_list[:-1], error[:,i], '-o', label = labels[i])
plt.plot(N_list[:-1], order_line, 
         linestyle = (0,(1,3)), color = '0.5', label = 'order 1')
plt.legend(prop={'size': 9})
plt.savefig('plot.pdf', bbox_inches = 'tight')
\end{lstlisting}

\section*{Acknowledgements}
This work has been partially supported by the NWO-research program VENI Vernieuwingsimpuls with the project number 639.031.549. This work has also been partially supported by the Internal Project Fund from Shenzhen Research Institute of Big Data under Grant T00120220001. The second author also gratefully acknowledges the Cluster of Excellence EXC 2044-390685587, Mathematics Münster: Dynamics-Geometry-Structure funded by the Deutsche Forschungsgemeinschaft (DFG, German Research Foundation).
\footnotesize
\bibliographystyle{alpha}
\bibliography{weakconv_wave}
%
\end{document}